\documentclass[a4paper, 10pt, reqno]{amsart}

\usepackage[utf8]{inputenc}
\usepackage[T1]{fontenc}
\usepackage{lmodern}
\usepackage{microtype}
\usepackage{amsmath}
\usepackage{amsthm}

\usepackage{amssymb}
\usepackage{amsfonts}
\usepackage{thmtools}
\usepackage{relsize}
\usepackage[mathscr]{eucal}
\usepackage[linktocpage]{hyperref}
\usepackage[sort,capitalize]{cleveref}
\usepackage{enumitem}			
\usepackage[dvipsnames]{xcolor}
\usepackage[msc-links, abbrev, non-sorted-cites]{amsrefs}
\usepackage{todonotes}
\usepackage{caption}

\SetMathAlphabet{\mathsf}{normal}{OT1}{lmss}{m}{n}
\SetMathAlphabet{\mathsf}{bold}{OT1}{lmss}{bx}{n}

\makeatletter
\def\nonumberfootnote{\xdef\@thefnmark{}\@footnotetext}			
\makeatother

\definecolor{colorred}{HTML}{B00000}
\definecolor{colorgreen}{HTML}{258300}
\definecolor{colorblue}{HTML}{2e32fa}
\definecolor{coloryellow}{HTML}{cbbb1a}
\hypersetup{colorlinks=true, linkcolor=colorred, citecolor=colorgreen, urlcolor=coloryellow}	

\numberwithin{equation}{section}

\newcommand{\rmc}{{\ensuremath{\mathrm{c}}}}
\newcommand{\rmd}{{\ensuremath{\mathrm{d}}}}
\newcommand{\rme}{{\ensuremath{\mathrm{e}}}}

\newcommand{\rmo}{{\ensuremath{\mathrm{o}}}}

\newcommand{\rmB}{{\ensuremath{\mathrm{B}}}}

\newcommand{\rmD}{{\ensuremath{\mathrm{D}}}}

\newcommand{\sfd}{{\ensuremath{\mathsf{d}}}}
\newcommand{\sfe}{{\ensuremath{\mathsf{e}}}}

\newcommand{\scrD}{{\ensuremath{\mathscr{D}}}}

\newcommand{\scrF}{{\ensuremath{\mathscr{F}}}}
\newcommand{\scrG}{{\ensuremath{\mathscr{G}}}}
\newcommand{\scrH}{{\ensuremath{\mathscr{H}}}}
\newcommand{\scrI}{{\ensuremath{\mathscr{I}}}}
\newcommand{\scrJ}{{\ensuremath{\mathscr{J}}}}

\newcommand{\scrS}{{\ensuremath{\mathscr{S}}}}

\newcommand{\bdpi}{{\ensuremath{\boldsymbol{\pi}}}}

\newcommand{\N}{\boldsymbol{\mathrm{N}}}						
\newcommand{\Q}{\boldsymbol{\mathrm{Q}}}						
\newcommand{\R}{\boldsymbol{\mathrm{R}}}						
\renewcommand{\d}{\,\mathrm{d}}				

\let\limsup\undefined
\let\liminf\undefined

\DeclareMathOperator*{\limsup}{limsup}		
\DeclareMathOperator*{\liminf}{liminf}		
\DeclareMathOperator{\supp}{spt}			

\theoremstyle{definition}
\newtheorem{bump}{Bump}[section]

\theoremstyle{plain}
\newtheorem{theorem}[bump]{Theorem}
\newtheorem{proposition}[bump]{Proposition}
\newtheorem{definition}[bump]{Definition}
\newtheorem{lemma}[bump]{Lemma}
\newtheorem{corollary}[bump]{Corollary}

\theoremstyle{remark}
\newtheorem{remark}[bump]{Remark}
\newtheorem{example}[bump]{Example}

\newtheorem{convention}[bump]{Convention}

\crefname{theorem}{Theorem}{Theorems}
\crefname{proposition}{Proposition}{Propositions}
\crefname{definition}{Definition}{Definitions}
\crefname{lemma}{Lemma}{Lemmas}
\crefname{corollary}{Corollary}{Corollaries}
\crefname{hypothesis}{Hypothesis}{Hypotheses}
\crefname{remark}{Remark}{Remarks}
\crefname{example}{Example}{Examples}
\crefname{notation}{Notation}{Notations}
\crefname{convention}{Convention}{Conventions}

\let\endconvention\endconvention

\renewenvironment{example}
  {\begin{oldexample}}
  {\hfill \scalebox{1}{$\blacksquare$}\end{oldexample}}

\renewenvironment{remark}
  {\begin{oldremark}}
  {\hfill \scalebox{1}{$\blacksquare$}\end{oldremark}}

  \renewenvironment{convention}
  {\begin{oldconvention}}
  {\hfill \scalebox{1}{$\blacksquare$}\end{oldconvention}}

\crefformat{section}{{§}#2#1#3}
\crefformat{subsection}{{§}#2#1#3}
\crefformat{subsubsection}{{§}#2#1#3}
\crefformat{appendix}{{§}#2#1#3}

\crefmultiformat{theorem}{Theorems #2#1#3}{ and #2#1#3}{, #2#1#3}{, and #2#1#3}
\crefmultiformat{proposition}{Propositions #2#1#3}{ and #2#1#3}{, #2#1#3}{, and #2#1#3}
\crefmultiformat{definition}{Definitions #2#1#3}{ and #2#1#3}{, #2#1#3}{, and #2#1#3}
\crefmultiformat{lemma}{Lemmas #2#1#3}{ and #2#1#3}{, #2#1#3}{, and #2#1#3}
\crefmultiformat{corollary}{Corollaries #2#1#3}{ and #2#1#3}{, #2#1#3}{, and #2#1#3}
\crefmultiformat{hypothesis}{Hypotheses #2#1#3}{ and #2#1#3}{, #2#1#3}{, and #2#1#3}
\crefmultiformat{remark}{Remarks #2#1#3}{ and #2#1#3}{, #2#1#3}{, and #2#1#3}
\crefmultiformat{example}{Examples #2#1#3}{ and #2#1#3}{, #2#1#3}{, and #2#1#3}
\crefmultiformat{notation}{Notations #2#1#3}{ and #2#1#3}{, #2#1#3}{, and #2#1#3}

\crefmultiformat{section}{{§§}#2#1#3}{ and #2#1#3}{, #2#1#3}{, and #2#1#3}
\crefmultiformat{subsection}{{§§}#2#1#3}{ and #2#1#3}{, #2#1#3}{, and #2#1#3}
\crefmultiformat{subsubsection}{{§§}#2#1#3}{ and #2#1#3}{, #2#1#3}{, and #2#1#3}

\crefrangeformat{equation}{#3\textcolor{black}{(}#1\textcolor{black}{)}#4 to #5\textcolor{black}{(}#2\textcolor{black}{)}#6}

\newcommand{\mms}{M}				  
\newcommand{\met}{\sfd}						

\newcommand{\meas}{\mathfrak{m}}				

\newcommand{\Leb}{\mathscr{L}}				
\newcommand{\vol}{\mathrm{vol}}				
\newcommand{\Prob}{\mathscr{P}}		        

\newcommand{\Id}{\mathrm{Id}}				
\newcommand{\ac}{{\mathrm{ac}}}

\newcommand{\TCD}{\mathsf{TCD}}

\newcommand{\TMCP}{\mathsf{TMCP}}

\newcommand{\CD}{\mathsf{CD}}

\newcommand{\comp}{\rmc}						

\newcommand{\loc}{\mathrm{loc}}				
\newcommand{\Ric}{\mathrm{Ric}}				
\newcommand{\Cont}{C}					

\newcommand{\Lip}{\mathrm{Lip}}				

\newcommand{\Dom}{\scrD}					

\DeclareMathOperator{\Hess}{Hess}			
\DeclareMathOperator{\diam}{diam}			
\newcommand{\eval}{\sfe}					

\newcommand{\push}{\sharp}					

\newcommand{\cl}{\mathrm{cl}}				
\newcommand{\Len}{\mathrm{Len}}
\newcommand{\TGeo}{\mathrm{TGeo}}

\usepackage{blindtext}

\allowdisplaybreaks

\let\oldtocsection=\tocsection

\let\oldtocsubsection=\tocsubsection

\let\oldtocsubsubsection=\tocsubsubsection

\renewcommand{\tocsection}[2]{\hspace{0em}\oldtocsection{#1}{#2}}
\renewcommand{\tocsubsection}[2]{\hspace{1em}\oldtocsubsection{#1}{#2}}
\renewcommand{\tocsubsubsection}[2]{\hspace{2em}\oldtocsubsubsection{#1}{#2}}

\newcommand{\nocontentsline}[3]{}
\newcommand{\tocless}[2]{\bgroup\let\addcontentsline=\nocontentsline#1{#2}\egroup}

\newcommand{\mres}{\mathbin{\vrule height 1.6ex depth 0pt width 0.13ex\vrule height 0.13ex depth 0pt width 1.3ex}}

\renewcommand{\q}{\mathfrak{q}}

\newcommand{\sing}{{\mathrm{sing}}}

\newcommand{\TCut}{\mathrm{TC}}

\allowdisplaybreaks

\makeatletter
\@namedef{subjclassname@2020}{\textup{2020} Mathematics Subject Classification}
\makeatother

\newcommand{\MB}[1]{\textcolor{cyan}{#1}}
\newcommand{\ms}[1]{\textcolor{magenta}{#1}}

\newcommand{\ICD}{\mathrm{def}}
\newcommand{\tang}{{\mathrm{tan}}}
\newcommand{\neas}{\mathfrak{n}}
\newcommand{\Vol}{\mathrm{Vol}}

\begin{document}

\title[Comparison theory for Lipschitz spacetimes]{Comparison theory for Lipschitz spacetimes}
\author{Mathias Braun}
\address{Institute of Mathematics, EPFL, 1015 Lausanne, Switzerland}
\email{\href{mailto:mathias.braun@epfl.ch}{mathias.braun@epfl.ch}}

\author{Marta Sálamo Candal}

\address{Fakultät für Mathematik, University of Vienna, 1090 Vienna, Austria}
\email{\href{mailto:marta.salamo.candal@univie.ac.at}{marta.salamo.candal@univie.ac.at}}

\thanks{MB is supported by the EPFL through a Bernoulli Instructorship. MSC was funded in part by the Austrian Science Fund (FWF) [Grant DOI \href{https://www.fwf.ac.at/en/research-radar/10.55776/EFP6}{10.55776/EFP6}] and the Vienna School of Mathematics. For open
access purposes, the authors have applied a CC BY public copyright license to any author
accepted manuscript version arising from this submission.}

\begin{abstract} We prove a globally hyperbolic spacetime with locally Lipschitz continuous metric and timelike distributional Ricci curvature bounded from below obeys the timelike measure contraction property. The remarkable class of examples of  spacetimes that are covered by this result includes impulsive gravity waves, thin shells, and matched spacetimes.

As applications, we get new comparison theorems for Lipschitz spacetimes in sharp form: d'Alembert, timelike Brunn--Minkowski, and timelike Bishop--Gromov. Under appropriate nonbranching assumptions (conjectured to hold in even lower regularity), our results also yield the timelike curvature-dimension condition, a volume incompleteness theorem, as well as  exact representation formulas and sharp comparison estimates for d'Alembertians of Lorentz distance functions from general spacelike submanifolds.

Moreover, we establish the sharp timelike Bonnet--Myers inequality \emph{ad hoc} using the localization technique from convex geometry. Alongside, we prove a timelike diameter estimate for spacetimes whose timelike Ricci curvature is positive up to a ``small'' deviation (in an $L^p$-sense). This adapts prior theorems for Riemannian manifolds by Petersen--Sprouse and Aubry to Lorentzian geometry, a transition the former two anticipated almost 30 years ago.
\end{abstract}

\maketitle



\thispagestyle{empty}

\tableofcontents

\addtocontents{toc}{\protect\setcounter{tocdepth}{2}}

\section{Introduction}

The goal of this work is twofold. First, we establish several sharp comparison theorems for globally hyperbolic spacetimes with a locally Lipschitz continuous metric, which we refer to as Lipschitz spacetimes. Second, we prove that analytic timelike Ricci curvature bounds for such spacetimes, formulated through the theory of tensor distributions, imply synthetic timelike Ricci curvature bounds in terms of optimal transport and metric geometry. The former results are obtained as applications of the latter. A key ingredient of our approach is a new link between low regularity spacetime geometry and the well-developed theory of Riemannian manifolds with integral curvature bounds.

\subsection{Background}\label{Sub:Bakground} There is a certain  agreement to call a spacetime $(\mms,g)$ ``nonsmooth'' if its metric tensor $g$ fails to have regularity $C^2$ (in which key concepts like geodesics and curvature would still be classically defined). As pointed out already by Lichnerowicz \cite{lichnerowicz1955-theorie-relativiste} in the 1950s, nonsmooth spacetimes arise in many natural ways: one source of motivation comes from the initial value problem for the Einstein equations, cf.~Rendall \cite{rendall2002}. Moreover, Hawking--Ellis \cite{hawking-ellis1973} stated in the 1970s that the persistence of the classical incompleteness theorems below regularity $C^2$ would rule out the undesired possibility that observers cease to exist merely due to a loss of regularity, thereby underlining the physical significance of ``singularities''. Progress on the incompleteness  aspect, however, has only been made rather recently, cf.~Calisti et al.~\cite{calisti-graf-hafemann-kunzinger-steinbauer2025+} and the references therein.

The framework of our contribution are Lipschitz spacetimes. There, the geodesic equation was studied by Sämann--Steinbauer \cite{samann-steinbauer2018}, Graf--Ling \cite{graf-ling2018}, and Lange--Lytchak--Sämann \cite{lange-lytchak-samann2021}. On the other hand, only recently Calisti et al.~\cite{calisti-graf-hafemann-kunzinger-steinbauer2025+} proved the first result involving curvature hypotheses on Lipschitz spacetimes, namely Hawking's in\-completeness theorem. Curved Lipschitz spacetimes arise in numerous prominent examples, such as impulsive gravity waves (cf.~e.g.~Choquet-Bruhat \cite{choquet1968espaces} and Penrose  \cite{penrose1972}) and thin shells and matched spacetimes (cf.~e.g.~Israel \cite{israel1966singular},  Khakshournia--Mansouri \cite{khakshournia2023art}, and Manzano--Mars \cite{manzano2024abstract}); in a distributional sense, the former constitute solutions to the vacuum Einstein equations $\smash{\Ric_g =0}$, while the latter often exhibit timelike nonnegative Ricci curvature $\smash{\Ric_g\geq 0}$ depending on the matter of the junction hypersurface. In particular, all these examples will be covered by our results. Moreover, Lipschitz spacetimes are the sharp threshold of regularity that avoids bubbling issues, cf.~Chrus\'{c}iel--Grant \cite{chrusciel-grant2012}. They are closed cone structures in the sense of Minguzzi \cite{minguzzi2019-applications} and regularly localizable Lorentzian length spaces in the sense of Kunzinger--Sämann \cite{kunzinger-samann2018}.

The importance of a systematic comparison theory in Lorentzian geometry, guided by its influential Riemannian precedent (cf.~e.g.~Cheeger--Ebin \cite{cheeger-ebin1975}), was emphasized by Ehrlich \cite{ehrlich2005}. It is naturally linked to  timelike Ricci curvature bounds, which correspond to energy conditions after Penrose \cite{penrose1965} and Hawking \cite{hawking1967}. For instance, it is a cornerstone for the Lorentzian splitting theorems of Eschenburg \cite{eschenburg1988}, Galloway \cite{galloway1989-splitting}, and Newman \cite{newman1990}, which are  understood as rigidity statements for the vacuum Einstein equations and ultimately led to the open  splitting conjecture of Bartnik \cite{bartnik1988}. However, besides the Hawking incompleteness theorem of Calisti et al.~\cite{calisti-graf-hafemann-kunzinger-steinbauer2025+}, comparison theorems have remained unknown for Lipschitz spacetimes. Central issues are the lack of an exponential map and the mere availability of distributional curvature tensors, which  make it difficult to carry out the usual Jacobi field computations. Moreover, as pointed out by Kunzinger--Oberguggenberger--Vickers \cite{kunzinger-oberguggenberger-vickers2024}, Braun--Calisti \cite{braun-calisti2023}, and Calisti et al.~\cite{calisti-graf-hafemann-kunzinger-steinbauer2025+}, a related open problem is the compatibility of distributional timelike Ricci curvature bounds for nonsmooth spacetimes with the synthetic ones introduced by Cavalletti--Mondino \cite{cavalletti-mondino2020}.

These motivations naturally lead to the main contributions of this paper we outlined above and detail below, together with relevant literature. 

\subsection{Sharp timelike Bonnet--Myers inequality} The Bonnet--Myers theorem is a classical result in Riemannian geometry. It states a positive lower bound on the Ricci curvature implies an upper bound on the diameter of the Riemannian manifold in question.

Its analog in Lorentzian geometry is interpreted as an incompleteness theorem: under lower boundedness of the Ricci curvature in timelike directions by a positive constant, it implies a sharp upper bound on the timelike diameter of the considered spacetime $(\mms,g)$. In the globally hyperbolic case, the classical smooth version found e.g.~in Beem--Ehrlich--Easley \cite{beem-ehrlich-easley1996} was generalized by Graf \cite{graf2016} when $g$ is $C^{1,1}$ and Braun--Calisti \cite{braun-calisti2023} if $g$ is $C^1$. We generalize their results as follows.

\begin{theorem}[Timelike Bonnet--Myers diameter estimate, \cref{Th:TimelikeBonnetMyers}]\label{Th: 1.1} We fix $K> 0$. In addition, assume $(\mms,g)$ is a globally hyperbolic Lipschitz spacetime such that $\smash{\Ric_{g}\geq K}$ timelike distributionally, cf.~\cref{Def:Bounds}. Then
\begin{align*}
    \diam\,(\mms,g)\leq\pi\sqrt{\frac{\dim\mms-1}{K}}.
\end{align*}
\end{theorem}

As for all of our results that involve a curvature hypothesis, a weighted version of \cref{Th: 1.1} holds as well.

Although not stated explicitly therein, \cref{Th: 1.1} can be deduced from the Hawking incompleteness theorem of Calisti et al.~\cite{calisti-graf-hafemann-kunzinger-steinbauer2025+} combined with the choice of a suitable hypersurface by an argument of Graf \cite{graf2016}*{Rem.~4.3}. 

We give a different proof which adds new insights, cf.~\cref{Th:TLPS} below. An ingredient we frequently use is Calisti et al.'s setup of an approximation of $g$ by smooth metrics $(g_n)_{n\in\N}$  such that the ``excess'' of the $g_n$-timelike Ricci curvature compared to $K$ tends to zero locally in $L^p$ as $n\to\infty$ for every $p\in[1,\infty)$ \cite{calisti-graf-hafemann-kunzinger-steinbauer2025+}, as recalled in \cref{Le:ApproxLemma,Le:Lpcvg}; we call such a sequence ``good approximation'' of $g$, cf.~\cref{Def:Goodapprox}. We combine this with the localization technique, cf.~\cref{Sub:Localization}, pioneered in the Lorentzian case by Cavalletti--Mondino \cite{cavalletti-mondino2020} and developed further by themselves \cite{cavalletti-mondino2024} and Braun--McCann \cite{braun-mccann2023}, to reduce \cref{Th: 1.1} to a  family of one-dimensional problems along a geodesic foliation. Using this synthetic tool, we are able to invoke the well-established theory of Riemannian manifolds with Ricci curvature in $L^p$ in the semiclassical range $p\in (\dim\mms/2,\infty)$ for  Schrödinger operators, whose comparison theory was initiated by Petersen--Wei \cite{petersen-wei1997}. Bonnet--Myers-type theorems in this setting were shown  by  Petersen--Sprouse \cite{petersen-sprouse1998}, Sprouse \cite{sprouse2000integral}, and Aubry \cite{aubry2007}; later extensions to Kato-type bounds were given by Carron \cite{carron2019geometric}, Rose \cite{rose2021}, and Carron--Rose \cite{carron2021geometric}. More recently, comparison theory for metric measure spaces with integral curvature bounds in the sense of Ketterer \cite{ketterer2017} was developed by Caputo--Nobili--Rossi \cite{caputo-nobili-rossi2025+}. Our new link of the preceding Riemannian theories to Lorentzian geometry was anticipated by Braun--McCann \cite{braun-mccann2023}. 

In fact, a byproduct of our proof of \cref{Th: 1.1} is the subsequent Lorentzian counterpart of Petersen--Sprouse's qualitative diameter estimate proved 30 years ago for Riemannian manifolds \cite{petersen-sprouse1998}. They also anticipated the connection of their result with general relativity (where it is reasonable to expect incompleteness under small quantum fluctuations, which can be modeled by a small curvature excess in an integral sense, as expressed in \cite{petersen-sprouse1998}), which the following theorem settles.

\begin{theorem}[Timelike Petersen--Sprouse diameter estimate, \cref{Cor:PetersenSprouse}]\label{Th:TLPS} Let  $K> 0$, $N\in [2,\infty)$, and $p\in (N/2,\infty)$. Then for every $\varepsilon>0$, there is a constant $\xi_{K,N,p,\varepsilon}>0$ with the following property. Let $(\mms,g,\meas)$ form a globally hyperbolic Lipschitz spacetime with $\dim\mms\leq N$ and let $k\colon \mms\to\R$ be a continuous function with $\smash{\Ric_g \geq k}$ in all timelike directions such that for every nonempty, open, and precompact  $W\subset\mms$,
\begin{align*}
    \frac{1}{\vol_g[W]}\int_W\big\vert(k-K)_-\big\vert^p\d\vol_g\leq \xi_{K,N,p,\varepsilon}.
\end{align*}
Then we have
\begin{align*}
    \diam\,(\mms,g)\leq\pi\sqrt{\frac{N-1}{K}} +\varepsilon.
\end{align*}
\end{theorem}

Lorentzian counterparts of the quantitative Riemannian diameter estimates of Aubry \cite{aubry2007} are  established as well, cf.~\cref{Th:AUBRY}.

Despite the hypothesized smoothness of the metric tensor  (which we believe can be relaxed), it is interesting to read \cref{Th:TLPS} in the context of current quests for precompactness in some Lorentz--Gromov--Hausdorff topology (cf.~e.g.~the review of Cavalletti--Mondino \cite{cavalletti-mondino2022-review}), as it implies a uniform and almost sharp  diameter bound across a family of spacetimes with uniformly small curvature excess.

\subsection{Analytic and synthetic spacetime geometry} The energy condition ``$\Ric_g\geq K$ in all timelike directions'' of Penrose \cite{penrose1965} and Hawking \cite{hawking1967}, where $K\in\R$, provides a natural connection of general relativity and Lorentzian comparison theory. We now outline two ways to define it on a Lipschitz spacetime $(\mms,g)$. To simplify the presentation, in the remainder we only state our main results (and adjacent definitions) in the case when $K$ is zero.

The usual \emph{analytic} way to define Ricci curvature (and its timelike lower boundedness) on Lipschitz spacetimes --- in fact, Geroch--Traschen regularity $\smash{H_\loc^{1,2}\cap C^0}$ suffices \cite{geroch-traschen1987} --- is the theory of tensor distributions. We refer to Grosser--Kunzinger--Oberguggenberger--Steinbauer \cite{grosser-kunzinger-oberguggenberger-steinbauer2001}, Graf \cite{graf2020}, and \cref{Sub:Approx} for details. Throughout our work, the reader should think of $(\mms,g)$ as a distributional solution to the Einstein equations with timelike energy-momentum tensor bounded from below, such as those mentioned in \cref{Sub:Bakground}.

On the other hand, current research devotes significant attention to \emph{synthetic} formulations of timelike Ricci curvature lower bounds, introduced by Cavalletti--Mondino \cite{cavalletti-mondino2020} inspired by preceding breakthroughs by McCann \cite{mccann2020} and Mondino--Suhr \cite{mondino-suhr2022}. It is foreshadowed by the well-developed theory of so-called CD metric measure spaces by Sturm \cite{sturm2006-i,sturm2006-ii} and Lott--Villani \cite{lott-villani2009}. In a nutshell, the synthetic theory from \cite{cavalletti-mondino2020} stipulates  convexity properties of an entropy functional along geodesics in the Lorentz--Wasserstein space introduced by Eckstein--Miller \cite{eckstein-miller2017}. An  advantage of this approach is that it makes sense on ``metric measure spacetimes'' without a notion of Ricci curvature or even a differentiable structure, such as  Busemann's timelike spaces \cite{busemann1967}, Kunzinger--Sämann's Lorentzian length spaces \cite{kunzinger-samann2018}, or Minguzzi--Suhr's bounded Lorentzian metric spaces \cite{minguzzi-suhr2022}.  Among several advantages pointed out e.g.~by Cavalletti--Mondino's review \cite{cavalletti-mondino2022-review}, let us note for now that several comparison theorems hold in this abstract setting, cf.~Cavalletti--Mondino \cite{cavalletti-mondino2020,cavalletti-mondino2024}, Beran et al.~\cite{beran-braun-calisti-gigli-mccann-ohanyan-rott-samann+-},  Braun \cite{braun2024+}, and \cref{Sub:APPLCOMP} below.

It is therefore natural to ask how analytic and synthetic timelike Ricci curvature bounds are related. This question was inspired by pioneering works of McCann \cite{mccann2020} and Mondino--Suhr \cite{mondino-suhr2022}, proving the equivalence of both approaches for smooth $g$. If $g$ is $C^1$, the implication from analytic to synthetic theory was accomplished by Braun--Calisti \cite{braun-calisti2023}. In Riemannian signature, an analogous implication was shown by Kunzinger--Oberguggenberger--Vickers \cite{kunzinger-oberguggenberger-vickers2024} in regularity $C^1$ and later Mondino--Ryborz \cite{mondino-ryborz2025} in Geroch--Traschen regularity $\smash{H^{1,2}_\loc\cap C^0}$ \cite{geroch-traschen1987} and Kunzinger--Ohanyan--Vardabasso \cite{kunzinger-ohanyan-vardabasso2024+} (based on a stability result of Ketterer \cite{ketterer2021-stability}) in Lipschitz regularity; in fact, \cite{mondino-ryborz2025} characterizes the Riemannian manifolds in Geroch--Traschen regularity satisfying the CD condition of Lott--Sturm--Villani. 

Our second main theorem bridges  analytic and synthetic theory for Lipschitz spacetimes by deriving the so-called \emph{timelike measure contraction property}, briefly TMCP, from distributional timelike Ricci curvature bounds. Thus, we reduce the known regularity for this implication from $C^1$ to local Lipschitz continuity. The TMCP was introduced by Cavalletti--Mondino \cite{cavalletti-mondino2020} and later adapted by Braun \cite{braun2023-renyi}. An advantage of the formulation of \cite{braun2023-renyi}, which relies on a different choice of entropy \eqref{Eq:Ren} and on which we focus, is that it implies the comparison theorems indicated above in sharp form, without requiring timelike nonbranching hypotheses; however, our arguments can be used to show the TMCP of \cite{cavalletti-mondino2020} as well.

To formulate our second main result, given $q\in (0,1)$  let $\smash{\ell_q}$ be the $q$-Lorentz--Wasserstein distance \eqref{Eq:ellq} induced by the canonical time separation $l$ of $(\mms,g)$. It is defined on the space $\Prob(\mms)$ of Borel probability measures on $\mms$. As defined by McCann \cite{mccann2020}, a  curve $\mu\colon[0,1]\to\Prob(\mms)$ will be called \emph{$q$-geodesic} if it is affinely parametrized with respect to $\smash{\ell_q}$, cf.~\cref{Def:qgeo}. For $N\in (1,\infty)$, define the \emph{$N$-Rényi entropy} $\smash{\scrS_N(\,\cdot\mid\vol_g)\colon\Prob(\mms)\to\R_-\cup\{-\infty\}}$ with respect to $\smash{\vol_g}$ by
\begin{align}\label{Eq:Ren}
    \scrS_N(\mu\mid\vol_g) := -\int_\mms \Big[\frac{\rmd\mu^\ac}{\rmd\vol_g}\Big]^{1-1/N}\d\vol_g,
\end{align}
where $\mu^\ac$ is the absolutely continuous part in the Lebesgue decomposition of the argument $\mu$ with respect to $\vol_g$.

\begin{definition}[Timelike measure contraction property, \cref{Def:TMCP}] The globally hyperbolic weighted Lipschitz spacetime $(\mms,g,\vol_g)$ is said to obey the past timelike measure contraction property, briefly $\smash{\TMCP^-(0,\dim\mms)}$, if for every $o\in\mms$ and every compactly supported, $\smash{\vol_g}$-absolutely continuous $\mu_1\in\Prob(\mms)$ concentrated on $\smash{I^+(o)}$, there exist
\begin{itemize}
    \item an exponent $q\in (0,1)$ and 
    \item a $q$-geodesic $\mu\colon[0,1]\to\Prob(\mms)$ from $\smash{\mu_0 := \delta_o}$ to $\mu_1$
\end{itemize}
such that for every $N'\in [\dim\mms,\infty)$ and every $t\in[0,1]$,
\begin{align*}
    \scrS_{N'}(\mu_t\mid \vol_g) \leq t\,\scrS_{N'}(\mu_1\mid\vol_g).
\end{align*}
\end{definition}

The TMCP does not depend on the transport exponent $q$, cf.~\cref{Sub:VarTMCP}. Moreover, the above condition $N'\in[\dim\mms,\infty)$ ensures dimensional consistency.

\begin{theorem}[Timelike measure contraction property, \cref{Th:ToTMCP}]\label{Th:TLTMCPintro} We assume that $(\mms,g)$ is a globally hyperbolic Lipschitz spacetime with $\smash{\Ric_g\geq 0}$ timelike distributionally. Then the globally hyperbolic weighted Lipschitz spacetime $\smash{(\mms,g,\vol_g)}$ satisfies $\TMCP^-(0,\dim\mms)$.
\end{theorem}

The ``future'' analog of the past TMCP (where the Dirac mass lies in the chronological future of the absolutely continuous distribution) is established as well.

\begin{remark}[Improvements of \cref{Th:TLTMCPintro}] In fact, we show a slightly stronger statement in \cref{Th:TowTCD}, which improves \cref{Th:TLTMCPintro} to the timelike curvature-dimension condition recalled in \cref{Sub:VarTCD} provided $(\mms,g)$ is timelike nonbranching after Cavalletti--Mondino \cite{cavalletti-mondino2020} by results of Braun \cite{braun2023-renyi}. Under such nonbranching hypotheses, one also gets isoperimetric-type inequalities à la Cavalletti--Mondino  \cite{cavalletti-mondino2024}, exact representation formulas and sharp comparison estimates for d'Alembertians of Lorentz distance functions from general spacelike submanifolds à la Braun \cite{braun2024+}, and a volume incompleteness theorem, cf.~García-Heveling \cite{garcia-heveling2023-volume} and Braun \cite{braun2024+}. Since the exponential map is not even well-defined, it is not clear that Lipschitz spacetimes are timelike nonbranching. However, the Riemannian results of Mondino--Ryborz \cite{mondino-ryborz2025} in even lower regularity and Deng's result about nonbranching of finite-dimensional RCD spaces in the sense of Gigli \cite{gigli2015} suggest the hypotheses of \cref{Th:TLTMCPintro} actually imply timelike non\-branching, hence the above results; see Cavaletti--Mondino \cite{cavalletti-mondino2020}*{Rem.~5.15}.
\end{remark}

Conceptually, our proof is based on Calisti et al.'s good approximation \cite{calisti-graf-hafemann-kunzinger-steinbauer2025+} and the strong stability properties of the TMCP realized by Cavalletti--Mondino \cite{cavalletti-mondino2020} (because it is a convexity inequality). However, in contrast to the work of Braun--Calisti \cite{braun-calisti2023} (which relies on approximation results of Graf \cite{graf2020}), the timelike lower bounds of the approximating Ricci tensors are not locally uniformly small relative to the bound prescribed by $g$, but only small in an integral sense. We will address this challenge by combining Braun--McCann's recent theory of metric measure space\-times with variable timelike lower Ricci curvature bounds \cite{braun-mccann2023} and estimates developed by Ketterer \cite{ketterer2021-stability} in Riemannian signature.

\subsection{Applications to comparison theory}\label{Sub:APPLCOMP} Lastly, we establish several sharp comparison theorems as a consequence of \cref{Th:TLTMCPintro}: time\-like Brunn--Minkowski, time\-like Bishop--Gromov, and d'Alembert comparison\footnote{We note that per se, the timelike Bonnet--Myers inequality is another standard implication of synthetic timelike Ricci curvature bounds, cf.~Cavalletti--Mondino \cite{cavalletti-mondino2020}. However, in our work \cref{Th: 1.1} is a prerequisite for \cref{Th:TLTMCPintro}, not a consequence: the estimates of Ketterer \cite{ketterer2021-stability} we use, cf.~\cref{Le:DefDist}, necessitate an upper bound on the length of the geodesics in question. Consequently, we must ensure ad hoc longer geodesics do not arise.}. When $g$ is smooth, comparison theory was systematically developed by Eschenburg \cite{eschenburg1988}, Ehrlich--Jung--Kim \cite{ehrlich-jung-kim1998}, Ehrlich--Sánchez \cite{ehrlich-sanchez2000}, Treude \cite{treude2011}, and Treude--Grant \cite{treude-grant2013}. On nonsmooth space\-times, the facts we prove for Lipschitz spacetimes were only known in regularity $\smash{C^{1,1}}$ by Graf \cite{graf2016} and $C^1$ by Braun--Calisti \cite{braun-calisti2023}. 

As established by Cavalletti--Mondino \cite{cavalletti-mondino2020} and later Braun \cite{braun2023-renyi}, the first two comparison theorems we will obtain are standard consequences of the TMCP from \cref{Th:TLTMCPintro}. Given $X_0,X_1\subset\mms$, let $G(X_0,X_1)$ be the set of all timelike affinely parametrized maximizing geodesics starting in $X_0$ and terminating in $X_1$. Given $t\in [0,1]$, let $\eval_t\colon C^0([0,1];\mms)\to\mms$ denote the evaluation map $\eval_t(\gamma):=\gamma_t$.

\begin{theorem}[Timelike Brunn--Minkowski inequality, \cref{Th:TimelikeBrunnMinkowski}] Assume $(\mms,g)$ is a globally hyperbolic Lipschitz spacetime with $\smash{\Ric_g\geq 0}$ timelike distributionally. Then for every $o\in\mms$, every compact set $\smash{X_1\subset I^+(o)}$, and every $t\in[0,1]$,
\begin{align*}
    \vol_g^*[\eval_t(G(\{o\},X_1))]^{1/\dim\mms}\geq t\,\vol_g[X_1]^{1/\dim\mms},
\end{align*}
where $\smash{\vol_g^*}$ denotes the outer measure induced by $\smash{\vol_g}$.
\end{theorem}

Next, let $\smash{E\subset I^+(o)\cup\{o\}}$ be compact and star-shaped with respect to a point $o\in\mms$. Given $r>0$, let $v(r)$ denote the volume of the intersection of $E$ with the hyperboloid  of all points in $\smash{I^+(o)}$ with distance at most $r$ to $o$; moreover, let $s(r)$ denote the area of the intersection of $E$ with the hypersurface of all points in $\smash{I^+(o)}$ with distance \emph{exactly} $r$ to $o$. We refer to \cref{Sub:Applic} for details about these notions.

\begin{theorem}[Timelike Bishop--Gromov inequality, \cref{Th:TimelikeBishopGromov}] Let $(\mms,g)$ be a globally hyperbolic Lipschitz spacetime with $\smash{\Ric_g\geq 0}$ timelike distributionally. Given  $o\in\mms$, fix a compact set $\smash{E\subset I^+(o)\cup\{o\}}$ that is star-shaped with respect to $o$. Then for every $r,R >0$ with $r<R$,
\begin{align*}
    \frac{s(r)}{s(R)} &\geq \frac{r^{\dim\mms-1}}{R^{\dim\mms-1}},\\
    \frac{v(r)}{v(R)} &\geq \frac{r^{\dim\mms}}{R^{\dim\mms}}.
\end{align*}
\end{theorem}

Our last major results are d'Alembert comparison theorems. They were established in the more abstract  setting of TMCP metric measure spacetimes by Beran et al.~\cite{beran-braun-calisti-gigli-mccann-ohanyan-rott-samann+-}. However, rather than verifying compatibility with the differential calculus developed in \cite{beran-braun-calisti-gigli-mccann-ohanyan-rott-samann+-} and its counterparts for Lipschitz spacetimes --- which would render the results below direct applications of \cite{beran-braun-calisti-gigli-mccann-ohanyan-rott-samann+-} --- we instead provide a self-contained and streamlined proof based on the arguments of \cite{beran-braun-calisti-gigli-mccann-ohanyan-rott-samann+-}; see \cref{App:One}. 

Let $l_o$ denote the future distance function \eqref{Eq:LorDistFct} from a point $o\in\mms$. It is  locally Lipschitz continuous on $\smash{I^+(o)}$, a result which is a special case of \cref{Le:EquiLip} we establish in \cref{Sub:Timelike cut locus}.

\begin{theorem}[D'Alembert comparison theorem for powers of Lorentz distance functions, \cref{Th:Dalembert powers}] Assume $(\mms,g)$ is a globally hyperbolic Lipschitz spacetime. Suppose that $\smash{\Ric_g\geq 0}$ timelike distributionally. Let $q\in (0,1)$ and let $q'<0$ be its conjugate exponent. Then for every $o\in\mms$, the $q'$-d'Alembert comparison estimate
\begin{align}\label{Eq:DIV1}
    \mathrm{div}\Big[\Big\vert\nabla \frac{l_o^q}{q}\Big\vert^{q'-2}\,\nabla\frac{l_o^q}{q}\Big]  \leq \dim\mms
\end{align}
holds distributionally on $\smash{I^+(o)}$, meaning that for every Lipschitz continuous function $\smash{\phi\colon I^+(o)\to\R_+}$ with compact support,
\begin{align*}
    -\int_\mms\rmd\phi\Big[\nabla\frac{l_o^q}{q}\Big]\,\Big\vert\nabla\frac{l_o^q}{q}\Big\vert^{q'-2}\d\vol_g \leq \dim\mms\int_\mms\phi\d\vol_g.
\end{align*}
\end{theorem}

The inequality \eqref{Eq:DIV1} should be compared to the analogous classical comparison estimate for the (two-)Laplacian of the squared distance over two on Riemannian manifolds, cf.~Cheeger--Ebin \cite{cheeger-ebin1975}.

\begin{theorem}[D'Alembert comparison theorem for Lorentz distance functions, \cref{Th:DAlembert}]  Let $(\mms,g)$ be a globally hyperbolic Lipschitz spacetime with $\smash{\Ric_g\geq 0}$ timelike distributionally. Then for every $o\in\mms$, the d'Alembert comparison estimate
\begin{align}\label{Eq:DIV2}
    \mathrm{div}\,\nabla  l_o  \leq \frac{\dim\mms-1}{l_o}
\end{align}
holds distributionally on $\smash{I^+(o)}$, meaning that for every Lipschitz continuous function $\smash{\phi\colon I^+(o)\to\R_+}$ with compact support,
\begin{align*}
    -\int_\mms\rmd\phi(\nabla l_o)\d\vol_g\leq (\dim\mms-1)\int_\mms\frac{\phi}{l_o}\d\vol_g.
\end{align*}
\end{theorem}

We expect these comparison theorems to be influential. First, they imply the formal left-hand sides of \eqref{Eq:DIV1} and \eqref{Eq:DIV2} define generalized signed Radon measures; see \cref{Cor:Ex1,Cor:Ex2}. This holds despite the nonsmoothness of the functions involved, which already arises even when $g$ is smooth; cf.~\cref{Sub:Timelike cut locus}. Second, the above weak formulations of \eqref{Eq:DIV1} and \eqref{Eq:DIV2} extend across the future timelike cut locus of $o$, an extension that was not known even in the smooth setting prior to the work of Beran et al.~\cite{beran-braun-calisti-gigli-mccann-ohanyan-rott-samann+-}, which is now established in Lipschitz regularity. Third, they have recently led to a significantly simpler and unexpected ``elliptic'' proof of the Eschenburg--Galloway--Newman splitting theorem mentioned above, as well as to its extension to regularity $C^1$ by Braun et al.~\cite{braun-gigli-mccann-ohanyan-samann+,braun-gigli-mccann-ohanyan-samann+++}. Our d'Alembert comparison theorems open the door to extending the splitting theorem to Lipschitz spacetimes, which we leave for future work.

\subsection{Organization} In \cref{Sec:ch2}, we review basic notions of the Lorentzian structure of (weighted)  Lipschitz spacetimes. We further discuss how to create a suitable smooth approximation of the nonsmooth metric tensor and discuss its properties regarding time separation and curvature. In \cref{Sec:SynTLRic}, we outline the basics of Lorentzian optimal transport. In \cref{Sub:Smooth}, we collect results for \emph{smooth}  spacetimes that are necessary for our main results. Then  \cref{Sec:BonnetMyers} contains the proof of Theorem \ref{Th: 1.1}. We provide the proofs for Theorem \ref{Th:TLTMCPintro} in \cref{Sub:FromAntoSyn} and discuss its main applications in \cref{Sub:Applic}. Finally, we establish \cref{Th:TLPS} in \cref{Sec:Further}.

\section{Analytic timelike Ricci curvature lower bounds}\label{Sec:ch2}

\subsection{Lipschitz spacetime geometry}\label{Sub:LipSptGeo} In this subsection, we collect basics about spacetimes with locally Lipschitz continuous metric tensors. For general causality theory, we refer to Minguzzi's review \cite{minguzzi2019-causality} and to Chrus\'{c}iel--Grant \cite{chrusciel-grant2012} and Sämann \cite{samann2016} for spacetimes with merely continuous metric tensors. 

Throughout the sequel, $\mms$ is a topological manifold that is smooth, Hausdorff, second-countable, and connected, with dimension $\dim\mms$ at least $2$. Let $g$ be a fixed locally Lipschitz continuous Lorentzian metric, i.e.~a metric tensor of Lorentzian signature $+,-,\dots,-$ whose coefficients in charts are locally Lipschitz continuous; we define Lorentzian metrics to have other regularities analogously. (In \cref{Sub:Smooth,Sec:Further}, we will in fact assume smoothness of $g$.) In particular, our sign convention for timelike vectors $v\in T\mms$ is $g(v,v)>0$, in which case we will occasionally write $\smash{\vert v\vert:= \sqrt{g(v,v)}}$. We assume $(\mms,g)$ is a \emph{Lipschitz spacetime}, i.e.~$g$ is locally Lipschitz continuous and there exists  a (tacitly fixed) continuous timelike vector field $Z$ on $\mms$; if $g$ is smooth, we simply call $(\mms,g)$ \emph{spacetime}. Given $o\in\mms$, let $\smash{T_o^+}$ be  the set of all future-directed timelike vectors in $T_o\mms$; moreover, let $\smash{T^+}$ denote the disjoint union of $\smash{T_o^+}$ over all $o\in\mms$. Unless  stated otherwise, all tangent vectors, curves, vector fields, etc.~will tacitly be assumed to be future-directed with respect to $Z$. 

Throughout, we fix a tacit complete Riemannian metric $r$ on $\mms$, assumed to be smooth. Its existence is guaranteed by Nomizu--Ozeki \cite{nomizu-ozeki1961}. For $v\in T\mms$, we write $\smash{\vert v\vert_r:=\sqrt{r(v,v)}}$. The induced length functional will be written $\Len_r$, while the induced metric will be denoted $\met_r$.

We will use the standard signs $\leq$ and $\ll$ to denote causality and chronology of $(\mms,g)$. In addition, we use the standard letters $J$ and $I$ to denote their induced sets; for instance, $\smash{I^+(o)\subset\mms}$ denotes the chronological future of a point $o\in\mms$.

The \emph{time separation function} $l\colon \mms^2\to\R_+\cup\{-\infty,\infty\}$ is given by
\begin{align}\label{Eq:lxy}
\begin{split}
    l(x,y) &:= \sup\{\Len_g(\gamma) : \gamma\colon[0,1]\to \mms \textnormal{ smooth}\\
    &\qquad\qquad \textnormal{causal curve with }\gamma_0=x\textnormal{ and }\gamma_1=y\}, 
    \end{split}
\end{align}
where we adopt the convention $\sup\emptyset :=-\infty$ and abbreviate
\begin{align*}
    \Len_g(\gamma) := \int_{[0,1]}\sqrt{g(\dot\gamma_t,\dot\gamma_t)}\d t.
\end{align*}
It obeys the customary reverse triangle inequality (where we adopt the convention $-\infty+z := z+(-\infty) := -\infty$ for every $z\in\R\cup\{-\infty,\infty\}$). Given $L>0$, we will say a curve $\gamma\colon[0,L]\to\mms$ is \emph{proper time  parametrized} if $l(\gamma_s,\gamma_t)=t-s$ for every $s,t\in[0,L]$ with $s<t$. We say a causal curve $\gamma\colon[0,1]\to\mms$ is \emph{affinely parametrized} provided $l(\gamma_s,\gamma_t)=(t-s)\,l(\gamma_0,\gamma_1)$ for every $s,t\in[0,1]$ with $s<t$. A causal curve attaining the supremum \eqref{Eq:lxy} and solving the geodesic equation in the sense of Filippov \cite{filippov1988} will be called \emph{maximizing geodesic}.  Lange--Lytchak--Sämann \cite{lange-lytchak-samann2021} show that on Lipschitz spacetimes, every causal maximizer of \eqref{Eq:lxy} admits a reparametrization into a $C^{1,1}$-curve which is a maximizing geodesic in the above sense. (If $g$ is even of regularity $C^1$, Graf \cite{graf2020} had previously proved reparametrizability of causal maximizers into $C^2$-curves solving the geodesic equation in the classical sense.) Furthermore, by the clause after \cite{lange-lytchak-samann2021}*{Prop.~1.4}, the $C^{1,1}$-norm of that reparametrization is bounded from above by the $C^{0,1}$-norm of $g$ and the $r$-length of $\gamma$ (where the occurring norms can be understood e.g.~as uniform norms with respect to $r$,  cf.~\eqref{Eq:NormFormulas}). In particular, maximizing geodesics (and maximizers) have a definite causal character, which means they cannot have both timelike and lightlike subsegments. This fact confirms an earlier result of Graf--Ling \cite{graf-ling2018}. 

\begin{convention}[Multiple metric tensors] If multiple metric tensors are involved, to avoid confusion or ambiguity we will occasionally tag a quantity induced by the corresponding metric tensor $g$ with a subscript. For instance, we write $\smash{\leq_g}$ instead of $\leq$, $\smash{I^+_g(o)}$ instead of $\smash{I^+(o)}$ (where $o\in\mms$), $\smash{l_g}$ instead of $l$, etc.
\end{convention}

\begin{definition}[Global hyperbolicity] We will call  the Lipschitz spacetime $(\mms,g)$ \emph{globally hyperbolic} if the following properties hold.
\begin{enumerate}[label=\textnormal{(\roman*)}]
    \item For every compact set $C\subset \mms$, there exists $c>0$ such that the $r$-length of all smooth causal curves whose image is contained in $C$ is  bounded from above by $c$.
    \item For every $x,y\in\mms$, the causal diamond $J(x,y)\subset\mms$ is compact.
\end{enumerate}
\end{definition}

The property from the first item is called \emph{nontotal imprisonment}.

Several equivalent characterizations of the preceding  (classical) notion of global hyperbolicity, even when $g$ is merely continuous, are given in the work of Sämann \cite{samann2016}. A few consequences are worth mentioning as well, cf.~Sämann \cite{samann2016}*{Prop.~3.3, Cor.~3.4} and Minguzzi \cite{minguzzi2019-applications}*{Prop. 2.26, Thm.~2.55} for details.

\begin{theorem}[Properties of globally hyperbolic Lipschitz spacetimes]\label{Th:ConsequGH} Let $(\mms,g)$ be a globally hyperbolic Lipschitz spacetime. Then the following   hold.
\begin{enumerate}[label=\textnormal{(\roman*)}]
    \item The causal relation $\leq$ is closed.
    \item For all compact $X,Y\subset\mms$, the causal emerald $J(X,Y)\subset\mms$ is compact.
    \item The positive part $\smash{l_+}$ of the time separation function is finite and continuous.
    \item All points $x,y\in\mms$ with $x\leq y$ are connected by a maximizing geodesic.
\end{enumerate}
\end{theorem}

Finally, let $\vol_g$ denote the volume measure induced by $g$. We will fix a Borel measure $\meas$ such that the negative logarithmic density $-\log\rmd\meas/\rmd\vol_g$ is locally Lipschitz continuous. We call $\meas$ smooth if $-\log\rmd\meas/\rmd\vol_g$ is. We will say a Borel subset of $\mms$ has measure zero if it is negligible with respect to some (hence every) smooth measure on $\mms$. We will define a function to be locally in $L^\infty$ analogously, without referring to a specific choice of  measure.

\begin{definition}[Weighted Lipschitz spacetime]\label{Def:WEighted} We will call the triple $(\mms,g,\meas)$  \emph{globally hyperbolic weighted Lipschitz spacetime} if 
\begin{enumerate}[label=\textnormal{\alph*.}]
    \item $(\mms,g)$ is a globally hyperbolic Lipschitz spacetime and 
    \item $\meas$ forms a Borel measure on $\mms$ such that the negative logarithmic density $\smash{-\log\rmd\meas/\rmd\vol_g}$ is locally Lipschitz continuous.
\end{enumerate}
\end{definition}

\subsection{Regularization}\label{Sub:Approx} Since $g$ has merely locally Lipschitz continuous coefficients in almost all of our work, fundamental geometric quantities like curvature or the exponential map are not classically defined. To work with these, we will regularize $g$ in such a way that possible distributional timelike Ricci curvature bounds are ``almost'' preserved in an $L^p$-sense. Such approximations have been constructed by Calisti et al. \cite{calisti-graf-hafemann-kunzinger-steinbauer2025+}, whose approach we extend to the weighted case. We also refer to Grosser--Kunzinger--Oberguggenberger--Steinbauer \cite{grosser-kunzinger-oberguggenberger-steinbauer2001} and Graf \cite{graf2020}  for more details and background about regularization of distributional tensor fields.

\subsubsection{Tensor distributions} Let $\Vol(\mms)$ denote the \emph{volume bundle} of $\mms$; if $\mms$ is orientable (which we do not assume), this is simply the vector bundle $\Lambda^{\dim\mms}T^*\mms$ of top-degree differential forms. 

Let $\Gamma_\comp(\Vol(\mms))$ denote the space of smooth and compactly supported sections of $\Vol(\mms)$, whose elements are called \emph{volume densities}. Given a volume density $\mu\in\Gamma_\comp(\Vol(\mms))$ and an open set $U\subset\mms$, the integral $\smash{\int_U \mu}$ is defined in analogy to the integration of top-degree differential forms on orientable manifolds. We say $\mu$ is nonnegative if $\smash{\int_U\mu\geq 0}$ for every open set $U\subset\mms$. 

The set $\scrD'(\mms)$ of scalar distributions is defined as the topological dual space of $\Gamma_c(\Vol(\mms))$ with duality pairing $\langle\,\cdot\mid\cdot\,\rangle$. The space $C^\infty(\mms)$ embeds into $\scrD'(\mms)$ by identifying $\psi\in C^\infty(\mms)$ with the assignment $\smash{\langle \psi\mid \mu\rangle :=  \int_\mms \psi\,\mu}$. 

\begin{definition}[Comparison of distributions]\label{Def:ComDist} Given $u,v\in\scrD'(\mms)$, we write $u\geq v$ if $\langle u\mid\mu\rangle\geq \langle v\mid \mu\rangle$ for every nonnegative volume density $\smash{\mu\in\Gamma_\comp(\Vol(\mms))}$.
\end{definition}

In fact, every nonnegative scalar distribution (i.e.~every $u\in\scrD'(\mms)$ satisfying $u\geq 0$ in the above sense) is a measure.

Since we will deal with distributional (metric tensors and) curvature, we also recall the space of $(0,2)$-tensor distributions defined by
\begin{align*}
    \scrD'((T^*)^{\otimes 2}\mms) := \scrD'(\mms)\otimes_{C^\infty(\mms)} \Gamma((T^*)^{\otimes 2}\mms),
\end{align*}
where $\smash{\otimes_{C^\infty(\mms)}}$ designates the classical tensor product over the ring $C^\infty(\mms)$. Like their smooth counterparts, $(0,2)$-tensor  distributions admit local coordinate representations on each chart. By pushing forward and pulling back the respective  coefficients, such  distributions can locally be seen as a vector-valued distribution on an open subset  of $\smash{\R^{\dim\mms}}$, cf.~Graf \cite{graf2020}*{Prop.~3.1}. In particular, since $g$ is locally Lipschitz continuous and letting   $f\colon\mms\to\R$ be locally Lipschitz continuous, each directional derivative of $f$  and the  \emph{Christoffel symbols}, locally defined by
\begin{align*}
    \Gamma_{ij}^k := \frac{g^{kl}}{2}\,\Big[\frac{\partial g_{jl}}{\partial x^i}+ \frac{\partial g_{il}}{\partial x^j} - \frac{\partial g_{ij}}{\partial x^l}\Big],
\end{align*}
are all locally in $L^\infty$. Given a smooth vector field $X\in\mathfrak{X}(\mms)$, say with support in a single chart, the local formulas
\begin{align}\label{Eq:LocalCoordCurv}
\begin{split}
    \Ric_{g}(X,X) &:= \Big[\frac{\partial\Gamma_{ij}^m}{\partial x^m}-\frac{\partial\Gamma_{im}^m}{\partial x^j}+ \Gamma_{ij}^m\,\Gamma_{km}^k - \Gamma_{ik}^m\,\Gamma_{jm}^k\Big]\,X^i\,X^j,\\
    \Hess_g f(X,X) &:= \Big[\frac{\partial^2f}{\partial x^i\partial x^j} -\Gamma_{ij}^k\,\frac{\partial f}{\partial x^k}\Big]\,X^i\,X^j,\\
    (\rmd f\otimes \rmd f)(X,X) &:= \frac{\partial f}{\partial x^i}\,\frac{\partial f}{\partial x^j}\,X^i\,X^j
    \end{split}
\end{align}
yield well-defined objects in $\scrD'(\mms)$. They induce  $(0,2)$-tensor distributions $\smash{\Ric_g}$, $\smash{\Hess_gf}$, and $\rmd f\otimes \rmd f$,  respectively. 

Recall the reference measure $\meas$ on $\mms$, for which  we assume $-\log\rmd\meas/\rmd\vol_g$ is locally Lipschitz continuous.

\begin{definition}[Distributional Bakry--\smash{Émery}--Ricci tensor]\label{Def:DistBER} For $N\in [\dim\mms,\infty)$ we define the \emph{distributional $N$-Bakry--\smash{Émery}--Ricci tensor} $\smash{\Ric_{g,\meas,N}\in\scrD'((T^*)^{\otimes 2}\mms)}$ through the formula
\begin{align*}
    \Ric_{g,\meas,N} := \begin{cases}\begin{aligned}&\displaystyle \Ric_g -  \Hess_g \log\frac{\rmd \meas}{\rmd\vol_g}\\
    &\qquad\qquad\displaystyle -\, \frac{1}{N-\dim\mms}\,\rmd \log\frac{\rmd\meas}{\rmd\vol_g}\otimes\rmd \log\frac{\rmd\meas}{\rmd\vol_g}\end{aligned} & \textnormal{\textit{if} } N > \dim\mms,\\
    \Ric_g & \textnormal{\textit{otherwise}}.
    \end{cases}
\end{align*}
\end{definition}


\begin{definition}[Bounds on distributional Bakry--\smash{Émery}--Ricci tensor]\label{Def:Bounds} For $K\in\R$ and $N\in [\dim\mms,\infty)$, we say \emph{$\smash{\Ric_{g,\meas,N}\geq K}$ timelike distributionally} if  
\begin{enumerate}[label=\textnormal{\alph*.}]
    \item in the case $N > \dim\mms$, every timelike vector field $X\in\mathfrak{X}(\mms)$ obeys
    \begin{align*}
    \Ric_{g,\meas,N}(X,X) \geq K\,g(X,X)    
    \end{align*}
    in the  sense of \cref{Def:ComDist},
    \item otherwise, the density $\rmd\meas/\rmd\vol_g$ is constant and every timelike vector field $X\in\mathfrak{X}(\mms)$ obeys
    \begin{align*}
    \Ric_g(X,X)\geq K\,g(X,X)    
    \end{align*}
    in the sense of \cref{Def:ComDist}.
\end{enumerate}
\end{definition}

\begin{remark}[Scaling] Note that the tensor distribution  $\smash{\Ric_{g,\meas,N}}$ (and its timelike distributional lower boundedness) are unchanged by scaling the reference measure $\meas$ by a positive constant.
\end{remark}

\subsubsection{Good approximation} Let $\{\rho_\varepsilon : \varepsilon>0\}$ be a family of standard mollifiers on $\smash{\R^{\dim\mms}}$. Using a partition of unity and chartwise convolution of local coefficients, given $\smash{T\in\scrD'((T^*)^{\otimes 2}\mms)}$ one can define its convolution $T*_M \rho_\varepsilon$, which becomes an element of $\smash{\Gamma((T^*)^{\otimes 2}\mms)}$ for every $\varepsilon >0$, cf.~e.g.~Graf \cite{graf2020}*{§3.3} for details;  an analogous statement is true for the regularization $u*_\mms\rho_\varepsilon$ of a scalar distribution $u\in\scrD'(\mms)$. Many standard properties of convolutions of distributions on $\smash{\R^{\dim\mms}}$ transfer to regularizations of such distributions on $\mms$.

The first is \emph{convergence}. Recall $r$ is the background Riemannian metric on $\mms$; we denote its  induced Levi-Civita connection and volume measure by $\smash{^r\nabla}$ and $\vol_r$, respectively. For a given $\smash{T\in\Gamma((T^*)^{\otimes 2}\mms)}$ and $i\in\N_0$, let $\smash{\vert{}^r\nabla^i T\vert\colon\mms\to\R_+}$ be the pointwise operator norm of $\smash{{}^r\nabla^i T}$; here, $\smash{^r\nabla^i}$ denotes the $i$-fold application of $\smash{^r\nabla}$, such that $\smash{^r\nabla^0T:= T}$. For a compact subset $C\subset\mms$ and $p\in [1,\infty)$, we set
\begin{align}\label{Eq:NormFormulas}
\begin{split}
    \big\Vert {}^r\nabla^i T\big\Vert_{L^p,C} &:= \Big[\!\int_C \big\vert{}^r\nabla^iT\big\vert^p\d\vol_r\Big]^p,\\
    \big\Vert {}^r\nabla^i T\big\Vert_{\infty,C} &:= \sup\{\big\vert{}^r\nabla^iT\big\vert(x) : x\in C\}.
    \end{split}
\end{align}

The following properties addressing the regularizations of the locally Lipschitz continuous metric $g$ were established by Calisti et al.~\cite{calisti-graf-hafemann-kunzinger-steinbauer2025+}*{Lem.~3.2}.

\begin{lemma}[Basic properties of regularizations of $g$] Let $C\subset\mms$ be compact. Then the following properties hold.
\begin{enumerate}[label=\textnormal{(\roman*)}]
    \item We have the uniform Lipschitz estimate
    \begin{align*}
        \sup\{\big\Vert{}^r\nabla (g*_\mms\rho_\varepsilon)\big\Vert_{\infty,C}: \varepsilon >0\}<\infty.
    \end{align*}
    \item For every $p\in [1,\infty)$, we have $g*_\mms\rho_\varepsilon\to g$ locally in $W^{1,p}$  as $\smash{\varepsilon\to 0^+}$, viz.
    \begin{align*}
        \lim_{\varepsilon\to 0^+} \big[\big\Vert (g*_\mms\rho_\varepsilon)-g\big\Vert_{L^p,C} + \big\Vert  {}^r\nabla (g*_\mms\rho_\varepsilon)-{}^r\nabla g\big\Vert_{L^p,C}\big]=0.
    \end{align*}
\end{enumerate}
\end{lemma}

\begin{remark}[About the choice of $\vol_r$] The choice of $\vol_r$ as reference measure in \eqref{Eq:NormFormulas} is merely cosmetic yet  inessential for our arguments to follow. It may and will be replaced by a converging sequence of weighted measures whose densities (with respect to metric tensors with Lorentzian signature) are positive, continuous, and converge locally uniformly, cf.~\cref{Le:ApproxLemma}. In particular, all adjacent statements about local convergence of tensor distributions remain unaltered.
\end{remark}

The second is  \emph{preservation of nonnegativity}. Namely, if $u,v\in\scrD'(\mms)$ satisfy $u\geq v$ in the distributional sense of \cref{Def:ComDist},  then $u*_\mms\rho_\varepsilon \geq v*_\mms\rho_\varepsilon$  holds pointwise on $\mms$ for every $\varepsilon>0$.

There are two issues with these ``naive'' regularizations. First, given $\varepsilon>0$ the causal structures from $g$ and $g*_\mms\rho_\varepsilon$ are in general unrelated. Second, let $X\in\mathfrak{X}(\mms)$ be a timelike vector field. If $\smash{\Ric_{g,\meas,N}\geq K}$ timelike distributionally for some $K\in\R$ and $N\in[\dim\mms,\infty)$, we obtain $\smash{\Ric_{g,\meas,N}(X,X)*_\mms\rho_\varepsilon \geq K\,g(X,X)*_\mms\rho_\varepsilon}$ by the preceding paragraph. As the formulas \eqref{Eq:LocalCoordCurv} depend on $g$ and $-\log\rmd\meas/\rmd\vol_g$ in a nonlinear manner, it is not true in general that the latter inequality  implies a geometrically more tangible condition of the form $\smash{\Ric_{g_\varepsilon,\meas_\varepsilon,N}(X,X) \geq K\,g_\varepsilon(X,X)}$ for any \emph{smooth} regularizations $\smash{\{g_\varepsilon:\varepsilon > 0\}}$ and $\smash{\{\meas_\varepsilon : \varepsilon > 0\}}$ of $g$ and $\meas$, respectively.

The setup of a net $\smash{\{\check{g}_\varepsilon:\varepsilon >0\}}$ of smooth metric tensors of Lorentzian signature that approximates $g$, respects its causal structure, and   preserves its distributional curvature bounds in an $L^p$-sense is one of the main results of Calisti et al.~\cite{calisti-graf-hafemann-kunzinger-steinbauer2025+}. Their approach was inspired by prior  works of Kunzinger--Steinbauer--Stojkovi\'{c}--Vickers \cite{kunzinger-steinbauer-stojkovic-vickers2015} (assuming $g$ is locally  $C^{1,1}$) and Graf \cite{graf2020} (assuming $g$ is of class $C^1$). The  curvature aspect will be discussed below; for now, we focus on causality. 

Recall for two rough Lorentzian metrics $g_1$ and $g_2$ on $\mms$, we write $g_1\prec g_2$ if for every $v\in T\mms\setminus\{0\}$, $g_1(v,v)\geq 0$ implies $g_2(v,v)>0$; more pictorially, $g_1$ has strictly narrower lightcones than $g_2$. Moreover, recall the local norms from \eqref{Eq:NormFormulas}.

\begin{lemma}[Existence of good approximation {\cite{calisti-graf-hafemann-kunzinger-steinbauer2025+}*{Lem.~2.4}}]\label{Le:ApproxLemma} There is a family $\smash{\{\check{g}_\varepsilon : \varepsilon > 0\}}$ of smooth Lorentzian metrics, time-orientable by the same continuous vector field as $g$, with the following properties.
\begin{enumerate}[label=\textnormal{(\roman*)}]
    \item For every $\varepsilon >0$, we have $\smash{\check{g}_\varepsilon \prec g}$. 
    \item For every $p\in[1,\infty)$, we have $\smash{\check{g}_\varepsilon\to g}$ and $\smash{\check{g}_\varepsilon^{-1}\to g^{-1}}$ locally uniformly and locally in $W^{1,p}$ as $\smash{\varepsilon\to 0^+}$, i.e. for every compact set $C\subset\mms$,
    \begin{align*}
        \lim_{\varepsilon\to 0^+}\big\Vert \check{g}_\varepsilon- g\big\Vert_{\infty,C}&=0,\\
        \lim_{\varepsilon\to 0^+}\big[\big\Vert \check{g}_\varepsilon- g\big\Vert_{L^p,C}+\big\Vert {}^r\nabla\check{g}_\varepsilon-{}^r\nabla g\big\Vert_{L^p,C}\big]&=0
    \end{align*}
    and analogously for the inverses.
    \item For every $p\in [1,\infty)$, we have $\smash{\check{g}_\varepsilon-g*_\mms\rho_\varepsilon\to 0}$ and $\smash{\check{g}_\varepsilon^{-1}-(g*_\mms\rho_\varepsilon)^{-1}\to 0}$ locally in $C^\infty$ as $\smash{\varepsilon\to 0^+}$, i.e.~for all $i\in\N_0$ and every compact set $C\subset\mms$,
    \begin{align*}
        \lim_{\varepsilon\to 0^+}\big\Vert{}^r\nabla^i\check{g}_\varepsilon-{}^r\nabla ^i(g*_\mms\rho_\varepsilon)\big\Vert_{\infty,C}=0
    \end{align*}
    and analogously for the inverses. 
    \item For every compact set $C\subset\mms$, there exists a  sequence $\smash{(\varepsilon_n)_{n\in\N}}$ in $(0,\infty)$ decreasing to zero, depending on $C$, with $\smash{\check{g}_{\varepsilon_n}\prec \check{g}_{\varepsilon_{n+1}}}$ for every $n\in\N$.
\end{enumerate}
\end{lemma}

\begin{remark}[Inheritance of global hyperbolicity] In the above lemma, $\smash{\check{g}_\varepsilon}$ inherits global hyperbolicity from $g$ for every $\varepsilon >0$ as a consequence of the first item, cf. Sämann \cite{samann2016}*{Thms.~5.7, 5.9}.
\end{remark}

\begin{definition}[Good approximation]\label{Def:Goodapprox} Given a compact set $C\subset\mms$, we call two sequences $(g_n)_{n\in\N}$ and $(\meas_n)_{n\in\N}$ of Lorentzian metrics and smooth measures on $\mms$ \emph{good approximation} of $g$ and $\meas$ on $C$, respectively, if for every $n\in\N$,
\begin{align*}
        g_n &=\check{g}_{\varepsilon_n},\\
        \log\frac{\rmd\meas_n}{\rmd\vol_{g_n}} &= \Big[\!\log\frac{\rmd\meas}{\rmd\vol_g}\Big]*_\mms\rho_{\varepsilon_n},
    \end{align*}
    where the family $\smash{\{\check{g}_\varepsilon:\varepsilon>0\}}$ satisfies the conclusions of \cref{Le:ApproxLemma} and $\smash{(\varepsilon_n)_{n\in\N}}$ is the sequence from the last statement therein.
\end{definition}

\subsubsection{Properties of good approximation} We fix good approximations $(g_n)_{n\in\N}$ and $(\meas_n)_{n\in\N}$  of $g$ and $\meas$  on a compact set $C\subset\mms$, respectively. Their first  property concerns the positive parts of their induced time separation functions, cf.~Calisti et al.~\cite{calisti-graf-hafemann-kunzinger-steinbauer2025+}*{Lem.~5.5} for the proof.

\begin{lemma}[Locally uniform convergence of time separations functions]\label{Le:UnifCvg} The sequence $\smash{(l_{g_n,_+})_{n\in\N}}$   converges locally uniformly to $\smash{l_{g,+}}$.
\end{lemma}

The retention of timelike distributional curvature bounds from $g$ and $\meas$ is more delicate. Kunzinger--Steinbauer--Stojkovi\'{c}--Vickers \cite{kunzinger-steinbauer-stojkovic-vickers2015} (when $g$ is $C^{1,1}$) and Graf \cite{graf2020} (when $g$ is $C^1$) previously obtained a lower bound for the curvature tensors induced by $\smash{\{\check{g}_\varepsilon:\varepsilon >0\}}$ whose difference to that of the timelike distributional curvature bound originally given by $g$ is locally uniformly small. Their considerations carry over to the weighted case, cf.~Braun--Calisti \cite{braun-calisti2023}*{Lem.~2.8}.  In our case where $g$ and $\smash{-\log\rmd\meas/\rmd\vol_g}$ are merely locally Lipschitz continuous, one cannot expect to retain this locally uniform error. However, one may still hope for a very rough uniform lower bound and, for more precise estimates, smallness in $L^p$. This is the content of \cref{Le:Crude,Le:Lpcvg} summarizing Calisti et al.'s results~\cite{calisti-graf-hafemann-kunzinger-steinbauer2025+}.

We first  introduce some  notation. We fix  $N\in[\dim\mms,\infty)$. Let $\smash{V\subset T^+\big\vert_C}$ be compact. 
We say that $V$ is eventually uniformly timelike if there are $\lambda > 0$ and $n_0\in\N$ such that for every $n\in\N$ with $n\geq n_0$, we have $g_n(v,v)\geq \lambda^2$ for every $v\in V$. With this definition, we can already state the following straightforward generalization of Calisti et al.'s result \cite{calisti-graf-hafemann-kunzinger-steinbauer2025+}*{Prop.~5.8} from the unweighted case.

\begin{lemma}[Rough uniform lower bound]\label{Le:Crude} Assume $K\in\R$ and $N\in[\dim\mms,\infty)$ satisfy $\smash{\Ric_{g,\meas,N}\geq K}$ timelike distributionally. Let $\smash{V\subset T^+\big\vert_C}$ be compact and uniformly timelike, where $C\subset\mms$ is compact. Then there exist $\rho\in\R_+$ and $n_0\in\N$ such that for every $n\in\N$ with $n\geq n_0$ and every $v\in V$,
\begin{align*}
    \Ric_{g_n,\meas_n,N}(v,v) \geq -\rho\,g_n(v,v). 
\end{align*}
\end{lemma}

To obtain finer estimates, assume again $V$ is as above. Then for every $n\in\N$ with $n\geq n_0$, the function $\smash{k_{V,g_n,\meas_n,N}\colon C \to \R}$ given by
\begin{align}\label{Eq:kVgm}
\begin{split}
    k_{V,g_n,\meas_n,N}(x) &:= \inf\!\Big\lbrace \frac{\Ric_{g_n,\meas_n,N}(v,v)}{g_n(v,v)} : v \in V\cap T_x\mms\Big\rbrace
    \end{split}
\end{align}
is continuous. In particular, for every $n\in\N$ with $n\geq n_0$, every $x\in C$, and every  $v\in V\cap T_x\mms$,  we obtain
\begin{align*}
    \Ric_{g_n,\meas_n,N}(v,v) \geq k_{V,g_n,\meas_n,N}(x)\,g_n(v,v).
\end{align*}
With an analogous proof as Calisti et al.~\cite{calisti-graf-hafemann-kunzinger-steinbauer2025+}*{Lem.~5.9} (where we choose the sequence of vector fields therein to  minimize  \eqref{Eq:kVgm}), the following statement is then readily established. It is easily seen to hold in the weighted case by applying the Friedrichs-type lemma \cite{calisti-graf-hafemann-kunzinger-steinbauer2025+}*{Lem.~3.3} to the weighted quantities from \eqref{Eq:NormFormulas} and noting  the densities of the involved reference measures converge locally uniformly. 

\begin{lemma}[$L^p$-convergence of curvature bounds]\label{Le:Lpcvg} Suppose there are $K\in\R$ and $N\in[\dim\mms,\infty)$ with $\smash{\Ric_{g,\meas,N} \geq K}$ timelike distributionally. Let $\smash{V\subset T^+\big\vert_C}$ be compact and uniformly timelike, with  $C\subset\mms$  compact. Then for every $p\in [1,\infty)$,
\begin{align*}
    \lim_{n\to\infty} \int_C \big\vert (k_{V,g_n,\meas_n,N}-K)_-\big\vert^p\d\meas_n =0.
\end{align*}
\end{lemma}

\subsubsection{Construction of compact and uniformly timelike sets} Finally, we identify sufficient conditions under which a set of tangent vectors as above can actually be constructed. A variant of the following result implicitly appears in the proof of Calisti et al.'s Hawking incompleteness theorem for Lipschitz spacetimes \cite{calisti-graf-hafemann-kunzinger-steinbauer2025+}*{Thm. 5.1} (for proper time instead of affinely parametrized maximizing geo\-de\-sics). Recall $r$ is the fixed background Riemannian metric on $\mms$.

\begin{lemma}[Bounds on Lorentzian and Riemannian lengths]\label{Le:Bounds} Let $X,Y\subset \mms$ be compact sets with $\smash{X \times Y \subset I_g}$.  We define $\smash{C:= J_g(X,Y)}$. Let $(g_n)_{n\in\N}$ be a good approximation of $g$ on $C$ according to \cref{Def:Goodapprox}. Then there exist constants $c>0$, $\lambda>0$, and $n_0\in\N$ with the following property.  For every $n\in\N$ with $n\geq n_0$, we have $\smash{X\times Y\subset I_{g_n}}$ and each tangent vector of each timelike affinely parametrized $g_n$-maximizing geodesic $\gamma\colon[0,1]\to\mms$ that starts in $X$ and terminates in $Y$ belongs to the  set
\begin{align*}
    V := \{v\in T^+\big\vert_C : g_n(v,v)\geq \lambda^2 \textnormal{ \textit{for every} }n\in\N \textnormal{ \textit{with} }n\geq n_0,\,\vert v\vert_r \leq c\},
\end{align*}
which is compact and eventually uniformly timelike.
\end{lemma}

\begin{proof}  It is clear that $V$ is compact as a closed subset of a compact subset of  $T\mms$. Moreover, it is eventually uniformly timelike by definition.

We start with some general considerations. Given \emph{any} $n_0\in\N$,  abbreviate the class of curves in question by $G_{\geq n_0}$. By nontotal imprisonment of $(\mms,g)$, implied by global hyperbolicity, there is  a constant $L>0$ such that the $r$-length of every $g$-causal curve is bounded from above by $L$. In particular, since every $\smash{\gamma\in G_{n_0}}$ is $g$-timelike, we obtain $\Len_r(\gamma)\leq L$. Since $\gamma$ solves the geodesic equation  with respect to $g_n$ for some $n\in\N$ with $n\geq n_0$, the results from Lange--Lytchak--Sämann \cite{lange-lytchak-samann2021} recalled in \cref{Sub:LipSptGeo} imply the $C^{1,1}$-norm of $\gamma$ on the compact set $C$ with respect to $r$ is bounded in terms of the $C^{0,1}$-norm of $g_n$ on $C$ and $\Len_r(\gamma)$. 
While the latter is bounded from above by $L$ as already inferred, the former is bounded in terms of the $C^{0,1}$-norm of $g$ on $C$ (which is finite by assumption on $g$) by \cref{Le:ApproxLemma} and Calisti et al.~\cite{calisti-graf-hafemann-kunzinger-steinbauer2025+}*{Lem.~3.2}. In particular, this implies the existence of a constant $c > 0$ such that for every $\gamma\in G_{\geq n_0}$, we have $\vert\dot\gamma\vert_r\leq c$ on $[0,1]$.

Our hypothesis on $X$ and $Y$, global hyperbolicity of $(\mms,g)$, and continuity of $l_+$ imply the existence of $\lambda > 0$ such that 
\begin{align*}
\inf\{l_g(x,y) :x\in X,\,y\in Y\} \geq 3\lambda.    
\end{align*}
By \cref{Le:UnifCvg} and compactness of $X\times Y$, there exists a number $n_0\in \N_0$ such that for every $n\in\N$ with $n\geq n_0$,
\begin{align}\label{Eq:sakjbaf}
    \inf\{{l_{g_n}}(x,y): x\in X,\,y\in Y\}  \geq 2\lambda.
\end{align}
On the other hand, by \cref{Le:ApproxLemma} again, up to increasing $n_0$, every $n,m\in\N$ with $n,m\geq n_0$ satisfies
\begin{align}\label{Eq:gngm}
    \big\Vert g_n-g_m\big\Vert_{\infty,C} \leq\frac{3\lambda^2}{c^2}.
\end{align}
We will prove the claim of the lemma for the above choices of $c$, $\lambda$, and $n_0$.

Let $\smash{\gamma\in G_{\geq n_0}}$ be a timelike affinely parametrized $g_n$-maximizing geodesic, where $n\in\N$ with $n\geq n_0$, and let $t\in[0,1]$. Recall $\smash{\vert\dot\gamma_t\vert_r\leq c}$. Affine parametrization with respect to $g_n$, \eqref{Eq:sakjbaf}, and  \eqref{Eq:gngm} yield, for every $m\in\N$ with $m\geq n_0$,
\begin{align*}
    g_m(\dot\gamma_t,\dot\gamma_t) &= g_m(\dot\gamma_t,\dot\gamma_t)- g_n(\dot\gamma_t,\dot\gamma_t)+g_n(\dot\gamma_t,\dot\gamma_t)\\
    &\geq -\big\Vert g_n-g_m\big\Vert_{\infty,C}\,\big\vert\dot\gamma_t\big\vert_r^2 + 4\lambda^2\\
    &\geq \lambda^2.
\end{align*}
This is the desired statement.
\end{proof}

With a similar proof, the following is verified.

\begin{lemma}[Uniform precompactness]\label{Le:UnifPrec} We retain the hypotheses and the notation from \cref{Le:Bounds}. Then the set of all timelike affinely parametrized $g$-maximizing geodesics $\gamma\colon[0,1]\to\mms$ that start in $X$ and end in $Y$ is equi-Lipschitz continuous with respect to the uniform topology on $\Cont^0([0,1];C)$ induced by $r$.
\end{lemma}

\subsection{Regularity of time separation function} We now use the previous considerations to establish some  key properties of $l$ on our globally hyperbolic weighted Lipschitz spacetime $(\mms,g,\meas)$. 

\begin{proposition}[Local equi-Lipschitz continuity]\label{Le:EquiLip} Let $(g_n)_{n\in\N}$ be a good approximation of $g$ as in \cref{Def:Goodapprox}. Then there is $n_0\in\N$ such that the class $\smash{\{l_{g_n}: n\in\N \textnormal{ \textit{with} }n\geq n_0\}}$ is locally equi-Lipschitz continous on $\smash{I_g}$; that is, for every $x,y\in\mms$ with $\smash{x\ll_gy}$, there exist open neighborhoods $\smash{U_x,U_y\subset \mms}$ of $x$ and $y$, respectively, and $c_0>0$ such that for every $\smash{x',x''\in U_x}$ and every $\smash{y',y''\in U_y}$,
\begin{align*}
    \sup\{\big\vert l_{g_n}(x',y') -l_{g_n}(x'',y'')\big\vert : n\in\N\textnormal{ \textit{with} }n\geq n_0 \}  \leq c_0\,\big[\met_r(x',x'')+\met_r(y',y'')\big].
\end{align*}

In particular, $\smash{l_g}$ is locally Lipschitz continuous on $\smash{I_g}$.
\end{proposition}

The proof of this proposition is partly based on well-known properties of time separation functions in the smooth setting; we recapitulate the necessary material in \cref{Sub:Timelike cut locus} below, to where we thus defer the proof of \cref{Le:EquiLip}.

Define the Lorentz distance function $\smash{l_o\colon I^+(o)\cup\{o\}\to\R_+}$ from $o$ by
\begin{align}\label{Eq:LorDistFct}
    l_o(x):= l(o,x).
\end{align}

\begin{corollary}[Eikonal equation]\label{Cor:Unit slope} The gradient $\nabla l_o$ of $l_o$, wherever it exists on $\smash{I^+(o)}$, is past-directed and satisfies
\begin{align*}
    g(\nabla l_o,\nabla l_o) = 1 \quad\meas\mres I^+(o)\textnormal{\textit{-a.e.}}
\end{align*}
\end{corollary}

\begin{proof} By \cref{Le:EquiLip} and Rademacher's theorem, $l_o$ is differentiable $\meas$-a.e. on $\smash{I^+(o)}$, thus the statement makes sense. 

We first claim $g(\nabla l_o,\nabla l_o)\geq 1$ at every differentiability point $\smash{x\in I^+(o)}$ of $l_o$. Let $\smash{v\in T_x^+}$. Moreover, let $\smash{^r\!\exp}$ denote the exponential map induced by the background Riemannian metric $r$. Since $g$ is continuous, for every sufficiently small $s>0$ the curve $\smash{\gamma^s\colon[0,1]\to\mms}$ defined through $\smash{\smash{\gamma_t^s:= {}^r\!\exp_x(tsv)}}$ is timelike; in particular, $\smash{^r\!\exp_x(sv)\in I^+(x)}$. By the reverse triangle inequality and the definition of $l$,
\begin{align*}
\frac{l_o({}^r\!\exp_x(sv))- l_o(x)}{s}   \geq \frac{l(x,{}^r\!\exp_x(sv))}{s} \geq \frac{1}{s}\int_{[0,1]} \sqrt{g(\dot\gamma_t^s,\dot\gamma_t^s)}\d t.
\end{align*}
By local Lipschitz continuity of $g$ and a Taylor expansion, uniformly in $t\in[0,1]$,
\begin{align*}
    \sqrt{g(\dot\gamma_t^s,\dot\gamma_t^s)} = s\sqrt{g(v,v)} + \rmo(s)\quad \textnormal{as }s\to 0^+.
\end{align*}
Consequently,
\begin{align}\label{Eq:COMBined}
    \rmd l_o(v) = \lim_{s\to 0^+}\frac{l_o({}^r\!\exp_x(sv))- l_o(x)}{s} \geq \sqrt{g(v,v)}.
\end{align}
By the arbitrariness of $v$, this gives $\smash{\rmd l_o(v) \geq 0}$ for every causal vector $\smash{v\in T_x\mms}$. In other words, $\rmd l_o$ belongs to the dual cone of the set of causal vectors in $T_x\mms$. This cone is precisely the image of the set of past-directed causal vectors in $T_x\mms$ under the musical isomorphism $\flat$. This implies $\nabla l_o$ is a past-directed  causal vector in $T_x\mms$. In turn, the well-known duality formula
\begin{align}\label{Eq:DualityFormula}
    \sqrt{g(\nabla l_o,\nabla l_o)} = \inf\{\rmd l_o(v) : v\in T_x^+ \textnormal{ with }g(v,v)\geq 1\}
\end{align}
combined with \eqref{Eq:COMBined} give the desired estimate.

It remains to show $g(\nabla l_o,\nabla l_o)\leq  1$ at the above  point $x$. Let $\gamma\colon[0,l_o(x)]\to\mms$ be a timelike proper time parametrized maximizing geodesic from $o$ to $x$.  By the results from Lange--Lytchak--Sämann \cite{lange-lytchak-samann2021} recalled in \cref{Sub:LipSptGeo}, $\gamma$ is a $C^{1,1}$-curve. By proper time parametrization, we have $\smash{g(\dot\gamma_t,\dot\gamma_t)=1}$ for every $t\in[0,l_o(x)]$. Thus, using the duality formula \eqref{Eq:DualityFormula} and proper time parametrization, 
\begin{align*}
    \sqrt{g(\nabla l_o,\nabla l_o)}\leq \rmd l_o(\dot{\gamma}_{l_o(x)}) = \lim_{s\to 0^+} \frac{l_o(\gamma_{l_o(x)}) -l_o(\gamma_{l_o(x)-s})}{s} = 1,
\end{align*}
which is the desired inequality.
\end{proof}

\section{Synthetic timelike Ricci curvature lower bounds}\label{Sec:SynTLRic}

\subsection{One-dimensional considerations} 

\subsubsection{Distortion coefficients} For $L>0$, let $\kappa\colon[0,L]\to\R$ be a continuous function. The induced \emph{generalized sine function} $\sin_\kappa\colon [0,L]\to\R$ is defined as the unique solution to the ODE $u'' +\kappa\,u=0$ with $u(0)=0$ and $u'(0)=1$. For instance, if $\kappa$ is constant, then for every $\theta\in[0,L]$,
\begin{align}\label{Eq:sinkappa}
    \sin_\kappa(\theta) = \begin{cases}
        \displaystyle\frac{\sin(\sqrt{\kappa}\,\theta)}{\sqrt{\kappa}} & \textnormal{if }\kappa>0,\\
        \theta & \textnormal{if }\kappa=0,\\
        \displaystyle\frac{\sinh(\sqrt{-\kappa}\,\theta)}{\sqrt{-\kappa}} & \textnormal{otherwise}.
    \end{cases}
\end{align}
Let $\pi_\kappa\in (0,\infty]$  denote the first positive zero of $\sin_\kappa$ (with the convention $\pi_\kappa :=\infty$ if $\sin_\kappa$ vanishes only at zero). For instance,  if $\kappa$ is constant,
\begin{align}\label{Eq:pikappa}
    \pi_\kappa = \begin{cases}
        \displaystyle\frac{\pi}{\sqrt{\kappa}} & \textnormal{if }\kappa>0,\\
        \infty & \textnormal{otherwise}.
    \end{cases}
\end{align}

\begin{definition}[$\sigma$-Distortion coefficients] Given $t\in [0,1]$ and $\theta\in[0,L]$, we set
    \begin{align*}
        \sigma_\kappa^{(t)}(\theta) := \begin{cases}
                    t & \textnormal{\textit{if} } \theta =0,\\
            \displaystyle \frac{\sin_\kappa(t\,\theta)}{\sin_\kappa(\theta)} & \textnormal{\textit{if} }\theta \in(0,\pi_\kappa),\\
            \infty & \textnormal{\textit{otherwise}}.
        \end{cases}
    \end{align*}
\end{definition}

For instance, if $\kappa$ is zero, then $\sigma_\kappa^{(t)}(\theta)=t$ for every $t\in [0,1]$ and every $\theta\in[0,L]$. 
Note that if $\smash{\sigma_\kappa^{(t)}(\theta) <\infty}$ for $\theta\in [0,L]$ and some (hence every) $t\in[0,1]$, then the function $v\colon[0,1]\to\R_+$ defined by $\smash{v(t):= \sigma_\kappa^{(t)}(\theta)}$ solves $v''(t) + \kappa(t\,\theta)\,\theta^2\,v(t)=0$ for every $t\in [0,1]$ as well as $v(0)=0$ and $v(1)=1$.

By taking a geometric mean, we now recapitulate another family of distortion coefficients. To this aim, we fix $N\in(1,\infty)$. We stipulate the infinity conventions $0\cdot\infty := 0$ and $\alpha\cdot\infty:=\infty^\alpha:=\infty$ for every $\alpha >0$.

\begin{definition}[$\tau$-Distortion coefficients]\label{Def:tauDistCoeff} Given $t\in [0,1]$ and $\theta\in [0,L]$, we set
\begin{align*}
    \tau_{\kappa,N}^{(t)}(\theta) := t^{1/N}\,\sigma_{\kappa/(N-1)}^{(t)}(\theta)^{1-1/N}.
\end{align*}
\end{definition}

\begin{remark}[Basic properties]\label{Rem: props coefs} For every $t\in[0,1]$ and every $\theta\in [0,L]$, the following statements hold.
\begin{itemize}
    \item If $\kappa'\colon[0,L]\to\mms$ satisfies $\kappa'\leq\kappa$ on $[0,L]$, then
    \begin{align*}
        \sigma_{\kappa'/N}^{(t)}(\theta)&\leq \sigma_{\kappa/N}^{(t)}(\theta),\\
                \tau_{\kappa'/N}^{(t)}(\theta)&\leq \tau_{\kappa/N}^{(t)}(\theta).
    \end{align*}
    \item We have the relation
    \begin{align*}
        \tau_{\kappa,N}^{(t)}(\theta)  \geq \sigma_{\kappa/N}^{(t)}(\theta).
    \end{align*}
    \item If $(\kappa_n)_{n\in\N}$ is a sequence of continuous functions $\kappa_n\colon[0,L]\to\R$ such that $\kappa\leq\liminf_{n\to\infty}\kappa_n$ pointwise on $[0,1]$, then
    \begin{align*}
        \sigma_{\kappa/N}^{(t)}(\theta) &\leq\liminf_{n\to\infty}\sigma_{\kappa_n/N}^{(t)}(\theta),\\
        \tau_{\kappa/N}^{(t)}(\theta) &\leq\liminf_{n\to\infty}\tau_{\kappa_n/N}^{(t)}(\theta).
    \end{align*}
\end{itemize}
For details, we refer to Ketterer \cite{ketterer2017}*{§3}.
\end{remark}

\subsubsection{A defect estimate} For later purposes, we will need to quantify if a function is ``close'' to a number $K\in\R$ in an integral sense, then the corresponding distortion coefficients are ``close'' as well.

For the estimate from \cref{Le:DefDist}, given $K\in\R$, $N\in [2,\infty)$, $p\in (N/2,\infty)$, and $\eta \in (0,\pi_{K/(N-1)}/2)$, we define two constants
\begin{align}\label{Eq:TwoConst}
\begin{split}
\Omega_{K,N,p,\eta} &:= \begin{cases} C_{N,p} &\textnormal{if }K\leq 0,\\
\max\{C_{N,p}, \sin_{K/(N-1)}(\eta)^{N+1-4p}\} & \textnormal{otherwise},
\end{cases}\\
\Lambda_{K,N,\eta} &:=   \begin{cases} D_{K,N,\eta}  & \textnormal{if }K> 0,\\
1 & \textnormal{otherwise},
\end{cases}
\end{split}
\end{align}
where
\begin{align*}
C_{N,p} &:= (2p-1)^p\,\Big[\frac{N-1}{2p-N}\Big]^{p-1},\\
D_{K,N,\eta} &:= 1+ \max\!\Big\{\!\sin_{K/(N-1)}(r)^{1-N} : r\in \Big[\frac{\pi_{K/(N-1)}}{2}, \pi_{K/(N-1)}-\eta\Big]\Big\}.
\end{align*}

The following estimate is established by Ketterer \cite{ketterer2021-stability}*{Cor.~4.7}.

\begin{lemma}[Defect of distortion coefficients]\label{Le:DefDist} Let us fix $K\in\R$, $N\in [2,\infty)$, $p\in (N/2,\infty)$, and $\eta\in(0,\pi_{K/(N-1)}/2)$. Moreover, let $\kappa\colon [0,L]\to\R$  be continuous, where $L>0$. Then for  every $t\in (0,1)$ and every $\theta\in (0,\min\{L,\pi_{K/(N-1)}-\eta\})$,
\begin{align*}
&\tau_{K,N}^{(t)}(\theta)- \tau_{\kappa,N}^{(t)}(\theta) \\ &\qquad\qquad\leq \big[\Lambda_{K,N,\eta}\,\Omega_{K,N,p,\eta}^{1/(2p-1)}\big]^{1/N}\,\Big[\!\int_{[t,1]} \tau_{\kappa,N}^{(r)}(\theta)^N\d r\Big]^{2(p-1)/N(2p-1)}\\
&\qquad\qquad\qquad\qquad\times\Big[t\,\theta^{2p}\int_{[0,1]} (\kappa(r\theta) - K)_-\,\sigma_{\kappa/(N-1)}^{(r)}(\theta)^{N-1}\d r\Big]^{1/N(2p-1)}. 
\end{align*}
\end{lemma}

\subsubsection{CD densities and diameter estimates}\label{Sub:CDdensities} Let $I\subset\R$ be an interval, $\kappa\colon I\to\R$ be continuous, and $N\in (1,\infty)$. Let $\gamma\colon [0,1]\to I$ be a straight line; in particular, it has constant speed $\vert\dot\gamma\vert$. We define the function $\smash{\kappa^+_\gamma\colon [0,\vert\dot\gamma\vert]\to\R}$  by $\kappa^+_\gamma:= \kappa\circ\bar{\gamma}$, where $\bar\gamma\colon[0,\vert\dot\gamma\vert]\to I$ designates the unit speed reparametrization of $\gamma$. In addition, we define $\smash{\kappa_\gamma^-\colon [0,\vert\dot\gamma\vert]\to I}$ by $\smash{\kappa_\gamma^-:=\kappa\circ\bar\gamma\circ T}$, where $T\colon [0,\vert\dot\gamma\vert]\to [0,\vert\dot\gamma\vert]$ means the orientation reversal $T(r) := \vert\dot\gamma\vert-r$.

\begin{definition}[CD density]\label{Def:CDdensity} A function $h\colon I\to\R_+$ is called \emph{$\CD(\kappa,N)$ density} if for every straight line $\gamma\colon [0,1]\to I$ and every $t\in[0,1]$,
\begin{align*}
    h(\gamma_t)^{1/(N-1)} \geq \sigma_{\kappa_\gamma^-/(N-1)}^{(1-t)}(\vert\dot\gamma\vert)\,h(\gamma_0)^{1/(N-1)}+ \sigma_{\kappa_\gamma^+/(N-1)}^{(t)}(\vert\dot\gamma\vert)\,h(\gamma_1)^{1/(N-1)}.
\end{align*}
\end{definition}

\begin{remark}[Basic properties]\label{Re:BasPrCD} Let $h\colon I\to\R_+$ be a $\CD(\kappa,N)$ density. Then the following statements hold.
\begin{itemize}
    \item It is is locally semiconcave (hence locally Lipschitz continuous) on $I^\circ$. In particular, its right-derivative exists everywhere on $I\setminus\{\sup I\}$ and is right-continuous; analogously for the left-derivative on $I\setminus\{\inf I\}$. Moreover, $h$ can be continuously extended to the closure of $I$.
    \item Either $h$ vanishes identically on $I$ or it is positive on $I^\circ$.
    \item If $h$ is positive on $I^\circ$, it is twice differentiable $\Leb^1\mres I$-a.e., and at every twice differentiability point in $I^\circ$,
    \begin{align*}
        (\log h)'' +\frac{1}{N-1}\,\big\vert (\log h)'\big\vert^2 = (N-1)\,\frac{(h^{1/(N-1)})''}{h^{1/(N-1)}}\leq-\kappa.
    \end{align*}
\end{itemize}
For details, we refer to Cavalletti--Milman \cite{cavalletti-milman2021}*{§A}.
\end{remark}

The following property justifies the name ``CD density''.

\begin{proposition}[Curvature-dimension condition \cite{ketterer2017}*{Prop.~3.8}] Let $h\colon I\to\R$ be locally $\Leb^1\mres I$-integrable. Then the topologically one-dimensional metric measure space $\smash{(\cl\,I, \vert\cdot-\cdot\vert,h\,\Leb^1\mres I)}$ obeys the curvature-dimension condition $\CD(\kappa,N)$ of Ketterer \cite{ketterer2017} if and only if $h$ has an $\Leb^1\mres I$-version that is a $\CD(\kappa,N)$ density.
\end{proposition}

A family of facts we recapitulate now are diameter estimates.  For  Riemannian manifolds, the next \cref{Th:DiamCDdensities} was proven by Aubry \cite{aubry2007}*{Thm.~1.2}. The version we report below states the extension of Aubry's result to essentially nonbranching CD metric measure spaces by Caputo--Nobili--Rossi \cite{caputo-nobili-rossi2025+}*{Thm.~1.2} for CD densities. Their extension hypothesizes a ``non\-collapsing'' of the reference measure which, in the case of CD densities, forces the reference measure to vanish identically. However, this additional  hypothesis is only needed in the proof of \cite{caputo-nobili-rossi2025+}*{Prop.~4.3} to deal with nonsmooth densities. In our setting, all relevant densities will be smooth; in particular,  the noncollaps\-ing hypothesis can be dispensed, cf.~\cite{aubry2007}*{Lem.~3.1}.

\begin{theorem}[One-dimensional diameter estimate]\label{Th:DiamCDdensities} We let $K>0$,  $N\in[2,\infty)$, and $p\in (N/2,\infty)$. There exists a constant $C_{K,N,p}>1$ with the following property. Let $h\colon I\to\R$ be a smooth $\CD(\kappa,N)$ probability density on an interval $I\subset\R$, where $\kappa\colon I\to\R$ is continuous, with
\begin{align*}
\int_I \big\vert(\kappa-K)_-\big\vert^p\,h\d\Leb^1\leq \frac{1}{C_{K,N,p}}.
\end{align*}
Then we have
\begin{align*}
\diam\,(I,\vert\cdot-\cdot\vert)\leq\pi\sqrt{\frac{N-1}{K}}\,\Big[1+C_{K,N,p}\,\Big[\!\int_I \big\vert(\kappa-K)_-\big\vert^p\,h\d\Leb^1\Big]^{1/5}\Big].
\end{align*}
\end{theorem}

\begin{corollary}[Qualitative one-dimensional diameter estimate]\label{Cor:deltacor} We let $K>0$, $N\in[2,\infty)$, and $p\in (N/2,\infty)$. Then for every $\varepsilon>0$, there exists $\delta_{K,N,p,\varepsilon}>0$ with the following property. Let $h\colon I\to\R$ be a smooth $\CD(\kappa,N)$ probability density on an interval $I\subset\R$, where $\kappa\colon I\to\R$ is continuous, with
\begin{align*}
\int_I \big\vert(\kappa-K)_-\big\vert^p\,h\d\Leb^1\leq \delta_{K,N,p,\varepsilon}.
\end{align*}
Then we have
\begin{align*}
\diam\,(I,\vert\cdot-\cdot\vert)\leq\pi\sqrt{\frac{N-1}{K}} +\varepsilon.
\end{align*}
\end{corollary}

\begin{remark}[Explicit integral deficit]\label{Re:Explicit} In the previous corollary, one can  take
\begin{align*}
    \delta_{K,N,p,\varepsilon} = \min\!\Big\lbrace\Big[\frac{\varepsilon}{\pi \,C_{K,N,p}}\sqrt{\frac{K}{N-1}}\Big]^5, \frac{1}{C_{K,N,p}}\Big\rbrace.
\end{align*}
This is  clear from \cref{Th:DiamCDdensities}.
\end{remark}

\subsection{Lorentzian optimal transport}\label{Sub:LorOptTran} For background about the material recalled here, we refer to Eckstein--Miller \cite{eckstein-miller2017}, McCann \cite{mccann2020}, and Cavalletti--Mondino \cite{cavalletti-mondino2020}.

Let $\Prob(\mms)$ denote the space of all Borel probability measures on $\mms$. It is endowed with the narrow topology induced by convergence against bounded and continuous test functions. The subscript $\smash{_\sharp}$ will denote the usual push-forward operation.

Given $\mu,\nu\in\Prob(\mms)$, let $\Pi(\mu,\nu)$ denote the set of all couplings of $\mu$ and $\nu$. We will call a coupling $\pi\in\Pi(\mu,\nu)$ \emph{causal} (which is equivalent to $\supp\pi \subset J$ thanks to \cref{Th:ConsequGH}) if it is concentrated on the causal relation $J$ and \emph{chronological} if it is concentrated on the chronological relation $I$. Note $\Pi(\mu,\nu)$ is never empty, as it always contains the product measure $\mu\otimes \nu$; however, in general $\mu$ and $\nu$ may not admit any causal coupling.

\subsubsection{$q$-Lorentz--Wasserstein distance} Throughout the sequel, we fix $q\in (0,1)$. Define the \emph{$q$-Lorentz--Wasserstein distance} $\smash{\ell_q\colon\Prob(\mms)^2\to \R_+\cup\{-\infty,\infty\}}$,  introduced by Eckstein--Miller \cite{eckstein-miller2017}, through
\begin{align}\label{Eq:ellq}
    \ell_q(\mu,\nu) := \sup\!\Big\lbrace\Big[\!\int_{\mms^2} l(x,y)^q\d\pi(x,y)\Big]^{1/q} : \pi \in \Pi(\mu,\nu)\Big\rbrace.
\end{align}
where we adopt the convention $(-\infty)^{q} := (-\infty)^{1/q}:=-\infty$. It inherits the reverse triangle inequality from $l$. We call a coupling $\pi\in\Pi(\mu,\nu)$ \emph{$\ell_q$-optimal} if it is causal and it attains the supremum in \eqref{Eq:ellq}; in particular, we exclude the cost $-\infty$ but include the cost $\infty$ in our notion of $\ell_q$-optimality. By a standard  argument and the properties of $l$ from \cref{Th:ConsequGH}, if $\mu$ and $\nu$ have compact support and admit a causal coupling, they also admit an $\smash{\ell_q}$-optimal coupling (which has finite cost).

As we will only deal with transport in timelike directions, the following notion introduced by Cavalletti--Mondino \cite{cavalletti-mondino2020}*{Def. 2.18} will be convenient.

\begin{definition}[Timelike $q$-dualizability]\label{Def:TLqDual} A pair $(\mu,\nu)\in\Prob(\mms)^2$ is called \emph{timelike $q$-dualizable} if it admits a chronological $\smash{\ell_q}$-optimal coupling. 
\end{definition}

Note timelike $q$-dualizability of $\mu$ and $\nu$ in the sense of the previous definition is stronger than requiring $\smash{\ell_q(\mu,\nu)>0}$. On the other hand, it is clear by the above discussions  that every pair $(\mu,\nu)\in\Prob(\mms)^2$ of two compactly supported measures satisfying $\supp\mu\times\supp\nu\subset I$ is timelike $q$-dualizable.

\begin{remark}[Nice cost function]\label{Re:ContinuousCost} If $\mu,\nu\in\Prob(\mms)$ are compactly supported and satisfy $\supp\mu\times\supp\nu\subset I$, we clearly have
\begin{align*}
    \ell_q(\mu,\nu) =  \sup\!\Big\lbrace\Big[\!\int_{\mms^2} l_+(x,y)^q\d\pi(x,y)\Big]^{1/q} : \pi \in \Pi(\mu,\nu)\Big\rbrace
\end{align*}
In other words, in this case the $\smash{\ell_q}$-optimal transport problem from $\mu$ to $\nu$ can be translated into an optimal transport problem with a  bounded and continuous cost function, making standard tools more accessible than for possibly degenerate cost functions such as $l^q$; cf.~e.g.~the proof of \cref{Pr:VaryingStab} below.
\end{remark}

The following stability and compactness property shown by Cavalletti--Mondino \cite{cavalletti-mondino2020}*{Lem.~2.11} will be used later.

\begin{lemma}[Static stability of optimal couplings]\label{Le:StaticCpl} Assume $\smash{(\mu^i)}_{i\in\N}$ and $\smash{(\nu^i)_{i\in\N}}$ are sequences that converge narrowly to $\mu,\nu\in\Prob(\mms)$, respectively. Let $\smash{(\pi^i)_{i\in\N}}$ be a sequence of $\smash{\ell_q}$-optimal couplings $\smash{\pi^i}$ of $\smash{\mu^i}$ and $\smash{\nu^i}$. Then $\smash{\{\pi_i :i \in\N\}}$ is narrowly precompact in $\Prob(\mms^2)$ and every narrow limit point $\pi$ of $\smash{(\pi_i)_{i\in\N}}$ that is concentrated on the chronological relation $I$ is $\smash{\ell_q}$-optimal.
\end{lemma}

\subsubsection{$q$-Geodesics and $\ell_q$-optimal dynamical couplings}\label{Sub:GeosLift} We go on with the following notion  introduced by McCann \cite{mccann2020}*{Def.~1.1}.

\begin{definition}[Geodesics of probability measures]\label{Def:qgeo} A curve $\mu\colon[0,1]\to\Prob(\mms)$ is called \emph{$q$-geodesic} if $\smash{\ell_q(\mu_0,\mu_1)\in(0,\infty)}$ and for every $s,t\in[0,1]$ with $s<t$,
\begin{align*}
    \ell_q(\mu_s,\mu_t) = (t-s)\,\ell_q(\mu_0,\mu_1).
\end{align*}
\end{definition}

We will later need the following stability and compactness property. It is proven by 
Braun--McCann \cite{braun-mccann2023}*{Thm.~2.69, Rem.~2.70}; note that Hypothesis 2.42 therein is satisfied by \cite{braun-mccann2023}*{Ex.~2.44} since we work on a Lipschitz spacetime.

\begin{lemma}[Static stability and compactness of geodesics]\label{Le:StaticGeo}  Let $\smash{(\mu^i)_{i\in\N}}$ be a given  sequence of $q$-geodesics. For a compact set $C\subset\mms$, assume $\smash{\mu^i_t}$ is supported in $C$ for every $i\in\N$ and every  $t\in[0,1]$. Let $\mu_0,\mu_1\in\Prob(\mms)$ be narrow limit points of $\smash{(\mu^i_0)_{i\in\N}}$ and $\smash{(\mu^i_1)_{i\in\N}}$, respectively, and assume
\begin{align*}
    0 < \ell_q(\mu_0,\mu_1) \leq\liminf_{n\to\infty}\ell_q(\mu_0^i,\mu_1^i).
\end{align*}
Then there is a $q$-geodesic $\mu\colon[0,1]\to\mms$ from $\mu_0$ to $\mu_1$ such that up to a non\-relabeled subsequence, $\smash{(\mu_t^i)_{i\in\N}}$ converges narrowly to $\mu_t$ for every $t\in[0,1]\cap\Q$.

Moreover, if every $\smash{\ell_q}$-optimal coupling of $\mu_0$ and $\mu_1$ is chronological, then any $q$-geodesic $\mu$ as in the previous statement is narrowly continuous. 
\end{lemma}

A tightly linked concept are optimal dynamical couplings, which we  review now. Given $t\in[0,1]$, let $\eval_t\colon\Cont^0([0,1];\mms)\to\mms$ be  the evaluation map $\eval_t(\gamma):=\gamma_t$. 

\begin{definition}[Optimal dynamical couplings] Given $\smash{\mu_0,\mu_1\in\Prob(\mms)}$, we call $\bdpi\in\Prob(\Cont^0([0,1];\mms))$ an \emph{$\smash{\ell_q}$-optimal dynamical coupling} of $\mu_0$ and $\mu_1$ if 
\begin{enumerate}[label=\textnormal{\alph*.}]
    \item it is concentrated on timelike affinely parametrized maximizing geodesics and
    \item $(\eval_0,\eval_1)_\push\bdpi$ is an $\smash{\ell_q}$-optimal coupling of $\mu_0$ and $\mu_1$.
\end{enumerate}
\end{definition}

\begin{remark}[Geodesy from optimal dynamical couplings] Given an $\smash{\ell_q}$-optimal dynamical coupling $\bdpi$ of $\mu_0,\mu_1\in\Prob(\mms)$,  it is not hard to check (cf.~e.g.~Cavalletti--Mondino \cite{cavalletti-mondino2020}*{Rem.~2.32}) that if $\smash{\ell_q(\mu_0,\mu_1)<\infty}$, the curve $\mu\colon[0,1]\to\Prob(\mms)$ given by $\mu_t := (\eval_t)_\push\bdpi$ is a $q$-geodesic.
\end{remark}

Conversely, there are recent ``lifting theorems''  stating general conditions under which a $q$-geodesic is in fact represented by an $\smash{\ell_q}$-optimal dynamical coupling in the sense of the previous remark. We will use the following result by Braun--McCann \cite{braun-mccann2023}*{Prop.~2.67}, but also refer to  Beran et al.~\cite{beran-braun-calisti-gigli-mccann-ohanyan-rott-samann+-}*{Thm.~2.43, Cor.~2.46}. 

\begin{lemma}[Lifting theorem]\label{Le:Lifting} Assume $\mu_0,\mu_1\in\Prob(\mms)$ are compactly supported and satisfy $\supp\mu_0\times\supp\mu_1\subset I$. Let $\mu\colon[0,1]\to\mms$ be a $q$-geodesic connecting $\mu_0$ and $\mu_1$. Then $\mu$ is represented by an $\smash{\ell_q}$-optimal dynamical coupling $\bdpi$ of $\mu_0$ and $\mu_1$, i.e.~for every $t\in[0,1]$, we have
\begin{align*}
    \mu_t=(\eval_t)_\push\bdpi.
\end{align*}
\end{lemma}

In the smooth framework, the following stronger characterization by McCann \cite{mccann2020}*{Lem.~4.4, Thm.~5.8, Cor. 5.9} holds. Finer structural properties, in particular with regards to the Monge--Ampère and Raychaudhuri equations, are summarized in \cref{Sub:WithDefect} below.

\begin{theorem}[Existence and uniqueness]\label{Th:UniqSmooth} Let $(\mms,g,\meas)$ be a globally hyperbolic weighted spacetime. Let $\mu_0,\mu_1\in\Prob(\mms)$ be compactly supported and $\meas$-absolutely continuous with $\supp\mu_0\times\supp\mu_1\subset I$. Then there exist an open neighborhood $U\subset\mms$ of $\supp\mu_0$ and a locally semiconvex \textnormal{(}hence locally Lipschitz continuous\textnormal{)} function $\varphi\colon U\to\R$ with $\meas$-a.e.~timelike gradient $\nabla\varphi$ such that for  the $\meas$-a.e. \textnormal{(}hence $\mu_0$-a.e.\textnormal{)} well-defined map $\Psi\colon[0,1]\times U\to\mms$ given by
\begin{align*}
    \Psi_t(x)  := \exp_x(-t\,\vert\nabla\varphi\vert^{q'-2}\,\nabla\varphi),
\end{align*}
where $q'<0$ is the conjugate exponent of $q$, we have the following properties.
\begin{enumerate}[label=\textnormal{(\roman*)}]
    \item The function $\varphi$ is a Kantorovich potential for $\mu_0$ and $\mu_1$.
    \item The measure $\smash{(\Id,\Psi_1)_\push\mu_0}$     is the unique \textnormal{(}necessarily chronological\textnormal{)} $\smash{\ell_q}$-optimal coupling of $\mu_0$ and $\mu_1$; in particular, $\mu_1 = (\Psi_1)_\push\mu_0$.
    \item The curve $\mu\colon[0,1]\to\mms$ defined by  $\smash{\mu_t := (\Psi_t)_\push\mu_0}$ 
    consists only of compactly supported, $\meas$-absolutely continuous measures. Furthermore, it defines the unique $q$-geodesic from $\mu_0$ to $\mu_1$.
    \item Define $F\colon U\to\Cont^0([0,1];\mms)$ by $F(x)_t := \Psi_t(x)$. Then the measure $F_\push\mu_0$ is the unique $\smash{\ell_q}$-optimal dynamical coupling of $\mu_0$ and $\mu_1$.
\end{enumerate}
\end{theorem}

We refer to McCann {\cite[§4]{mccann2020}}  and Cavalletti--Mondino {\cite[§2.4]{cavalletti-mondino2020}} for the notion of Kantorovich potentials and their duality theory in the Lorentzian framework.

\begin{remark}[Invertibility]\label{Re:Invert} In fact, since $\mu_0$ and $\mu_1$ are assumed to be $\meas$-absolutely continuous, the results of McCann \cite{mccann2020} quoted above also show for every $t\in[0,1]$, the transport map $\Psi_t$ is invertible on $\supp\mu_0$ and its inverse is the restriction of a locally semiconvex (hence locally Lipschitz continuous) nonrelabeled map $\smash{\Psi_t^{-1}}$ defined on an open neighborhood of $\supp\mu_t$.
\end{remark}

In addition to \cref{Le:StaticCpl,Le:StaticGeo}, we will need stability and compactness properties along the good approximations from \cref{Def:Goodapprox}. Similar \emph{static} facts are established by Braun \cite{braun2023-renyi}*{Prop.~B.11} and Braun--McCann \cite{braun-mccann2023}*{Prop.~2.67}.

\begin{proposition}[Varying stability and compactness of optimal dynamical couplings]\label{Pr:VaryingStab} Let $(g_n)_{n\in\N}$ be a good approximation of $g$ according to \cref{Def:Goodapprox}. Let $X,Y\subset\mms$ be compact sets such that $\smash{X\times Y\subset I_g}$. Let $\smash{(\mu^n_0)_{n\in\N}}$ and $\smash{(\mu^n_1)_{n\in\N}}$ form sequences supported in $X$ and $Y$, respectively. Finally, let $(\bdpi^n)_{n\in\N}$ constitute a sequence such that $\smash{\bdpi^n}$ is an $\smash{\ell_{g_n,q}}$-optimal dynamical coupling of $\smash{\mu_0^n}$ and $\smash{\mu_1^n}$. Then $\smash{\{\bdpi^n: n\in\N\}}$ is narrowly precompact in $\Prob(\Cont^0([0,1];\mms))$ and every narrow limit $\bdpi$ of $\smash{(\bdpi^n)_{n\in\N}}$ is an $\smash{\ell_{g,q}}$-optimal dynamical coupling of $\mu_0$ and $\mu_1$; in particular,
\begin{align}\label{Eq:Statem}
    \lim_{n\to\infty} \ell_{g_n,q}(\mu_0^n,\mu_1^n) = \ell_{g,q}(\mu_0,\mu_1).
\end{align}
\end{proposition}

\begin{proof} By definition, $\smash{\bdpi^n}$ is concentrated on timelike affinely parametrized maximizing $g_n$-geodesics for every $n\in\N$. These belong to an eventually $n$-independent compact subset of $\Cont^0([0,1];\mms)$ by \cref{Le:Bounds} and the Arzelà--Ascoli theorem. By Prokhorov's theorem, this implies narrow compactness of $\smash{\{\bdpi^n:n\in\N\}}$.

Now let $\bdpi$ be any narrow limit point of $\smash{(\bdpi^n)_{n\in\N}}$. Since the involved evaluation maps are continuous, $(\eval_0,\eval_1)_\push\bdpi$ is a (necessarily chronological) coupling of $\mu_0$ and $\mu_1$. We claim it is $\smash{\ell_{g,q}}$-optimal. Since $\smash{X\times Y\subset I_g}$ with compact inclusion and $\smash{l_{g,+}}$ is continuous by \cref{Th:ConsequGH}, \cref{Le:UnifCvg} implies there exists $n_0\in\N$ such that $\smash{X\times Y\subset I_{g_n}}$ for every $n\in\N$ with $n\geq n_0$. Combining this with \cref{Re:ContinuousCost}, \cref{Le:UnifCvg} again, and uniform boundedness of the involved cost functions, we can apply well-known stability properties for optimal transports with uniformly convergent continuous cost functions, cf.~e.g.~Villani \cite{villani2009}*{Thm.~5.20}, to infer that $\smash{(\eval_0,\eval_1)_\push\bdpi}$ is an $\smash{\ell_{g,q}}$-optimal coupling of $\mu_0$ and $\mu_1$, as desired.

Next, we show \eqref{Eq:Statem}. As $l_{g,+}$ is continuous, hence bounded on $X\times Y$,  Villani \cite{villani2009}*{Lem.~4.3} and the narrow convergence of $(\bdpi^n)_{n\in\N}$ to $\bdpi$ imply
\begin{align*}
    \lim_{n\to\infty} \int l_{g,+}(\gamma_0,\gamma_1)^q\d\bdpi^n(\gamma)=\int l_{g,+}(\gamma_0,\gamma_1)^q\d\bdpi(\gamma).
\end{align*}
On the other hand, \cref{Le:UnifCvg} implies that given $\varepsilon>0$ there exists $n_1\in\N$ such that for every $n\in\N$ with $n\geq n_1$, we have
\begin{align*}
    \sup\{\big\vert l_{g_n}(x,y)^q - l_g(x,y)^q\big\vert : x\in X,\,y\in Y\} \leq \varepsilon.
\end{align*}
This implies
\begin{align*}
    &\limsup_{n\to\infty} \big\vert \ell_{g_n,q}(\mu_0^n,\mu_1^n)^q - \ell_{g,q}(\mu_0,\mu_1)^q\big\vert\\
    &\qquad\qquad \leq \limsup_{n\to\infty} \int \big\vert l_{g_n}(\gamma_0,\gamma_1)^q- l_g(\gamma_0,\gamma_1)^q\big\vert\d\bdpi^n(\gamma)\\
    &\qquad\qquad\qquad\qquad + \limsup_{n\to\infty} \Big\vert\!\int l_{g,+}(\gamma_0,\gamma_1)^q\d\bdpi^n(\gamma) -  \int l_{g,+}(\gamma_0,\gamma_1)^q\d\bdpi(\gamma)\Big\vert\\
    &\qquad\qquad\leq \varepsilon.
\end{align*}
The arbitrariness of $\varepsilon$ yields the desired statement.

It remains to establish $\bdpi$ is concentrated on timelike affinely parametrized maximizing geodesics. By Braun--McCann \cite{braun-mccann2023}*{Cor.~2.65}, it suffices to show the curve $\mu\colon[0,1]\to\Prob(\mms)$ defined by $\mu_t:=(\eval_t)_\push\bdpi$ forms a $q$-geodesic (for the $q$-Lorentz--Wasserstein distance induced by $g$). We only sketch the argument.  Given $n\in\N$, define $\smash{\mu^n\colon[0,1]\to\Prob(\mms)}$ by $\smash{\mu_t^n:=(\eval_t)_\push\bdpi^n}$. Let $s,t\in [0,1]$ with $s<t$. We claim
\begin{align}\label{Eq:usclgn}
    \limsup_{n\to\infty}\ell_{g_n,q}(\mu_s^n,\mu_t^n)\leq\ell_{g,q}(\mu_s,\mu_t).
\end{align}
Indeed, given $n\in\N$, the reverse triangle inequality satisfied by $\smash{\ell_{g_n,q}}$ and $\smash{\ell_{g_n,q}}$-optimality of $\smash{(\eval_0,\eval_1)_\push\bdpi^n}$ imply that $\smash{(\eval_s,\eval_t)_\push\bdpi^n}$ is an $\smash{\ell_{g_n,q}}$-optimal coupling of its marginals.  Since its marginal sequences are tight by the first paragraph, the latter converges narrowly to a coupling $\pi_{s,t}$ of $\mu_s$ and $\mu_t$, up to a nonrelabeled subsequence. Since the property of admitting a causal coupling is closed in the narrow topology, cf.~Braun--McCann \cite{braun-mccann2023}*{Thm.~B.5}, $\pi_{s,t}$ is a causal coupling of $\mu_s$ and $\mu_t$; a priori, however, we do not know if it is $\smash{\ell_{g,q}}$-optimal. As in the previous paragraph, using basic semicontinuity properties of cost functionals, cf.~e.g.~Villani \cite{villani2009}*{Lem.~4.3}, it is not difficult to show the intermediate estimate in
\begin{align*}
    \limsup_{n\to\infty}\ell_{g_n,q}(\mu_s^n,\mu_t^n)^q=\limsup_{n\to\infty} \int l_{g_n}(\gamma_s,\gamma_t)^q\d\bdpi^n\leq \int_{\mms^2} l_g^q\d\pi_{s,t}\leq \ell_{g,q}(\mu_s,\mu_t)^q,
\end{align*}
as desired. Using the reverse triangle inequality for $\smash{\ell_{g,q}}$, \eqref{Eq:usclgn} thrice, and \eqref{Eq:Statem} gives
\begin{align*}
    \ell_{g,q}(\mu_0,\mu_1) &\geq \ell_{g,q}(\mu_0,\mu_s)+\ell_{g,q}(\mu_s,\mu_t) +\ell_{g,q}(\mu_t,\mu_1)\\
    &\geq \limsup_{n\to\infty} \big[\ell_{g_n,q}(\mu_0^n,\mu_s^n) + \ell_{g_n,q}(\mu_s^n,\mu_t^n)+\ell_{g_n,q}(\mu_t^n,\mu_1^n)\big]\\
    &= \big[s + (t-s) + (1-t)\big]\,\limsup_{n\to\infty}\ell_{g_n,q}(\mu_0^n,\mu_1^n)\\
    &= \ell_{g,q}(\mu_0,\mu_1).
\end{align*}
This forces equality in the simple consequence
\begin{align*}
    (t-s)\,\ell_{g,q}(\mu_0,\mu_1) =(t-s)\,\limsup_{n\to\infty} \ell_{g_n,q}(\mu_0^n,\mu_1^n) \leq \ell_{g,q}(\mu_s,\mu_t)
\end{align*}
of \eqref{Eq:usclgn}, which terminates the proof.
\end{proof}

\subsection{Variable timelike curvature-dimension condition}\label{Sub:VarTCD} Now we define the variable timelike curvature-dimension condition, briefly TCD condition, we will work with in this article. Inspired by the works of McCann \cite{mccann2020} and Mondino--Suhr \cite{mondino-suhr2022}, for constant timelike Ricci curvature bounds it was introduced on general metric measure spacetimes by Cavalletti--Mondino \cite{cavalletti-mondino2020}. A variable version of their TCD condition was given by Braun--McCann \cite{braun-mccann2023}.

The definitions we adopt in the next two subsections differ from \cite{braun-mccann2023} and follow the proposal of Braun \cite{braun2023-renyi} for constant timelike Ricci curvature bounds. For metric measure spaces, an analogous definition was given by Ketterer \cite{ketterer2017}. The advantage of this formulation is that as in \cite{braun2023-renyi}, it will yield the geometric inequalities for the globally hyperbolic weighted spacetime  $(\mms,g,\meas)$ we will obtain later in \emph{sharp} form a priori, without relying on nonbranching of timelike maximizing geodesics (which is unclear even if $g$ is $C^1$-regular).

Let $\meas$ be a Radon measure on $\mms$, such as the one from \cref{Def:WEighted}. Given $N\in (1,\infty)$, let us define the \emph{$N$-Rényi entropy} $\smash{\scrS_N(\,\cdot\mid\meas)\colon \Prob(\mms) \to \R_-\cup\{-\infty\}}$ with respect to $\meas$ by
\begin{align*}
    \scrS_N(\mu\mid\meas) := -\int_\mms \Big[\frac{\rmd\mu^\ac}{\rmd\meas}\Big]^{-1/N}\d\mu = -\int_\mms \Big[\frac{\rmd\mu^\ac}{\rmd\meas}\Big]^{1-1/N}\d\meas,
\end{align*}
where $\mu^\ac$ denotes the absolutely continuous part in the Lebesgue decomposition of its input $\mu$ with respect to $\meas$. If $\mu$ has compact support, Jensen's inequality gives
\begin{align*}
    \scrS_N(\mu\mid\meas) \geq -\meas[\supp \mu]^{1/N}.
\end{align*}

\begin{lemma}[Joint lower semicontinuity \cite{lott-villani2009}*{Thm.~B.33}]\label{Le:JointLSC} Given a compact set $C\subset\mms$, let $(\mu_n)_{n\in\N}$ and $(\meas_n)_{n\in\N}$ be two sequences in $\Prob(\mms)$ with support in $C$. Then all narrow limit points $\mu,\meas\in \Prob(\mms)$ of these sequences, respectively, satisfy
\begin{align*}
    \scrS_N(\mu\mid\meas) \leq \liminf_{n\to\infty}\scrS_N(\mu_n\mid\meas_n).
\end{align*}
\end{lemma}

Let $\gamma\colon[0,1]\to\mms$ be a  timelike affinely para\-metrized maximizing geodesic; in particular, it has constant speed $\vert\dot\gamma\vert := l(\gamma_0,\gamma_1)$. Let $\smash{\bar\gamma\colon[0,\vert\dot\gamma\vert]\to\mms}$ denote its proper time reparametrization. As in \cref{Sub:CDdensities}, given a continuous function $k\colon\mms\to\R$ we define $\smash{k_\gamma^\pm\colon[0,\vert\dot\gamma\vert]\to\R}$ by $\smash{k_\gamma^+:= k\circ \bar\gamma}$ and $\smash{k_\gamma^-:=k\circ\bar{\gamma}\circ T}$, where $T\colon[0,\vert\dot\gamma\vert]\to[0,\vert\dot\gamma\vert]$ means the function $T(r) := \vert\dot\gamma\vert-r$. Lastly,  $\smash{\tau_{k_\gamma^\pm,N}^{(t)}}$ means the $\tau$-distortion coefficients from \cref{Def:tauDistCoeff} induced by $\smash{k_\gamma^\pm}$, where $N\in(1,\infty)$ and $t\in[0,1]$.

\begin{definition}[Variable timelike curvature-dimension condition]\label{Def:VarTCD} For $q\in (0,1)$, a continuous function $k\colon\mms\to\R$, and $N\in(1,\infty)$, we say $(\mms,g,\meas)$ satisfies the \emph{timelike curvature-dimension condition}, briefly  $\smash{\TCD_q(k,N)}$, if for every timelike $q$-dualizable pair $(\mu_0,\mu_1)\in\Prob(\mms)^2$ of compactly supported, $\meas$-absolutely continuous measures, there exist
\begin{itemize}
    \item a $q$-geodesic $\mu\colon[0,1]\to\Prob(\mms)$ from $\mu_0$ to $\mu_1$ and
    \item an $\smash{\ell_q}$-optimal dynamical coupling $\smash{\bdpi}$ of $\mu_0$ and $\mu_1$
\end{itemize}
such that for every $N'\in[N,\infty)$ and every $t\in[0,1]$,
\begin{align*}
\scrS_{N'}(\mu_t\mid\meas)&\leq -\int \tau_{k_\gamma^-,N'}^{(1-t)}(\vert\dot\gamma\vert)\,\frac{\rmd\mu_0}{\rmd\meas}(\gamma_0)^{-1/N'}\d\bdpi(\gamma)\\
&\qquad\qquad -\int\tau_{k_\gamma^+,N'}^{(t)}(\vert\dot\gamma\vert)\,\frac{\rmd\mu_1}{\rmd\meas}(\gamma_1)^{-1/N'}\d\bdpi(\gamma).
\end{align*}
\end{definition}

The proof of the following result is performed as in Braun--Ohta \cite{braun-ohta2024}*{Thm.~5.9}. For more precise estimates, we refer to \cref{Sub:FromAntoSyn}  below. See also Braun \cite{braun2023-renyi}*{Thm.~3.35} for a similar pathwise version of Cavalletti--Mondino's TCD condition \cite{cavalletti-mondino2020}.

\begin{theorem}[Essential pathwise semiconcavity for TCD]\label{Th:EssPwConc} Assume $(\mms,g,\meas)$ is a globally hyperbolic weighted spacetime. Let $\mu_0,\mu_1\in\Prob(\mms)$ be two given compactly supported, $\meas$-absolutely continuous measures with $\supp\mu_0\times\supp\mu_1\subset I$. Let $\mu$ and $\bdpi$ be their unique  $q$-geodesic and $\smash{\ell_q}$-optimal dynamical coupling from   \cref{Th:UniqSmooth}, respectively, where $q\in (0,1)$. Suppose that there exist a continuous function $k\colon\mms\to\R$ and $N\in [\dim\mms,\infty)$ such that $\smash{\Ric_{g,\meas,N}(\dot\gamma,\dot\gamma)\geq (k\circ\gamma)\,g(\dot\gamma,\dot\gamma)}$ on $[0,1]$ for every timelike affinely parametrized maximizing geodesic $\gamma\colon[0,1]\to\mms$ starting in $\supp\mu_0$ and ending in $\supp\mu_1$. Then $\bdpi$-a.e.~$\gamma$ obeys the following inequality for every $N'\in[N,\infty)$ and every $t\in[0,1]$:
\begin{align*}
\frac{\rmd\mu_t}{\rmd\meas}(\gamma_t)^{-1/N'} \geq \tau_{k_\gamma^-,N'}^{(1-t)}(\vert\dot\gamma\vert)\,\frac{\rmd\mu_0}{\rmd\meas}(\gamma_0)^{-1/N'} +\tau_{k_\gamma^+,N'}^{(t)}(\vert\dot\gamma\vert)\,\frac{\rmd\mu_1}{\rmd\meas}(\gamma_1)^{-1/N'}.
\end{align*}
\end{theorem}

As opposed to the previous theorem, the following is not needed in the sequel, but it clarifies the name ``timelike curvature-dimension condition''.

\begin{corollary}[TCD from timelike Ricci curvature bounds] Assume $(\mms,g,\meas)$ is a globally hyperbolic weighted spacetime. Let $k\colon\mms\to\R$ be continuous and let $N\in[\dim\mms,\infty)$. Then the following statements are equivalent. 
\begin{enumerate}[label=\textnormal{(\roman*)}]
\item We have $\smash{\Ric_{g,\meas,N}\geq k}$ in all timelike directions.
\item The triple $(\mms,g,\meas)$ satisfies $\smash{\TCD_q(k,N)}$ for some $q\in (0,1)$.
\item The triple $(\mms,g,\meas)$ satisfies $\smash{\TCD_q(k,N)}$ for every $q\in (0,1)$.
\end{enumerate}
\end{corollary}

The extension to merely timelike $q$-dualizable endpoints in this claim is argued as in McCann \cite{mccann2020}*{Thm.~7.4} and Braun \cite{braun2023-renyi}*{Thm.~3.35}. 

\subsection{Variable timelike measure contraction property}\label{Sub:VarTMCP}

\begin{definition}[Variable timelike measure contraction property]\label{Def:TMCP} Given a continuous function $k\colon\mms\to\R$ and $N\in[\dim\mms,\infty)$, we say $(\mms,g,\meas)$ satisfies the \emph{past timelike measure contraction property}, briefly  $\smash{\TMCP^-(k,N)}$, if for every point $o\in\mms$ and every compactly supported, $\meas$-absolutely continuous $\smash{\mu_1\in\Prob(\mms)}$ that is concentrated on $\smash{I^+(o)}$,  there exist 
\begin{itemize}
    \item an exponent $q\in (0,1)$,
    \item a $q$-geodesic $\mu\colon[0,1]\to\Prob(\mms)$ from $\smash{\mu_0:=\delta_o}$ to $\mu_1$, and
    \item an $\smash{\ell_q}$-optimal dynamical coupling $\smash{\bdpi}$ of $\mu_0$ and $\mu_1$
\end{itemize}
such that for every $N'\in[N,\infty)$ and every $t\in[0,1]$,
\begin{align}\label{Eq:SNTMCP}
\scrS_{N'}(\mu_t\mid\meas)&\leq  -\int\tau_{k_\gamma^+,N'}^{(t)}(\vert\dot\gamma\vert)\,\frac{\rmd\mu_1}{\rmd\meas}(\gamma_1)^{-1/N'}\d\bdpi(\gamma).
\end{align}

Moreover, we say $(\mms,g,\meas)$ obeys
\begin{enumerate}[label=\textnormal{\alph*.}]
    \item the \emph{future timelike measure contraction property}, briefly $\smash{\TMCP^+(k,N)}$, if the spacetime  $(\mms,g)$ endowed with opposite time orientation yet the same reference measure $\meas$ satisfies $\smash{\TMCP^-(k,N)}$, and
    \item the \emph{timelike measure contraction property}, briefly $\TMCP(k,N)$, if it obeys both $\smash{\TMCP^-(k,N)}$ and $\smash{\TMCP^+(k,N)}$.
\end{enumerate}
\end{definition}

If $k$ is constant, \eqref{Eq:SNTMCP} becomes
\begin{align*}
    \scrS_{N'}(\mu_t\mid\meas)\leq -\int_\mms\tau_{k,N}^{(t)}\circ l_o\,\Big[\frac{\rmd\mu_1}{\rmd\meas}\Big]^{1-1/N'}\d\meas,
\end{align*}
where $l_o$ is from \eqref{Eq:LorDistFct}. This follows 
as $\vert\dot\gamma\vert=l_o(\gamma_1)$ for $\bdpi$-a.e.~$\gamma$ and as $(\eval_0,\eval_1)_\push\bdpi =\delta_o\otimes\mu_1$ is the only (necessarily chronological and $\smash{\ell_q}$-optimal) coupling of  $\mu_0$ and $\mu_1$. 

The TMCP does not depend on the transport exponent $q$, cf.~\cite{cavalletti-mondino2020}*{Rem.~3.8}. In addition, since every point in $\mms$ has an arbitrarily close point in its chronological past by strong causality, one can equivalently require  $\{o\}\times\supp\mu_1\subset I$ instead of the weaker property $\smash{\mu_1[I^+(o)]=1}$ in the previous definition.

Via \cref{Le:JointLSC} (approximating Dirac masses with uniform distributions), it is not difficult to prove that $\TCD_q(k,N)$ implies $\TMCP(k,N)$ for every $q\in (0,1)$, every continuous $k\colon \mms\to\R$, and every $N\in[\dim\mms,\infty)$, cf.~Cavalletti--Mondino \cite{cavalletti-mondino2020}*{Prop.~3.12} and Braun--McCann \cite{braun-mccann2023}*{Prop.~A.3}; see the proof of \cref{Th:ToTMCP} below. In general, however, the TMCP is strictly weaker than the TCD condition, cf.~Cavalletti--Mondino \cite{cavalletti-mondino2020}*{Rem.~A.3}. Yet, it has the same consequences we detail in \cref{Sub:Applic} as the TCD condition.

The following TMCP version of \cref{Th:EssPwConc} is proven following the lines of Braun--Ohta \cite{braun-ohta2024}*{Lem.~5.4}.

\begin{theorem}[Essential pathwise semiconcavity for TMCP]\label{Th:EssPwConcTMCP} Assume $(\mms,g,\meas)$ is a globally hyperbolic weighted spacetime. Moreover, let $o\in\mms$ and let $\mu_1\in\Prob(\mms)$ be a  compactly supported, $\meas$-absolutely continuous measures with $\smash{\{o\}\times\supp\mu_1\subset I}$. Let $\mu$ and $\bdpi$ constitute the unique  $q$-geodesic and $\smash{\ell_q}$-optimal dynamical coupling from \cref{Th:UniqSmooth} that connect $\smash{\mu_0:=\delta_o}$ and $\mu_1$, respectively, where $q\in (0,1)$. Suppose that there exist a continuous function $k\colon\mms\to\R$ and $N\in [\dim\mms,\infty)$ such that $\smash{\Ric_{g,\meas,N}(\dot\gamma,\dot\gamma)\geq (k\circ\gamma)\,g(\dot\gamma,\dot\gamma)}$ on $[0,1]$ for every timelike affinely parametrized maximizing geodesic $\gamma\colon[0,1]\to\mms$ starting in $o$ and ending in $\supp\mu_1$. Then $\bdpi$-a.e.~$\gamma$ obeys the following inequality for every $N'\in[N,\infty)$ and every $t\in[0,1]$:
\begin{align*}
\frac{\rmd\mu_t}{\rmd\meas}(\gamma_t)^{-1/N'} \geq \tau_{k_\gamma^+,N'}^{(t)}(\vert\dot\gamma\vert)\,\frac{\rmd\mu_1}{\rmd\meas}(\gamma_1)^{-1/N'}.
\end{align*}
\end{theorem}

\section{Smooth considerations}\label{Sub:Smooth}

Now we collect some materials from the smooth case. To this aim, in this section we assume $(\mms,g)$ is a globally hyperbolic spacetime (in particular, $g$ is \emph{smooth}) endowed with a \emph{smooth} measure $\meas$ on $\mms$, except for the proof of \cref{Le:EquiLip} at the end of \cref{Sub:Timelike cut locus}. We will abbreviate these hypotheses by calling $(\mms,g,\meas)$ globally hyperbolic weighted spacetime.

\subsection{Timelike cut locus}\label{Sub:Timelike cut locus} Let $o\in\mms$ be a given point. For details about the facts to follow, we refer to Treude \cite{treude2011} and Treude--Grant \cite{treude-grant2013} (where the timelike cut locus is called ``causal cut locus''). Their discussion for suitable submanifolds includes points \cite{treude2011}*{p.~117}, the only relevant case for us.

For $\smash{v\in T^+_o}$, we let $\gamma_v\colon I_v\to\mms$ denote the unique timelike geodesic with initial velocity $v$ such that $I_v$ is the maximal domain of definition of $\gamma_v$ relative to $\R_+$. Then the \emph{$o$-future cut function} $\smash{s_o^+\colon T_o^+\to \R_+\cup\{\infty\}}$ defined by
\begin{align}\label{Eq:CutTime}
    s_o^+(v):=\sup\{t \in I_v : l_o((\gamma_v)_t) = \Len_g(\gamma_v\big\vert_{[0,t]})\}
\end{align}
has no zeros \cite{treude2011}*{Cor.~3.2.23} and is lower semicontinuous \cite{treude2011}*{Prop.~3.2.29}. 

\begin{definition}[Timelike cut locus] The set 
\begin{align*}
    \TCut^+(o):=\{\exp_o(s^+_o(v)\,v) \in I^+(o) : v\in T_o^+,\, s_o^+(v)\in I_v\}
\end{align*}
is called \emph{future timelike cut locus} of $o$.
\end{definition}

Any element of $\TCut^+(o)$ will be called \emph{future timelike cut point} of $o$.

\begin{proposition}[Characterization of timelike cut locus \cite{treude2011}*{Prop.~3.2.28}]\label{Pr:CharTCut} A point $\smash{y\in I^+(o)}$ belongs to $\smash{\TCut^+(o)}$ if and only if either there exists more than one timelike maximizing geodesic from $o$ to $y$ or $y$ is the first focal point of $o$ along a timelike maximizing geodesic.
\end{proposition}

In fact, the set of all $\smash{y\in I^+(o)}$ for which there exists more than one timelike maximizing geodesic  from $o$ to $y$ is dense in $\smash{\TCut^+(o)}$ \cite{treude2011}*{Prop.~3.2.30}.

Recall we say a Borel measurable subset of $\mms$ has measure zero if it is negligible with respect to some (hence every) smooth measure on $\mms$.

\begin{theorem}[Properties of future before timelike cut locus \cite{treude2011}*{Thm.~3.2.31, Prop. 3.2.32}]\label{Th:BeforeTCut} We define 
\begin{align*}
    \scrJ_\tang^+(o) :=\{tv \in T_o\mms: v\in T_o^+,\, t\in [0,s_o^+(v))\}
\end{align*}
and let $\smash{\scrI_\tang^+(o)}$ denote the interior of $\smash{\scrJ_\tang^+(o)}$. Then the following statements hold.
\begin{enumerate}[label=\textnormal{(\roman*)}]
    \item The ``future before timelike cut locus''
    \begin{align*}
        \scrI^+(o) := \exp_o(\scrI^+_\tang(o))
    \end{align*}
    is open and diffeomorphic to $\smash{\scrI_\tang^+(o)}$ through $\exp_o$.
    \item We have $\smash{\scrI^+(o) = I^+(o)\setminus \TCut^+(o)}$.
    \item The set $\smash{\TCut^+(o)}$ is closed and has measure zero.
    \item The set $\smash{\scrI^+(o)}$ is the largest open subset of $\smash{I^+(o)}$ with the property that each of its points can be connected to $o$ by a unique timelike maximizing geodesic.
\end{enumerate}
\end{theorem}

Recall the Lorentz distance function $l_o$ from \eqref{Eq:LorDistFct}.

\begin{theorem}[Smoothness of distance function \cite{treude-grant2013}*{Prop.~2.7}]\label{Th:SmoothnessDist} The re\-striction of the Lorentz distance function \eqref{Eq:LorDistFct} to $\smash{\scrI^+(o)}$ is smooth.

Moreover, for every $\smash{y\in\scrI^+(o)}$ we have 
\begin{align*}
    -\nabla l_o(y) = (\dot\gamma_{\exp_o^{-1}(y)})_{l_o(y)}.
\end{align*}
In other words, $\smash{\gamma_{\exp_o^{-1}(y)}}$ coincides with the unique negative gradient flow of $l_o$  that starts in $o$ and passes through $y$.
\end{theorem}

In fact, as established by Andersson--Galloway--Howard \cite{andersson-galloway-howard1998-cosmological}*{Prop.~3.1} and later McCann \cite{mccann2020}*{Prop.~3.4, Thm.~3.5}, the Lorentz distance function \eqref{Eq:LorDistFct} is locally semiconvex (hence locally Lipschitz continuous) on all of $\smash{I^+(o)}$ and fails to be semiconcave precisely on $\smash{\TCut^+(o)}$.

In our case, we can at least obtain local Lipschitz continuity as follows. Local semiconvexity seems harder, since the proofs of \cite{andersson-galloway-howard1998-cosmological,mccann2020} quoted above (and Braun et al.~\cite{braun-gigli-mccann-ohanyan-samann+++}*{Prop.~13} in the $C^1$-case) are unclear to be ``stable'' if the metric tensor is merely locally Lipschitz continuous.

\begin{proof}[Proof of \cref{Le:EquiLip}] We first assume $g$ is smooth and show a bound on the local Lipschitz constant of $l$ on $I$ in terms of the $l$-distance of the neighborhoods in question and the local $C^0$-norm of $g$. We recall $r$ is the complete background Riemannian metric on $\mms$. Given $x,y\in\mms$ with $\smash{x\ll_g y}$, let $\smash{U_x,U_y\subset \mms}$ be small geodesically convex Riemannian balls with respect to $r$. Without restriction, we may and will assume $\smash{\cl\,U_x\times \cl\,U_y\subset I}$. Let $c>0$ and $\lambda>0$ be from \cref{Le:Bounds} applied to a constant sequence of metric tensors, $\smash{X:= \cl\,U_x}$, and $\smash{Y:= \cl\,U_y}$. By the triangle inequality and a symmetric argument it suffices to prove there is $c_0> 0$ such that for every $o\in U_x$ and every $\smash{y',y''\in U_y}$,
\begin{align*}
    \sup\{\big\vert l_o(y')-l_o(y'')\big\vert : n\in\N \textnormal{ with }n\geq n_0\} \leq c_0\,\met_r(y',y'').
\end{align*}

Given $\smash{y',y''\in U_y\setminus \TCut^+(o)}$, let $\gamma\colon[0,1]\to\mms$ be the unique affinely parametrized maximizing $r$-geodesic from $y'$ to $y''$; it does not leave $\smash{I^+(o)}$. Since the curve $\gamma$ is Lipschitz continuous, the set of its intersections with $\smash{\TCut^+(o)}$ has one-dimensional Lebesgue measure zero. Then by \cref{Th:SmoothnessDist} and \cref{Le:Bounds} (exchanging proper time by affine parametrization), we obtain
\begin{align}\label{Eq:Bdlg_n}
    \sqrt{r(\nabla l_o,\nabla l_o)} \leq \frac{c}{\lambda} \quad\gamma_\push(\Leb^1\mres[0,1])\textnormal{-a.e.}
\end{align}
On the other hand, again by \cref{Th:SmoothnessDist},
\begin{align*}
    l_o(y') - l_o(y'') = \int_{[0,1]} \rmd l_o(\dot\gamma_t)\d t = \int_{[0,1]} g(\nabla l_o,\dot\gamma_t)\d t.
\end{align*}
Since the coefficients of $g$ are uniformly bounded on $U_x$ and by \cref{Le:ApproxLemma}, we use the usual Cauchy--Schwarz inequality and local equivalence of $r$ to any other Riemannian metric (here a pull-back of the Euclidean metric) to find a constant $c_1>0$ depending on the $C^0$-norms of $g$ and $r$ on $\cl\,U_x$ such that
\begin{align*}
    \int_{[0,1]} g(\nabla l_o,\dot\gamma_t)\d t
    &\leq c_1\int_{[0,1]} \sqrt{r(\nabla l_o,\nabla l_o)}\,\sqrt{r(\dot\gamma,\dot\gamma)}\d\big[\gamma_\push(\Leb^1\mres[0,1])\big]\\
    &\leq c_0\,\met_r(y',y''),
\end{align*}
where we used \eqref{Eq:Bdlg_n} and geodesy of $\gamma$ with respect to $r$ and we set $c_0:=c_1\,c/\lambda$. As $\smash{\TCut^+(o)}$ has measure zero, the consequential inequality
\begin{align*}
    \big\vert l_o(y')-l_o(y'')\big\vert\leq c_0\,\met_r(y',y'')
\end{align*}
easily extends to arbitrary $y',y''\in U_y$ by approximation.

If $g$ is merely locally Lipschitz continuous, we invoke its good approximation $(g_n)_{n\in\N}$ from \cref{Def:Goodapprox}. Using locally uniform convergence of that sequence combined with \cref{Le:Bounds},  the previous arguments thus show  that there exist constants $n_0\in\N$ and  $c_0>0$ such that for every $n\in\N$ with $n\geq n_0$, every $o\in U_x$, and every $\smash{y',y'' \in U_y}$, we have the desired equi-Lipschitz estimate
\begin{align*}
    \big\vert l_{g_n,o}(y') - l_{g_n,o}(y'')\big\vert\leq c_0\,\met_r(y',y'').
\end{align*}

The last statement follows since on every compact subset of $\smash{I_g}$, $l_g$ is the uniform limit of the equi-Lipschitz continuous functions $\smash{(l_{g_n})_{n\in\N}}$ by \cref{Le:ApproxLemma}.
\end{proof}

\subsection{Localization}\label{Sub:Localization} The paradigm of localization (also called ``needle decomposition'') originates in convex geometry. In Lorentzian geometry, even for possibly nonsmooth metric measure spacetimes, it was introduced by Cavalletti--Mondino \cite{cavalletti-mondino2020} and then expanded by themselves \cite{cavalletti-mondino2024} and Braun--McCann \cite{braun-mccann2023}. In a nutshell, the strategy is as follows. A set in question, such as the chronological future of a given point $o\in\mms$, is foliated into timelike proper time parametrized maximizing geodesics (called ``rays'') starting at $o$. The reference measure disintegrates as
\begin{align*}
    \meas\mres I^+(o) = \int_{\mms_\alpha}\meas_\alpha\d\q(\alpha),
\end{align*}
where
\begin{itemize}
    \item $\q$ is a probability measure on a set $\smash{Q\subset I^+(o)}$ which ``labels the rays'',
    \item for $\q$-a.e.~$\alpha\in Q$, $\mms_\alpha$ forms a timelike proper time parametrized maximizing geodesic starting at $o$, and
    \item for $\q$-a.e.~$\alpha\in Q$, the conditional measure $\meas_\alpha$ is concentrated on $\mms_\alpha$.
\end{itemize}
As pointed out by Cavalletti--Mondino \cite{cavalletti-mondino2020}*{Rem.~5.4}, this result is tightly linked to the area formula. (A nonexpert can think of a ``nonstraight Fubini theorem''.) By proper time parametrization, for $\q$-a.e.~$\alpha\in Q$ the ray $\mms_\alpha$ becomes isometric (through $\smash{l}$) to an interval  $[0,b_\alpha)\subset\R_+$, where $b_\alpha\in (0,\infty]$, where each interval will be endowed with the  Euclidean distance $\vert\cdot-\cdot\vert$. By this identification and with a slight abuse of notation, $(\mms_\alpha,\vert\cdot-\cdot\vert,\meas_\alpha)$ becomes a topologically one-dimensional metric measure space. (It can be shown $\meas_\alpha$ is absolutely continuous with respect to the one-dimensional Lebesgue measure on $\mms_\alpha$ with Radon--Nikodým density $h_\alpha$ for $\q$-a.e.~$\alpha\in Q$.) The crucial point of the above paradigm is that ambient timelike Ricci curvature bounds for $(\mms,g,\meas)$ turn into Ricci curvature bounds for $(\mms_\alpha,\vert\cdot-\cdot\vert,\meas_\alpha)$ after Sturm \cite{sturm2006-ii} and Lott--Villani \cite{lott-villani2009} for $\q$-a.e.~$\alpha\in Q$. 

We will clarify two aspects about the results of Cavalletti--Mondino \cite{cavalletti-mondino2020,cavalletti-mondino2024} and Braun--McCann \cite{braun-mccann2023}, cf.~\cref{Th:Disintegration}. First, since we assume $g$ and $\meas$ to be smooth, we will show  the conditional density $h_\alpha$ is smooth for $\q$-a.e.~$\alpha\in Q$. On Riemannian manifolds, corresponding results are obtained by Klar\-tag \cite{klartag2017}. Second, the works cited above hypothesize ambient timelike Ricci curvature bounds in \emph{every} timelike direction; in fact, it suffices to stipulate these only in all ``relevant'' directions, namely those tangential to the rays. Our results below are not trimmed for optimality, yet will suffice for our purposes.

\subsubsection{Framework} Let $o\in\mms$ be a given point. Recall its future timelike cut locus $\smash{\TCut^+(o)\subset I^+(o)}$ from \cref{Sub:Timelike cut locus} as well as the Lorentz distance function $\smash{l_o}$ from \eqref{Eq:LorDistFct}. In addition, fix an open Riemannian ball  $U\subset I^+(o)$ with compact inclusion. Then continuity of  $l_+$ implies that $\vartheta>0$, where
\begin{align*}
    2\vartheta := \inf\{l_o(y): y\in \cl\,U\} > 0.
\end{align*}
Reminiscent of \cref{Th:BeforeTCut}, define
\begin{align*}
    \scrH_\tang^+(o,U) := \{tv\in T_o\mms: v\in\exp_o^{-1}(U),\,t\in[0,s_o^+(v))\}
\end{align*}
and let $\smash{\scrG_\tang^+(o,U)}$ denote the interior of $\smash{\scrH_\tang^+(o,U)}$. The set
\begin{align}\label{Eq:GoU}
    \scrG^+(o,U) :=\exp_o(\scrG_\tang^+(o,U))
\end{align}
is the ``thin ice cone'' consisting of $o$ and all points in $I^+(o)$ that lie in the relative interior of a (necessarily uniquely determined) time\-like proper time parametrized maximizing geodesic from $o$ that meets $U$. In particular, $\smash{\scrG^+(o,U) \cap\TCut^+(o) =\emptyset}$ thanks to \cref{Pr:CharTCut}. Also, by \cref{Th:BeforeTCut} the set $\smash{\scrG^+(o,U)\setminus\{o\}}$ is open and the inclusion $\smash{U\subset \scrG^+(o,U)}$ holds up to an $\meas$-negligible set, for
\begin{align*}
    \meas[U\setminus \scrG^+(o,U)] \leq \meas[U \cap \TCut^+(o)] \leq \meas[\TCut^+(o)]=0.
\end{align*}

\subsubsection{Smooth disintegration} Fix $r\in (0,\vartheta)$ and set
\begin{align}\label{Eq:Sigmar}
    \Sigma_r := \{l_o = r\}\cap \scrG^+(o,U).
\end{align}
As $r$ is a regular value of the restriction of $l_o$ to $\smash{\scrG^+(o,U)}$ (because $\smash{g(\nabla l_o,\nabla l_o)=1}$ there), $\Sigma_r$ is a hypersurface. It is clearly  precompact and achronal. As its normal vector field $\smash{-\nabla l_o\big\vert_{\Sigma_r}}$, cf. \cref{Th:SmoothnessDist}, is timelike, the hypersurface $\Sigma_r$ is spacelike. Let $\mu_r$ be the restriction of $\meas$ to $\Sigma_r$. Employing that $g$ restricts to a Riemannian metric on $\Sigma_r$, up to the implicit exponential weight we see that $\mu_r$ corresponds to the induced Riemannian volume measure on $\Sigma_r$.

Let $\smash{\Phi\colon W \to \mms}$ form the maximally defined (jointly smooth) negative gradient flow of $l_o$, where $W\subset \R \times \scrG^+(o,U)$ is open, defined  by
\begin{align*}
    \Phi_s(y) := \exp_y(-s\,\nabla l_o).
\end{align*}
Recall by \cref{Th:SmoothnessDist}, for every $\smash{y\in \scrG^+(o,U)}$ the translation $\smash{\Phi_{\bullet- l_o(y)}(y)}$ restricts to the unique timelike geo\-desic that starts in $o$ and passes through $y$ on a suitable subinterval of $\R_+$ containing zero. 

Let $b\colon \scrG^+(o,U)\to (0,\infty)$ be defined by
\begin{align*}
    b_y := \sup\{t\in \R_+ : \Phi_{t-l_o(y)}(y) \in \scrG^+(o,U)\}.
\end{align*}
which is clearly positive and finite. Since $\smash{\scrG^+(o,U)}$ does not contain future timelike cut points of $o$, we evidently have 
\begin{align*}
    b \leq s_o^+\circ \exp_o^{-1},
\end{align*}
where the $o$-future cut function on the right-hand side is from \eqref{Eq:CutTime}. In particular, we see $\smash{\Phi_{\bullet-l_o(y)}(y)}$ is maximizing on all of $\smash{[0,b_y)}$.

By assumption on $\meas$, its density $\smash{\rmd\meas/\rmd\vol_g}$ is smooth. Then the induced  weighted Jacobian $J_{\meas}\Phi_\bullet\colon W\to (0,\infty)$ is given by
\begin{align}\label{Eq:weightedJac}
    J_{\meas}\Phi_s(y) := \frac{\rmd\meas}{\rmd\vol_g}(y)^{-1}\Big[\frac{\rmd\meas}{\rmd\vol_g}\circ\Phi_s(y)\Big]\,\big\vert\!\det\rmD\Phi_s(y)\big\vert,
\end{align}
where the last factor denotes the Jacobian of the map
$\Phi_s\colon \Sigma_r\to\Sigma_{r+s}$ with respect to the Riemannian volume
measures induced by $g$ on $\Sigma_r$ and $\Sigma_{r+s}$.

With these preparations, the subsequent area formula is standard. We refer to Treude \cite{treude2011}*{§§1.3, A.2} for comprehensive proofs. The formulas provided therein directly generalize to the weighted case by modifying the test function.

\begin{proposition}[Area formula]\label{Pr:AreaDist} Let $\psi\colon \scrG^+(o,U)\to\R$ be continuous and compactly supported. Then we have
\begin{align*}
    \int_{\scrG^+(o,U)}\psi\d\meas &= \int_{\Sigma_r}\int_{[0,b_y)}\psi\circ\Phi_{t-r}(y)\,J_{\meas}\Phi_{t-r}(y)\d t\d\mu_r(y).
\end{align*}
\end{proposition}

\subsubsection{Smooth CD disintegration} Now we  invoke curvature properties along the individual rays given by \cref{Pr:AreaDist}. 

Recall a map $F\colon X\to Y$ between metrizable spaces $X$ and $Y$ is called \emph{universally measurable} if it is $\mu$-measurable for every Borel probability measure $\mu$ on $X$. Let $\mathscr{M}(M)$ denote the space of finite Borel measures on $\mms$, which becomes a metric space when endowed with the total variation distance.

\begin{theorem}[Smooth CD disintegration of $\meas$]\label{Th:Disintegration} Let $(\mms,g,\meas)$ be a globally hyperbolic weighted spacetime. Given a point $o\in\mms$ and an open Riemannian ball $\smash{U\subset I^+(o)}$ with compact inclusion, let $\smash{\scrG^+(o,U)\subset I^+(o)\cup\{o\}}$ be defined by \eqref{Eq:GoU}. Let $N\in [\dim\mms,\infty)$ and assume there is  a continuous function $\smash{k\colon \scrG^+(o,U)\to\R}$ with $\smash{\Ric_{g,\meas,N}(\dot\gamma,\dot\gamma) \geq (k\circ\gamma)\,g(\dot\gamma,\dot\gamma)}$ on $[0,1]$ for every timelike affinely para\-met\-rized maximizing geodesic $\gamma\colon[0,1]\to\mms$ starting in $o$ and intersecting $U$. Then there are
\begin{itemize}
    \item a spacelike hypersurface $\smash{Q\subset \scrG^+(o,U)}$,
    \item a probability measure $\q$ on $Q$, and
    \item a continuous map $\meas_\bullet\colon Q\to \mathscr{M}(\mms)$
\end{itemize}
satisfying the following properties. 
\begin{enumerate}[label=\textnormal{(\roman*)}]
\item  We have 
\begin{align*}
    \meas\mres \scrG^+(o,U)=\int_Q\meas_\alpha\d\q(\alpha).
\end{align*}
\item For $\q$-a.e.~$\alpha\in Q$, the conditional measure  $\meas_\alpha$ is concentrated on the timelike proper time parametrized maximizing geodesic $\smash{\Phi_{\bullet-l_o(\alpha)}(\alpha)\big\vert_{[0,b_\alpha)}}$.
\item For $\q$-a.e.~$\alpha\in Q$, the conditional measure $\meas_\alpha$ is absolutely continuous with respect to the one-dimensional Lebesgue measure on $\mms_\alpha$. Letting $f_\alpha$ be its Radon--Nikodým density,  $f_\alpha \circ \Phi_{\bullet-l_o(\alpha)}(\alpha)$ is a smooth $\CD(k\circ \Phi_{\bullet-l_o(\alpha)}(\alpha),N)$ density on $(0,b_\alpha)$ in the sense of \cref{Def:CDdensity}.
\end{enumerate}
\end{theorem}

\begin{proof} We focus on the construction and only sketch the proof. Given $r\in (0,\vartheta)$ as above, we set $Q := \Sigma_r$ from \eqref{Eq:Sigmar} and $\q := \mu_r[\Sigma_r]^{-1}\,\mu_r$. Moreover, for $\alpha\in Q$ we define $\meas_\alpha$ as the push-forward of the measure $\smash{J_{\meas,\bullet-r}(\alpha)\,\Leb^1\mres [0,b_\alpha)}$ under the map $\Phi_{\bullet-r}(\alpha)$. By \cref{Pr:AreaDist}, this directly yields the  claimed disintegration formula and the second statement. The last statement follows from standard Jacobi field computations, cf.~Treude \cite{treude2011}*{§1.4} and Treude--Grant \cite{treude-grant2013}*{§3}, in combination with \cref{Re:BasPrCD}. Alternatively, this follows by arguing as in the proofs of Cavalletti--Mondino \cite{cavalletti-mondino2024}*{Thm.~3.2} and Braun--McCann \cite{braun-mccann2023}*{Thm.~6.37}, observing that the ambient timelike Ricci curvature bounds hypothesized therein are only required in the directions tangential to the rays (which correspond to the flow trajectories of $\Phi$ in our situation).
\end{proof}

Evidently, by multiplying the measures in question with suitable constants, this yields a disintegration theorem for the probability measure
\begin{align*}
    \neas:= \meas[\scrG^+(o,U)]^{-1}\,\meas\mres\scrG^+(o,U).
\end{align*}

Recall $\Prob(\mms)$ is the space of Borel probability measures on $\mms$, endowed with the narrow topology (which is metrizable, cf.~Ambrosio--Gigli--Savaré \cite{ambrosio-gigli-savare2008}*{Rem.~5.1.1}).

\begin{corollary}[Smooth CD disintegration of $\neas$]\label{Cor:Disintegration} We retain the hypotheses and the notation from \cref{Th:Disintegration}. Then there exist
\begin{itemize}
    \item a spacelike hypersurface $\smash{Q\subset \scrG^+(o,U)}$,
    \item a probability measure $\q$ on $Q$, and
    \item a continuous map $\neas_\bullet\colon Q\to \Prob(\mms)$
\end{itemize}
satisfying the following properties. 
\begin{enumerate}[label=\textnormal{(\roman*)}]
\item We have 
\begin{align*}
    \neas=\int_Q\neas_\alpha\d\q(\alpha).
\end{align*}
\item For $\q$-a.e.~$\alpha\in Q$, the conditional measure  $\neas_\alpha$ is concentrated on the timelike proper time parametrized maximizing geodesic $\smash{\Phi_{\bullet-l_o(\alpha)}(\alpha)\big\vert_{[0,b_\alpha)}}$.
\item  For $\q$-a.e.~$\alpha\in Q$, the conditional measure $\neas_\alpha$ is absolutely continuous with respect to the one-dimensional Lebesgue measure on $\mms_\alpha$. Letting $h_\alpha$ be its Radon--Nikodým density,  $h_\alpha \circ \Phi_{\bullet-l_o(\alpha)}(\alpha)$ is a smooth $\CD(k\circ \Phi_{\bullet-l_o(\alpha)}(\alpha),N)$ density on $(0,b_\alpha)$ in the sense of \cref{Def:CDdensity}.
\end{enumerate}
\end{corollary}

We will call a triple $(Q,\q,\neas_\bullet)$ as given by the previous corollary \emph{smooth $\CD(k,N)$ disintegration} of $\neas$.

\subsection{Displacement convexity with defect}\label{Sub:WithDefect} For the proof of our main technical result, \cref{Th:DisThm}, finer structural properties of a certain $q$-geodesic $\mu\colon[0,1]\to\mms$ through a globally hyperbolic weighted spacetime $(\mms,g,\meas)$, where $q\in (0,1)$, are needed and recalled now. They were obtained by McCann \cite{mccann2020} and Braun--Ohta \cite{braun-ohta2024}, whose results even hold in the Lorentz--Finsler framework. Moreover, it will be convenient to represent $\mu$ as push-forward of an intermediate measure $\mu_t$, where $t\in (0,1)$, instead of $\mu_0$ as done in \cref{Th:UniqSmooth}.

Let $t\in (0,1)$. Assume $\mu_0,\mu_1\in\Prob(\mms)$ are compactly supported, $\meas$-absolutely continuous, and satisfy $\supp\mu_0\times\supp\mu_1\subset I$. Then the pair $(\mu_0,\mu_1)$ is $q$-separated in the sense of McCann \cite{mccann2020}*{Def.~4.1} by \cite{mccann2020}*{Lem.~4.4}, which (unlike the chronology hypothesis on $\supp\mu_0$ and $\supp\mu_1$) propagates to the interior of the unique $q$-geodesic $\mu\colon[0,1]\to\Prob(\mms)$ connecting $\mu_0$ and $\mu_1$ from \cref{Th:UniqSmooth} \cite{mccann2020}*{Prop.~5.5}. Let $\Psi$ be the transport map from \cref{Th:UniqSmooth}. By this theorem and \cref{Re:Invert}, the map $\smash{\Phi^t\colon[0,1]\times  \supp\mu_t\to\mms}$ defined by
\begin{align}\label{Eq:TransportMap}
    \Phi^t_r := \Psi_r \circ\Psi_t^{-1}
\end{align}
is well-defined; in particular,
\begin{align}\label{Eq:Id}
    \Phi_t^t= \Id\big\vert_{\supp\mu_t}.
\end{align}
Moreover, defining $G^t\colon \supp\mu_t\to C^0([0,1];\mms)$ by
\begin{align}\label{Eq:TransportMapGeo}
G^t(x)_r:= \Phi^t_r(x),
\end{align}
we see $(G^t)_\push\mu_t$ equals the unique $\smash{\ell_q}$-optimal dynamical coupling $\bdpi$ of $\mu_0$ and $\mu_1$.

On the other hand, there exist both an open neighborhood $U\subset\mms$ of $\supp\mu_t$ and a locally semiconvex (hence locally Lipschitz continuous) Kantorovich potential $\smash{\varphi^t\colon U\to\R}$ for $\mu_t$ and $\mu_1$. It can be computed from $\varphi$ using the Hopf--Lax evolution, cf.~McCann \cite{mccann2020}*{Rem.~5.6}. Again by uniqueness, for every $r\in[0,1]$ we have
\begin{align*}
    \exp_\bullet(-(r-t)\,\vert\nabla\varphi^t\vert^{q'-2}\,\nabla\varphi^t) = \Phi_r^t\quad\mu_t\textnormal{-a.e.}
\end{align*}

An explicit relation between the Radon--Nikodým densities of $\mu_t$ and $\mu_1$ with respect to $\meas$, coming from the change of variables formula and ultimately leading to the Jacobi field computations linking optimal transport and timelike Ricci curvature, is the Monge--Ampère equation. To this aim, given a point $x\in\supp\mu_t$ let $\smash{\vert \!\det \rmD \Phi^t(x)\vert}$ denote the Jacobian of $\smash{\Phi^t(x)}$ associated with $\vol_g$ along the timelike affinely parametrized maximizing geodesic $\smash{G^t(x)}$; here, $\rmD$ is understood as  the approximate derivative from McCann \cite{mccann2020}*{Def.~3.8}. Reminiscent of \eqref{Eq:weightedJac}, let us define the corresponding Jacobian with respect to $\meas$ by
\begin{align*}
    J_\meas\Phi^t_r(x) := \frac{\rmd\meas}{\rmd\vol_g}(x)^{-1}\,\Big[\frac{\rmd\meas}{\rmd\vol_g}\circ \Phi^t_r(x)\Big]\,\big\vert\!\det\rmD \Phi_r^t(x)\big\vert.
\end{align*}

\begin{proposition}[Monge--Ampère  identity {\cite{mccann2020}*{Cor.~5.11}}]\label{Pr:Monge} We have
\begin{align*}
    \frac{\rmd\mu_t}{\rmd\meas} = J_\meas \Phi_1^t\,\frac{\rmd\mu_1}{\rmd\meas}\circ \Phi_1^t\quad\mu_t\textnormal{\textit{-a.e.}}
\end{align*}
\end{proposition}

Moreover, the computations from Step 4 in the proof of Braun--Ohta \cite{braun-ohta2024}*{Thm. 5.9} using an orthonormal frame around each point of the transport geodesic in question show given $r\in[0,1]$ and  $x\in\supp\mu_t$, the above transport Jacobian from $\mu_t$ to $\mu_r$ admits a factorization of the form
\begin{align*}
    J_\meas\Phi_r^t(x)
    =
    L_r^t(x)\,\big|\!\det B_r^t(x)\big|\,\frac{\rmd\meas}{\rmd\vol_g}(x)^{-1}\,\frac{\rmd\meas}{\rmd\vol_g}\circ \Phi_r^t(x),
\end{align*}
where $\smash{B_r^t(x)}$ is the matrix representation of the approximate derivative 
$\smash{\rmD\Phi_r^t}$ on the orthogonal complement of the tangent vector at $\smash{G^t(x)_r}$ relative to $g$ and
\begin{align*}
    L_r^t(x) &:= \frac{J_\meas\Phi_r^t(x)}{Y_r^t(x)},
\end{align*}
describes the distortion in the tangential direction at $\smash{G^t(x)_r}$, where
\begin{align}\label{Eq:FctsLI}
    Y_r^t(x) &:= \big|\!\det B_r^t(x)\big|\,\frac{\rmd\meas}{\rmd\vol_g}(x)^{-1}\,\frac{\rmd\meas}{\rmd\vol_g}\circ \Phi_r^t(x).
\end{align}
Note that by \eqref{Eq:Id} and with the above normalization, we have
\begin{align}\label{Eq:Ltt}
    L_t^t(x)=1.
\end{align}

\begin{lemma}[Concavity inequalities {\cite{braun-ohta2024}*{Thm.~5.9}}]\label{Le:FineEst}
The following statements hold for every $x\in\supp\mu_t$.
\begin{enumerate}[label=\textnormal{(\roman*)}]
    \item The tangential part $\smash{L^t(x)}$ is concave, i.e.~for every $r\in[0,1]$,
    \begin{align*}
        L_r^t \geq (1-r)\,L_0^t+r\,L_1^t.
    \end{align*}
    \item Given a continuous function $k\colon \mms\to\R$, define $\smash{k_\gamma^\pm\colon [0, \vert\dot\gamma\vert]\to\R}$ as before \cref{Def:VarTCD}, where $\smash{\gamma:= G^t(x)}$. In addition, given $N\in[\dim\mms,\infty)$, we assume $\smash{\Ric_{g,\meas,N}(\dot\gamma,\dot\gamma)\geq (k\circ\gamma)\,g(\dot\gamma,\dot\gamma)}$ on $[0,1]$. Then the orthogonal part $\smash{Y^t(x)}$ from \eqref{Eq:FctsLI} satisfies the following inequality for every $r\in [0,1]$:
    \begin{align*}
        Y_r^t(x)^{1/(N-1)} &\geq \sigma_{k_\gamma^-/(N-1)}^{(1-r)}(\vert\dot\gamma\vert)\,Y_0^t(x)^{1/(N-1)}\\
        &\qquad\qquad + \sigma_{k_\gamma^+/(N-1)}^{(r)}(\vert\dot\gamma\vert)\,Y_1^t(x)^{1/(N-1)}.
    \end{align*}
\end{enumerate}
\end{lemma}

We are in a position to prove our main technical ingredient.

\begin{theorem}[Displacement convexity with defect]\label{Th:DisThm} Assume that $(\mms,g,\meas)$ is a globally hyperbolic weighted spacetime. We suppose $\mu_0,\mu_1\in\Prob(\mms)$ are compactly supported and $\meas$-absolutely continuous. Let $q\in (0,1)$, $K\in\R$, and $p \in (N/2,\infty)$. Suppose there exist $\lambda>0$ and $\eta \in (0,\pi_{K/(N-1)}/2)$ with 
\begin{align}\label{Eq:Requ}
\lambda\leq l\circ(\eval_0,\eval_1)&\leq \pi_{K/(N-1)}-\eta\quad\bdpi\textnormal{\textit{-a.e.}},
\end{align}
where $\bdpi$ is the unique $\smash{\ell_q}$-optimal dynamical coupling representing the $q$-geodesic $\mu\colon[0,1]\to\Prob(\mms)$ given by $\mu_0$ to $\mu_1$ from \cref{Th:UniqSmooth} and $\pi_{K/(N-1)}$ is from \eqref{Eq:pikappa}. Let a given continuous function  $k\colon\mms\to\R$ and $N\in [\dim\mms,\infty)$ have the property that $\smash{\Ric_{g,\meas,N}(\dot\gamma,\dot\gamma)\geq (k\circ \gamma)\,g(\dot\gamma,\dot\gamma)}$ on $[0,1]$ for every timelike affinely parametrized maximizing geodesic $\gamma\colon[0,1]\to\mms$ starting in $\supp\mu_0$ and terminating in $\supp\mu_1$. Lastly, let $L>0$ be an upper bound on the Lipschitz constant of $l$ on $\supp\mu_0\times\supp\mu_1$, cf. \cref{Le:EquiLip}. 
Then for every $t\in (0,1)$,
\begin{align*}
\scrS_N(\mu_t\mid\meas) &\leq -\int \tau_{K,N}^{(1-t)}(\vert\dot\gamma\vert)\,\frac{\rmd\mu_0}{\rmd\meas}(\gamma_0)^{-1/N} \d\bdpi(\gamma)\\
&\qquad\qquad -\int\tau_{K,N}^{(t)}(\vert\dot\gamma\vert)\,\frac{\rmd\mu_1}{\rmd\meas}(\gamma_1)^{-1/N}\d\bdpi(\gamma)\\
&\qquad\qquad + 2\,\big[\Lambda_{K,N,\eta}\,\Omega_{K,N,p,\eta}^{1/(2p-1)}\big]^{1/N}\,\meas[C]^{2(p-1)/N(2p-1)}\\
&\qquad\qquad\qquad\qquad\times \big[\!\diam\,(C,g)^{2p-q+1}\,\lambda^{q-1}\,L\big]^{1/N(2p-1)}\\
&\qquad\qquad\qquad\qquad \times \Big[\!\int_C\big\vert(k-K)_-\big\vert^p\d\meas\Big]^{1/N(2p-1)},
\end{align*}
where $\smash{\Lambda_{K,N,\eta}}$ and $\smash{\Omega_{K,N,p,\eta}}$ are from \eqref{Eq:TwoConst} and we set $C:= J(\supp\mu_0,\supp\mu_1)$.
\end{theorem}

\begin{proof} By  \cref{Th:EssPwConc}, $\bdpi$-a.e.~$\gamma$ satisfies
\begin{align}\label{Eq:Concest}
\frac{\rmd\mu_t}{\rmd\meas}(\gamma_t)^{-1/N}\geq \tau_{k_\gamma^-,N}^{(1-t)}(\vert\dot\gamma\vert)\,\frac{\rmd\mu_0}{\rmd\meas}(\gamma_0)^{-1/N} + \tau_{k_\gamma^+,N}^{(t)}(\vert\dot\gamma\vert) \,\frac{\rmd\mu_1}{\rmd\meas}(\gamma_1)^{-1/N}.
\end{align}
Our objective will be to replace the variable $k$ distortion coefficients  by the constant $K$ distortion coefficients after integration against $\bdpi$, up to an explicit error. In the following, we concentrate on replacing the distortion coefficient $\smash{\tau_{k_\gamma^+,N}^{(t)}(\vert\dot\gamma\vert)}$ by $\smash{\tau_{K,N}^{(t)}(\vert\dot\gamma\vert)}$; analogously, $\smash{\tau_{k_\gamma^-,N}^{(1-t)}(\vert\dot\gamma\vert)}$ is replaced by $\smash{\tau_{K,N}^{(1-t)}(\vert\dot\gamma\vert)}$.

By \cref{Le:DefDist} and \eqref{Eq:Requ}, $\bdpi$-a.e.~$\gamma$ satisfies 
\begin{align*}
&\tau_{K,N}^{(t)}(\vert\dot\gamma\vert) - \tau_{k_\gamma^+,N}^{(t)}(\vert\dot\gamma\vert)\\
&\qquad\qquad \leq \big[\Lambda_{K,N,\eta}\,\Omega_{K,N,p,\eta}^{1/(2p-1)}\big]^{1/N}\,\Big[\!\int_{[t,1]} \tau_{k_\gamma^+,N}^{(r)}(\vert\dot\gamma\vert)^N\d r\Big]^{2(p-1)/N(2p-1)}\\
&\qquad\qquad\qquad\qquad \times\Big[t\,\vert\dot\gamma\vert^{2p}\int_{[0,1]} \big\vert (k(\gamma_r) - K)_-\big\vert\,\sigma_{k_\gamma^+/(N-1)}^{(r)}(\vert\dot\gamma\vert)^{N-1}\d r\Big]^{1/N(2p-1)}. 
\end{align*}
Multiplying this inequality by $\smash{\rmd\mu_1/\rmd\meas(\gamma_1)^{-1/N}}$ and applying Jensen's inequality first and Hölder's inequality second,
\begin{align*}
&\int \tau_{K,N}^{(t)}(\vert\dot\gamma\vert)\,\frac
{\rmd\mu_1}{\rmd\meas}(\gamma_1)^{-1/N}\d\bdpi(\gamma) - \int \tau_{k_\gamma^+,N}^{(t)}(\vert\dot\gamma\vert)\,\frac{\rmd\mu_1}{\rmd\meas}(\gamma_1)^{-1/N}\d\bdpi(\gamma)\\
&\qquad\qquad \leq  \big[\Lambda_{K,N,\eta}\,\Omega_{K,N,p,\eta}^{1/(2p-1)}\big]^{1/N}\int \rmd\bdpi(\gamma)\,\Big[\frac{\rmd\mu_1}{\rmd\meas}(\gamma_1)^{-1/N}\\
&\qquad\qquad\qquad\qquad \times \Big[\!\int_{[t,1]} \tau_{k_\gamma^+,N}^{(r)}(\vert\dot\gamma\vert)^N\d r\Big]^{2(p-1)/N(2p-1)}\\
&\qquad\qquad\qquad\qquad\times\Big[t\,\vert\dot\gamma\vert^{2p}\int_{[0,1]} \big\vert (k(\gamma_r) - K)_-\big\vert\, \sigma_{k_\gamma^+/(N-1)}^{(r)}(\vert\dot\gamma\vert)^{N-1}\d r\Big]^{1/N(2p-1)}\Big]\\
&\qquad\qquad\leq  \big[\Lambda_{K,N,\eta}\,\Omega_{K,N,p,\eta}^{1/(2p-1)}\big]^{1/N}\,\Big[\!\int \rmd\bdpi(\gamma)\,\frac{\rmd\mu_1}{\rmd\meas}(\gamma_1)^{-1}\\
&\qquad\qquad\qquad\qquad \times \Big[\!\int_{[t,1]} \tau_{k_\gamma^+,N}^{(r)}(\vert\dot\gamma\vert)^N\d r\Big]^{2(p-1)/(2p-1)}\\
&\qquad\qquad\qquad\qquad\times\Big[t\,\vert\dot\gamma\vert^{2p}\int_{[0,1]} \big\vert (k(\gamma_r) - K)_-\big\vert\, \sigma_{k_\gamma^+/(N-1)}^{(r)}(\vert\dot\gamma\vert)^{N-1}\d r\Big]^{1/(2p-1)}\Big]^{1/N}\\
&\qquad\qquad\leq \big[\Lambda_{K,N,\eta}\,\Omega_{K,N,p,\eta}^{1/(2p-1)}\big]^{1/N}\,\Big[\!\underbrace{\iint_{[t,1]} \tau_{k_\gamma^+,N}^{(r)}(\vert\dot\gamma\vert)^N\,\frac{\rmd\mu_1}{\rmd\meas}(\gamma_1)^{-1}\d r\d\bdpi(\gamma)}_{\mathrm{A}}\!\Big]^{2(p-1)/N(2p-1)}\\
&\qquad\qquad\qquad\qquad\times\Big[t\int\d\bdpi(\gamma)\,\vert\dot\gamma\vert^{2p}\,\Big[\!\int_{[0,1]}\d r\,\big\vert (k(\gamma_r)-K)_-\big\vert^p \\
&\qquad\qquad\qquad\qquad\phantom{\times\Big[\ }\underbrace{\phantom{t}\qquad\times\sigma_{k_\gamma^+/(N-1)}^{(r)}(\vert\dot\gamma\vert)^{N-1}\,\frac{\rmd\mu_1}{\rmd\meas}(\gamma_1)^{-1}\Big]}_{\mathrm{B}}\Big]^{1/N(2p-1)}.
\end{align*}
In the sequel, we will estimate the integrals $\mathrm{A}$ and $\mathrm{B}$ separately. In particular, note carefully the factor $\smash{t\,\sigma_{k_\gamma^+/(N-1)}^{(r)}(\vert\dot\gamma\vert)^{N-1}}$ appearing in $\rmB$ does \emph{not} summarize to the $N$-th power of a single $\tau$-distortion coefficient.

Finding an upper bound for $\mathrm{A}$ is simple. Indeed, using \eqref{Eq:Concest}, Fubini's theorem, and the fact that $\supp\mu_r \subset C$ for every $r\in [t,1]$,
\begin{align*}
\mathrm{A} &= \iint_{[t,1]} \tau_{k_\gamma^+,N}^{(r)}(\vert\dot\gamma\vert)^N\,\frac{\rmd\mu_1}{\rmd\meas}(\gamma_1)^{-1}\d r\d\bdpi(\gamma)\\
&\leq \int_{[t,1]}\int \frac{\rmd\mu_r}{\rmd\meas}(\gamma_r)^{-1}\d\bdpi(\gamma)\d r\\
&=\int_{[t,1]}\int_\mms \Big[\frac{\rmd\mu_r}{\rmd\meas}\Big]^{-1}\d\mu_r\d r\\
&\leq \meas[C].
\end{align*}
In particular, the contribution from $\mathrm{A}$ is bounded.

To find an upper bound for $\mathrm{B}$, we first rewrite the integral with respect to $\bdpi$ as an integral with respect to $\mu_t$. To this aim, we recall the transport maps $\Psi_1$, $\Psi_t$, and $\smash{\Phi_1^t}$ from  \cref{Th:UniqSmooth} and \eqref{Eq:TransportMap}; in particular,
\begin{align}\label{Eq:Becomesss}
    \Psi_1 = \Phi_1^t\circ \Psi_t\quad\mu_0\textnormal{-a.e.}
\end{align}
Given a point $x\in\supp\mu_t$, let the timelike affinely parametrized maximizing geodesic $\gamma^x\colon[0,1]\to\mms$ be defined by $\smash{\gamma^x := G^t(x)}$, where $\smash{G^t(x)}$ is from \eqref{Eq:TransportMapGeo}. Combining the identity $\mu_1 = (\Psi_1)_\push\mu_0$ from \cref{Th:UniqSmooth}, \eqref{Eq:Becomesss}, and $(\Psi_t)_\push \mu_0 = (\eval_t)_\push\bdpi$,
\begin{align*}
    \rmB  &= t\int \rmd\bdpi(\gamma) \,\vert\dot\gamma\vert^{2p}\, \Big[\!\int_{[0,1]} \rmd r\,\big\vert(k(\gamma_r)-K)_-\big\vert^p\\
    &\qquad\qquad \times \sigma_{k_\gamma^+/(N-1)}^{(r)}(\vert\dot\gamma\vert)^{N-1}\,\Big[\frac{\rmd\mu_1}{\rmd\meas}\circ \Phi_1^t\Big](\gamma_t)^{-1}\Big]\\
     &= t\int_\mms \rmd\mu_t\,\vert \dot\gamma^\bullet\vert^{2p} \,\Big[\!\int_{[0,1]} \rmd r\,\big\vert(k(\gamma_r^\bullet)-K)_-\big\vert^p\\
     &\qquad\qquad\times \sigma_{k_{\gamma^\bullet}^+/(N-1)}^{(r)}(\vert\dot\gamma^\bullet\vert)^{N-1}\,\Big[\frac{\rmd\mu_1}{\rmd\meas}\circ \Phi_1^t\Big]^{-1}\Big].
\end{align*}
Let $\smash{L^t}$ and $\smash{B^t}$ be as above; more precisely, they are  given by decomposing the $\smash{\ell_q}$-optimal transport from $\mu_t$ to $\mu_1$ into a tangential and an orthogonal part, respectively. The Monge--Ampère-type identity from \cref{Pr:Monge} implies
\begin{align*}
    \rmB &= \int_\mms\rmd\mu_t\, \vert\dot\gamma^\bullet\vert^{2p}\,\Big[\Big[\!\int_{[0,1]}\rmd r\,\big\vert(k(\gamma_r^\bullet)-K)_-\big\vert^p\\
    &\qquad\qquad\times \sigma_{k_{\gamma^\bullet}^+/(N-1)}^{(r)}(\vert\dot\gamma^\bullet\vert)^{N-1} \,\big\vert\!\det B_1^t\big\vert\,\frac{\rmd\meas}{\rmd\vol_g}\circ \Phi_1^t\Big]\\
    &\qquad\qquad\times t\,L_1^t\,\Big[\frac{\rmd\meas}{\rmd\vol_g}\Big]^{-1}\,\Big[\frac{\rmd\mu_t}{\rmd\meas}\Big]^{-1}\Big].
\end{align*}
Noting that $\smash{\Phi_1^t(x)= \gamma^x_1}$ for $\mu_t$-a.e.~$x\in\mms$ by \eqref{Eq:TransportMap} yields
\begin{align*}
    \rmB &\leq \int_\mms\rmd\mu_t\,\vert\dot\gamma^\bullet\vert^{2p}\, \Big[\Big[\!\int_{[0,1]} \big\vert (k(\gamma_r^\bullet)-K)_-\big\vert^p\,\big\vert\!\det B_r^t\big\vert\,\frac{\rmd\meas}{\rmd\vol_g}(\gamma_r^\bullet)\d r\Big]\\
    &\qquad\qquad \times L_t^t\,\Big[\frac{\rmd\meas}{\rmd\vol_g}\Big]^{-1}\,\Big[\frac{\rmd\mu_t}{\rmd\meas}\Big]^{-1}\Big]\\
    &= \underbrace{\int_\mms\Big[\!\int_{[0,1]} \big\vert (k(\gamma_r^\bullet)-K)_-\big\vert^p\,\big\vert\!\det B_r^t\big\vert\,\frac{\rmd\meas}{\rmd\vol_g}(\gamma_r^\bullet)\d r\Big]\,\vert\dot\gamma^\bullet\vert^{2p} \d\vol_g}_{\rmB'};
\end{align*}
here, we used the relations
\begin{align*}
    \sigma_{k_\gamma^\bullet/(N-1)}^{(r)}(\vert\dot\gamma\vert)^{N-1}\,\big\vert\!\det B_1^t\big\vert\,\frac{\rmd\meas}{\rmd\vol_g}\circ\Phi_1^t &\leq \big\vert\!\det B_r^t\big\vert\,\frac{\rmd\meas}{\rmd\vol_g}(\gamma_r^\bullet),\\
    t\,L_1^{t,1} &\leq L_t^{t,1} =1
\end{align*}
implied by \cref{Le:FineEst} and the normalization \eqref{Eq:Ltt}.

A standard application of the coarea formula, cf.~e.g.~Ketterer \cite{ketterer2021-stability}*{Lem.~5.2}, now allows us to disintegrate $\smash{\vol_g}$ along the level sets of the Kantorovich potential $\smash{\varphi^t}$ for the $\smash{\ell_q}$-optimal transport from $\mu_t$ to $\mu_1$; more precisely, there exist both an $\smash{\Leb^1}$-measurable set $Q\subset \R$ contained in the image of $\smash{\varphi^{t,1}}$ and, for $\Leb^1$-a.e.~$a\in Q$, an $\smash{\mathscr{H}^{\dim\mms-1}}$-measurable subset $\smash{\Sigma_a\subset\{\varphi^t=a\}}$ such that
\begin{align*}
    \rmB' &= \int_{Q} \int_{\Sigma_a} \Big[\!\int_{[0,1]} \big\vert (k(\gamma_r^\bullet)-K)_-\big\vert^p\,\big\vert\!\det B_r^t\big\vert\,\frac{\rmd\meas}{\rmd\vol_g}(\gamma_r^\bullet)\d r\Big]\,\frac{\vert\dot\gamma^\bullet\vert^{2p}}{\vert\nabla \varphi^t\vert}\d \scrH^{\dim\mms-1}\d a\\
    &=   \int_{Q}\rmd a \int_{\Sigma_a}\rmd\scrH^{\dim\mms-1}\, \vert\dot\gamma^\bullet\vert^{2p-q}\, \Big[\!\int_{[0,1]} \rmd r\,\big\vert (k(\gamma_r^\bullet)-K)_-\big\vert^p\\
    &\qquad\qquad\times\big\vert\!\det B_r^t\big\vert\,\frac{\rmd\meas}{\rmd\vol_g}(\gamma_r^\bullet)\,\vert\dot\gamma^\bullet\vert\Big]\\
    &\leq \diam\,(C,g)^{2p-q}\int_Q\rmd a\int_{\Sigma_a}\rmd\scrH^{\dim\mms-1}\,\Big[\!\int_{[0,1]}\rmd r\,\big\vert(k(\gamma_r^\bullet)-K)_-\big\vert^p\\
    &\qquad\qquad \times \big\vert\!\det B_r^t\big\vert\,\frac{\rmd\meas}{\rmd\vol_g}(\gamma_r^\bullet)\,\vert\dot\gamma^\bullet\vert\Big];
\end{align*}
here, $\smash{\scrH^{\dim\mms-1}}$ designates the $(\dim\mms-1)$-dimensional Hausdorff measure and we have used the identity $\smash{\vert\dot\gamma^\bullet\vert= \vert\nabla\varphi^t\vert^{q'-1}}$ $\mu_t$-a.e.~implied by \cref{Th:UniqSmooth} in the second equality. We abbreviate the iterated integral from the last line above by $\rmB''$. Now a standard application of the area formula, cf.~e.g.~Ketterer \cite{ketterer2021-stability}*{Prop.~5.6}, shows that $\Leb^1$-a.e.~$a\in Q$ satisfies
\begin{align*}
    &\int_{[0,1]}\int_{\Sigma_a}\big\vert (k(\gamma_r^\bullet)-K)_-\big\vert^p\,\big\vert\!\det B_r^t\big\vert\,\vert\dot\gamma^\bullet\vert\,\frac{\rmd\meas}{\rmd\vol_g}(\gamma_r^\bullet)\d\mathscr{H}^{\dim\mms-1}\d r\\
    &\qquad\qquad \leq \int_C \big\vert(k-K)_-\big\vert^p\d\meas,
\end{align*}
and consequently
\begin{align*}
    \rmB'' \leq \Leb^1[Q]\int_C\big\vert (k-K)_-\big\vert^p\d\meas.
\end{align*}

It remains to find an upper bound for $\Leb^1[Q]$. By construction of $Q$, 
\begin{align*}
\Leb^1[Q]\leq\Leb^1\big[\varphi^t(\supp\mu_t)\big].    
\end{align*}
On the other hand, using $\mu$ is a $q$-geodesic that connects its endpoints $\mu_0$ and $\mu_1$ which satisfy $\supp\mu_0\times\supp\mu_1\subset I$, the results of  McCann \cite{mccann2020}*{Thm.~4.3, Lem.~4.4, Prop.~5.5} imply that the Lipschitz constant of $\smash{\varphi^t}$ on $\supp\mu_t$ is no larger than the Lipschitz constant of $l^q/q$ on $\supp\mu_0\times\supp\mu_1$. In turn, by \eqref{Eq:Requ} the latter is no larger than $\lambda^{q-1}$ times the Lipschitz constant of $l$ on $\supp\mu_0\times\supp\mu_1$, where we used $q\in(0,1)$. Since $\smash{\Leb^1[\varphi^t(\supp\mu_t)]}$ is no larger than the oscillation of $\smash{\varphi^t}$ on $\supp\mu_t$, by our choice of $L$ this easily implies
\begin{align*}
    \Leb^1\big[\varphi^t(\supp\mu_t)\big] \leq \diam\,(C,g)\,\lambda^{q-1}\,L.
\end{align*}
This concludes the proof.
\end{proof}

\section{Sharp timelike Bonnet--Myers inequality}\label{Sec:BonnetMyers}

We now prove our first main result. Besides the good approximation of Lipschitz continuous metric tensors from \cref{Sub:Approx} and the localization paradigm from \cref{Sub:Localization}, our argument relies on the qualitative diameter estimate from \cref{Cor:deltacor}.

\begin{theorem}[Sharp timelike Bonnet--Myers diameter estimate]\label{Th:TimelikeBonnetMyers} Let $K>0$ and $N\in [\dim\mms,\infty)$. Assume that $(\mms,g,\meas)$ is a globally hyperbolic weighted Lipschitz spacetime such that $\Ric_{g,\meas,N}\geq K$ timelike distributionally. Then
\begin{align*}
    \diam\,(\mms,g)\leq\pi\sqrt{\frac{N-1}{K}}.
\end{align*}
\end{theorem}

\begin{proof} Assume to the contrary there exist $\varepsilon > 0$ and $o,x\in\mms$ with
\begin{align*}
    l_o(x) \geq\pi\sqrt{\frac{N-1}{K}} + 4\varepsilon.
\end{align*}
Since $l_+$ is continuous, there is an open Riemannian ball $U\subset\mms$ of $x$ with 
\begin{align*}
    \inf\{l_{g,o}(y) : y\in U\} \geq \pi\sqrt{\frac{N-1}{K}} +3\varepsilon.
\end{align*}

Let $(g_n)_{n\in\N}$ and $(\meas_n)_{n\in\N}$ be fixed good  approximations of $g$ and $\meas$ on the compact set $\smash{C:= J_g(o,\cl\,U)}$, respectively. Recall the induced weighted spacetime $(\mms,g_n,\meas_n)$ is globally hyperbolic for every $n\in\N$. By \cref{Le:UnifCvg}, we may and will  assume there exists $n_1\in\N$ such that for every $n\in\N$ with $n\geq n_1$,
\begin{align}\label{Eq:infln}
    \inf\{l_{g_n,o}(y): y\in U\} \geq\pi\sqrt{\frac{N-1}{K}} +2\varepsilon.
\end{align}

In the sequel, for $n$ as above we abbreviate  
\begin{align*}
    G_n := \scrG_{g_n}^+(o,U),
\end{align*}
where the right-hand side is from \eqref{Eq:GoU}, defined with respect to $g_n$. By \cref{Th:BeforeTCut} and narrow convergence of $(\meas_n)_{n\in\N}$ to $\meas$,
\begin{align*}
    \liminf_{n\to\infty} \meas_n[G_n] \geq \liminf_{n\to\infty} \meas_n[U\setminus \TCut^+_{g_n}(o)] = \liminf_{n\to\infty} \meas_n[U] = \meas[U] >0.
\end{align*}
On the other hand, since $\smash{G_n \subset J_g(o,\cl\,U)}$ for every $n\in\N$,
\begin{align*}
    \limsup_{n\to\infty} \meas_n[G_n] \leq \limsup_{n\to\infty}\meas_n[C] = \meas[C] <\infty.
\end{align*}
In particular, we may and will assume $(\meas_n[G_n])_{n\in\N}$ converges in $(0,\infty)$, up to a nonrelabeled subsequence. In turn, there exist constants $w>0$ and $n_2\in\N$ with $n_2\geq n_1$ such that for every $n\in \N$ with $n\geq n_2$, 
\begin{align}\label{Eq:UW_n}
    \meas_n[G_n] &\geq w.
\end{align}

Up to increasing the value of $n_2$, we define  $k_{V,g_n,\meas_n,N}\colon\mms\to\R$ by \eqref{Eq:kVgm}, where $V\subset T^+\big\vert_C$ is from \cref{Le:Bounds}, for every $n\in\N$ with $n\geq n_2$. For  $p\in (N/2,\infty)$, let $\delta_{K,N,p,\varepsilon}>0$ be as in \cref{Cor:deltacor}. Moreover, given $\eta \in (0,1)$, by \cref{Le:Lpcvg} there exist $\delta > 0$ such that
\begin{align}\label{Eq:sqrtdelta}
    \sqrt{\delta} \leq \min\{\delta_{K,N,p,\varepsilon},\eta\}
\end{align}
and $n_3\in\N$ with $n_3\geq n_2$ such that for every $n\in \N$ with $n\geq n_3$,
\begin{align}\label{Eq:Wne}
    \int_{G_n} \big\vert (k_{V,g_n,\meas_n,N}-K)_-\big\vert^p\d\meas_n  \leq w\,\delta.
\end{align}
For every such $n$, defining
\begin{align*}
    \neas_n := \meas_n[G_n]^{-1}\,\meas_n\mres G_n,
\end{align*}
the relations \eqref{Eq:Wne} and \eqref{Eq:UW_n} imply
\begin{align}\label{Eq:leqdelta}
    \int_{G_n}\big\vert (k_{V,g_n,\meas_n,N}-K)_-\big\vert^p\d\neas_n \leq \delta.
\end{align}

Given $n$ as above, \cref{Le:Bounds} combines with  \cref{Cor:Disintegration} to ensure there is a smooth $\CD(k_{V,g_n,\meas_n,N},N)$ disintegration $(Q_n,\q_n,\neas_{n,\bullet})$ of $\neas_n$. We now use this tool to seek the desired contradiction. First, note that by construction  of $G_n$ and the disintegration, the ray $\mms_{n,\alpha}$ intersects $U$ for $\q_n$-a.e.~$\alpha\in Q_n$. 
On the other hand, let $R_n$ be the set of all $\alpha\in Q_n$ such that
\begin{align*}
    \int_{\mms_{n,\alpha}}\big\vert(k_{V,g_n,\meas_n,N}- K)_-\big\vert^p\d\neas_{n,\alpha} \leq \sqrt{\delta}.
\end{align*}
By Markov's inequality, the disintegration \cref{Cor:Disintegration}, \eqref{Eq:leqdelta}, and \eqref{Eq:sqrtdelta},
\begin{align*}
    \q_n[Q_n\setminus R_n] &\leq \frac{1}{\sqrt{\delta}}\int_{Q_n}\int_{\mms_{n,\alpha}} \big\vert (k_{V,g_n,\meas_n,N}-K)_-\big\vert^p\d\neas_{n,\alpha}\d\q_n(\alpha)\\
    &= \frac{1}{\sqrt{\delta}}\int_{G_n} \big\vert(k_{V,g_n,\meas_n,N}-K)_-\big\vert^p\d\neas_n\\
    &\leq \eta.
\end{align*}
Therefore, by \eqref{Eq:sqrtdelta} again together with  \cref{Cor:deltacor}, the diameter of $\q_n$-a.e.~ray passing through the set $R_n$ (of $\q_n$-measure at least $1-\eta$) is bounded from above by $\smash{\pi\sqrt{(N-1)/K}+\varepsilon}$. On the other hand, as $\q_n$-a.e.~ray intersects $U$ by construction, we see the diameter of $\q_n$-a.e.~ray is bounded from below by $\smash{\pi\sqrt{(N-1)/K}+2\varepsilon}$ by \eqref{Eq:infln}. As the intersection of the two sets in question has $\q_n$-measure at least $1-\eta>0$, hence is nonempty, we obtain the sought contradiction.
\end{proof}

\section{Main results and applications}

\subsection{From analytic to synthetic}\label{Sub:FromAntoSyn}

\begin{theorem}[Timelike measure contraction property]\label{Th:ToTMCP} Let $(\mms,g,\meas)$ be a globally hyperbolic weighted Lipschitz spacetime. Assume $K\in\R$ and $N\in[\dim\mms,\infty)$ satisfy $\smash{\Ric_{g,\meas,N}\geq K}$ timelike distributionally. Then $(\mms,g,\meas)$ obeys $\TMCP(K,N)$.
\end{theorem}

In fact, the main work in showing \cref{Th:ToTMCP} will be invested in the proof of the following somewhat stronger statement.

\begin{proposition}[Displacement semiconvexity under chronological support]\label{Th:TowTCD} We retain the hypotheses and the notation from \cref{Th:ToTMCP}. For every $q\in (0,1)$ and every compactly supported, $\meas$-absolutely continuous measures $\mu_0,\mu_1\in\Prob(\mms)$ with $\supp\mu_0\times \supp\mu_1\subset I$, there exist
\begin{itemize}
    \item a $q$-geodesic $\mu\colon[0,1]\to\Prob(\mms)$ from $\mu_0$ to $\mu_1$ and
    \item an $\smash{\ell_q}$-optimal dynamical coupling $\bdpi$ of $\mu_0$ and $\mu_1$
\end{itemize}
such that for every $N'\in [N,\infty)$ and every $t\in[0,1]$,
\begin{align*}
    \scrS_{N'}(\mu_t\mid\meas) &\leq -\int \tau_{K,N'}^{(1-t)}(\vert\dot\gamma\vert)\,\frac{\rmd\mu_0}{\rmd\meas}(\gamma_0)^{-1/N'}\d\bdpi(\gamma)\\
    &\qquad\qquad -\int \tau_{K,N'}^{(t)}(\vert\dot\gamma\vert)\,\frac{\rmd\mu_1}{\rmd\meas}(\gamma_1)^{-1/N'}\d\bdpi(\gamma).
\end{align*}
\end{proposition}

\begin{proof} Let $(g_n)_{n\in\N}$ and $(\meas_n)_{n\in\N}$ be fixed good approximations of $g$ and $\meas$ on the compact set $C:=J_g(\supp\mu_0,\supp\mu_1)$ as in \cref{Def:Goodapprox}. Note that by construction, for every $i\in\{0,1\}$ and every $n\in\N$ we have
\begin{align}\label{Eq:a_n}
    \frac{\rmd\mu_i}{\rmd\meas_n}= a_n\,\frac{\rmd\mu_i}{\rmd\meas},
\end{align}
where $(a_n)_{n\in\N}$ is the sequence of continuous functions $a_n:=\rmd\meas/\rmd\meas_n$ converging uniformly to one on $C$; in particular, $\mu_0$ and $\mu_1$ are $\meas_n$-absolutely continuous and compactly supported. By \cref{Le:Bounds}, we may and will also assume $\supp\mu_0\times\supp\mu_1\subset I_{g_n}$ without restriction. Consider the corresponding sequence $\smash{(\bdpi^n)_{n\in\N}}$, where $\smash{\bdpi^n}$ is the unique $\smash{\ell_{g_n,q}}$-optimal dynamical coupling of $\mu_0$ and $\mu_1$ provided by \cref{Th:UniqSmooth}. By \cref{Pr:VaryingStab}, we may and will assume that it converges narrowly to an $\smash{\ell_{g,q}}$-optimal dynamical coupling $\bdpi$  of $\mu_0$ and $\mu_1$, up to a nonrelabeled subsequence. In the sequel, given $t\in[0,1]$ we write $\mu_t:=(\eval_t)_\push\bdpi$ and $\smash{\mu_t^n:=(\eval_t)_\push\bdpi^n}$, where $n\in\N$.

The idea now is to take the limit of the inequality from \cref{Th:DisThm} along this converging sequence, which requires some further preparations. We will only address the case when $K$ is positive; the other cases are covered analogously (with simpler arguments, since we do not have to take into account diameter bounds in this complementary situation). By \cref{Le:Bounds}, we may and will assume the existence of $\lambda > 0$ such that for every $n\in\N$,
\begin{align*}
    l_{g_n}\circ(\eval_0,\eval_1)\geq \lambda\quad\bdpi^n\textnormal{-a.e.}
\end{align*}
On the other hand, let us fix $K' \in (0,K)$. By \cref{Th:TimelikeBonnetMyers} and since $\smash{\Ric_{g,\meas,N}\geq K}$ timelike distributionally, there exists $\eta \in (0,\pi_{K/(N-1)}/2)$ such that
\begin{align*}
    \diam\,(J_g(\supp\mu_0,\supp\mu_1),g) \leq \pi\sqrt{\frac{N-1}{K'}} - 2\eta.
\end{align*}
By \cref{Le:UnifCvg}, we may and will assume every $n\in\N$ satisfies
\begin{align*}
    \diam\,(J_{g_n}(\supp\mu_0,\supp\mu_1),g_n)\leq \pi\sqrt{\frac{N-1}{K'}}-\eta,
\end{align*}
which in turn implies
\begin{align*}
    l_{g_n}\circ(\eval_0,\eval_1) \leq \pi\sqrt{\frac{N-1}{K'}}-\eta\quad\bdpi^n\textnormal{-a.e.}
\end{align*}
Lastly, given $n\in\N$ let $V$ be the set of tangent vectors from \cref{Le:Bounds}; given $n\in\N$, let $k_{V,g_n,\meas_n,N'}$ be the continuous function from \eqref{Eq:kVgm}, where $N'\in[N,\infty)$ is fixed. By \cref{Le:Bounds} again, we may and will assume for every $n\in\N$, the Bakry--\smash{Émery}--Ricci tensor $\smash{\Ric_{g_n,\meas_n,N'}}$ is bounded from below by $\smash{k_{V,g_n,\meas_n,N'}}$ along every timelike affinely parametrized maximizing $g_n$-geodesic from $\supp\mu_0$ to $\supp\mu_1$. Lastly, let $L>0$ be a uniform upper bound  the Lipschitz constant of $\smash{l_{g_n}}$ on $\supp\mu_0\times\supp\mu_1$ for every $n\in\N$, as given  by \cref{Le:EquiLip}. Thus, given $p\in (N'/2,\infty)$ and $t\in (0,1)$, these considerations combine with \cref{Th:DisThm} (and the dimensional consistency of its hypotheses) to imply
\begin{align*}
    \scrS_{N'}(\mu_t^n\mid\meas_n) &\leq -\int\tau_{K',N'}^{(1-t)}(\vert\dot\gamma\vert)\,\frac{\rmd\mu_0}{\rmd\meas_n}(\gamma_0)^{-1/N'}\d\bdpi^n(\gamma)\\
    &\qquad\qquad -\int \tau_{K',N'}^{(t)}(\vert\dot\gamma\vert)\,\frac{\rmd\mu_1}{\rmd\meas_n}(\gamma_1)^{-1/N'}\d\bdpi^n(\gamma)\\
&\qquad\qquad + 2\,\big[\Lambda_{K',N',\eta}\,\Omega_{K',N',p,\eta}^{1/(2p-1)}\big]^{1/N'}\,\meas_n[C]^{2(p-1)/N'(2p-1)}\\
&\qquad\qquad\qquad\qquad\times \big[\!\diam\,(C,g_n)^{2p-q+1}\,\lambda^{q-1}\,L\big]^{1/N(2p-1)}\\
&\qquad\qquad\qquad\qquad \times \Big[\!\int_C\big\vert(k_{V,g_n,\meas_n,N'}-K)_-\big\vert^p\d\meas_n\Big]^{1/N'(2p-1)},
\end{align*}
where $\smash{\Lambda_{K',N',\eta}}$ and $\smash{\Omega_{K',N',p,\eta}}$ are from \eqref{Eq:TwoConst} and we write $\smash{C:= J_g(\supp\mu_0,\supp\mu_1)}$; moreover, here we have employed the inclusion $\smash{J_{g_n}(\supp\mu_0,\supp\mu_1)\subset C}$ implied by $g_n\prec g$ for every $n\in\N$.

Now we send $n\to\infty$. First, given $t\in(0,1)$,  \cref{Le:JointLSC} implies
\begin{align*}
    \scrS_{N'}(\mu_t\mid\meas)\leq \liminf_{n\to\infty}\scrS_{N'}(\mu_t^n\mid\meas_n).
\end{align*}
Second, \eqref{Eq:a_n} and the proof of Sturm \cite{sturm2006-ii}*{Lem.~3.3} easily yield
\begin{align*}
    &\int\tau_{K',N'}^{(1-t)}(\vert\dot\gamma \vert)\,\frac{\rmd\mu_0}{\rmd\meas}(\gamma_0)^{-1/N'}\d\bdpi(\gamma)\\
    &\qquad\qquad \leq \liminf_{n\to\infty} \int\tau_{K',N'}^{(1-t)}(\vert\dot\gamma \vert)\,\frac{\rmd\mu_0}{\rmd\meas_n}(\gamma_0)^{-1/N'}\d\bdpi^n(\gamma),\\
    &\int\tau_{K',N'}^{(t)}(\vert\dot\gamma \vert)\,\frac{\rmd\mu_1}{\rmd\meas}(\gamma_1)^{-1/N'}\d\bdpi(\gamma)\\
    &\qquad\qquad \leq \liminf_{n\to\infty} \int\tau_{K',N'}^{(t)}(\vert\dot\gamma \vert)\,\frac{\rmd\mu_1}{\rmd\meas_n}(\gamma_1)^{-1/N'}\d\bdpi^n(\gamma).
\end{align*}
Third, since the sequence $(a_n)_{n\in\N}$ from \eqref{Eq:a_n} is uniformly bounded on $C$, so is $(\meas_n[C])_{n\in\N}$. Fourth, by \cref{Le:UnifCvg} the sequence $(\diam\,(C,g_n))_{n\in\N}$ is bounded. Fifth, using \cref{Le:Lpcvg} we obtain
\begin{align*}
    \lim_{n\to\infty}\int_C\big\vert(k_{V,g_n,\meas_n,N'}-K)_-\big\vert^p\d\meas_n =0.
\end{align*}
In summary, this entails
\begin{align*}
    \scrS_{N'}(\mu_t\mid\meas) &\leq -\int\tau_{K',N'}^{(1-t)}(\vert\dot\gamma\vert)\,\frac{\rmd\mu_0}{\rmd\meas}(\gamma_0)^{-1/N'}\d\bdpi(\gamma)\\
    &\qquad\qquad -\int \tau_{K',N'}^{(t)}(\vert\dot\gamma\vert)\,\frac{\rmd\mu_1}{\rmd\meas}(\gamma_1)^{-1/N'}\d\bdpi(\gamma).
\end{align*}
Sending $K'\to K$ and using Levi's monotone convergence theorem with \cref{Rem: props coefs},
\begin{align*}
    \scrS_{N'}(\mu_t\mid\meas) &\leq -\int\tau_{K,N'}^{(1-t)}(\vert\dot\gamma\vert)\,\frac{\rmd\mu_0}{\rmd\meas}(\gamma_0)^{-1/N'}\d\bdpi(\gamma)\\
    &\qquad\qquad -\int \tau_{K,N'}^{(t)}(\vert\dot\gamma\vert)\,\frac{\rmd\mu_1}{\rmd\meas}(\gamma_1)^{-1/N'}\d\bdpi(\gamma),
\end{align*}
which is the desired inequality.
\end{proof}

\begin{proof}[Proof of \cref{Th:ToTMCP}] The result in question follows by combining \cref{Th:TowTCD} and {\cite[Prop.~4.9]{braun2023-renyi}}. Indeed, while the latter assumes the so-called  ``weak timelike curvature-dimension condition'', its proof only uses displacement semiconvexity of the Rényi entropy between mass distributions satisfying the hypotheses from  \cref{Th:TowTCD}. The idea is to shrink one of the $\meas$-absolutely continuous masses from the above proposition to a Dirac mass and using similar stability properties at the level of geodesics of probability measures as in the above proof.
\end{proof}

\subsection{Applications of \cref{Th:ToTMCP}}\label{Sub:Applic} As well-known, the TMCP established in \cref{Th:ToTMCP} (and the property from  \cref{Th:TowTCD}) has two standard consequences: the timelike Brunn--Minkowski inequality as well as the timelike Bishop--Gromov inequality. (The third usual consequence, the timelike Bonnet--Myers inequality, was shown in \cref{Th:TimelikeBonnetMyers} and used to prove \cref{Th:ToTMCP}.) We refer to Cavalletti--Mondino \cite{cavalletti-mondino2020}*{Props.~3.4, 3.5} and Braun \cite{braun2023-renyi}*{Prop.~3.11, Thm.~3.16} for the proofs.

In the following, given $X_0,X_1\subset\mms$, let $G(X_0,X_1)$ denote the set of all timelike affinely parametrized maximizing geodesics that start in $X_0$ and terminate in $X_1$. Moreover, recall the definition of the evaluation map $\eval_t$ from \cref{Sub:GeosLift}.

\begin{theorem}[Timelike Brunn--Minkowski inequality]\label{Th:TimelikeBrunnMinkowski} Let $(\mms,g,\meas)$ be a globally hyperbolic weighted Lipschitz spacetime. Assume that $K\in\R$ and $N\in [\dim\mms,\infty)$ satisfy $\smash{\Ric_{g,\meas,N}\geq K}$ timelike distributionally. Then the following claims hold.
\begin{enumerate}[label=\textnormal{(\roman*)}]
    \item For every $o\in\mms$, every precompact Borel set $X_1\subset\mms$ with $\{o\}\times X_1\subset I$, and  every $t\in [0,1]$,
\begin{align*}
    &\meas^*[\eval_t(G(\{o\},X_1))]^{1/N}\\
    &\qquad\qquad\geq \inf\{\tau_{K,N}^{(t)}(\vert\dot\gamma\vert) : \gamma\in G(\{o\},X_1)\}\,\meas[X_1]^{1/N},
\end{align*}
where $\meas^*$ denotes the outer measure induced by $\meas$.
\item For every precompact Borel sets $X_0,X_1\subset\mms$ such that  $\cl\,X_0\times\cl\,X_1\subset I$ and every $t\in[0,1]$,
\begin{align*}
    &\meas^*[\eval_t(G(X_0,X_1))]^{1/N}\\
    &\qquad\qquad\geq \inf\{\tau_{K,N}^{(1-t)}(\vert\dot\gamma\vert) : \gamma\in G(X_0,X_1)\}\,\meas[X_0]^{1/N}\\
    &\qquad\qquad\qquad\qquad + \inf\{\tau_{K,N}^{(t)}(\vert\dot\gamma\vert) : \gamma\in G(X_0,X_1)\}\,\meas[X_1]^{1/N}.
\end{align*}
\end{enumerate}
\end{theorem}

To formulate the timelike Bishop--Gromov inequality, let $o\in\mms$. Given $r>0$, let
\begin{align*}
    B_g(o,r) := \{y\in I^+(o): l_o(y) < r\} \cup\{o\}
\end{align*}
be the future $r$-``ball'' centered at $o$, which typically has infinite $\meas$-measure. Thus, we fix a compact set $\smash{E\subset I^+(o)\cup\{o\}}$ that is \emph{star-shaped} with respect to $o$; that is,  every timelike affinely parametrized maximizing geodesic $\gamma\colon[0,1]\to\mms$ that starts in $o$ and ends in $E$ does not leave $E$. We define the induced volume function $v\colon \R_+ \to \R_+$ and the induced area function $s\colon \R_+\to\R_+\cup\{\infty\}$ by
\begin{align*}
    v(r) &:= \meas[E\cap \cl\,B_g(o,r)],\\
    s(r) &:= \limsup_{\delta\to 0^+}\frac{v(r+\delta)-v(r)}{\delta}\\
    &\phantom{:}= \limsup_{\delta\to 0^+} \frac{1}{\delta}\,\meas[E \cap ((\cl\,B_g(o,r+\delta))\setminus B_g(o,r))].
\end{align*}
Moreover, given $\kappa\in\R$ we recall the generalized sine function $\sin_\kappa$ from \eqref{Eq:sinkappa} and its first positive root $\pi_\kappa$ from \eqref{Eq:pikappa}, respectively.

\begin{theorem}[Timelike Bishop--Gromov inequality]\label{Th:TimelikeBishopGromov} Let $(\mms,g,\meas)$ be a globally hyperbolic weighted Lipschitz spacetime satisfying $\smash{\Ric_{g,\meas,N}\geq K}$ timelike distributionally for some $K\in\R$ and $N\in[\dim\mms,\infty)$.  Let $o\in\mms$ and let $\smash{E\subset I^+(o)\cup\{o\}}$ be compact  and star-shaped with respect to $o$. Then for every $r,R\in (0,\pi_{K/(N-1)})$ such that $r<R$, 
\begin{align*}
    \frac{s(r)}{s(R)} &\geq \frac{\sin_{K/(N-1)}(r)^{N-1}}{\sin_{K/(N-1)}(R)^{N-1}},\\
    \frac{v(r)}{v(R)} &\geq \frac{\displaystyle\int_{[0,r]} \sin_{K/(N-1)}(t)^{N-1}\d t}{\displaystyle\int_{[0,R]}\sin_{K/(N-1)}(t)^{N-1}\d t}.
\end{align*}
\end{theorem}

Our final application of \cref{Th:ToTMCP} are two  d'Alembert comparison theorems, one for distance functions from points and one for appropriate powers. In a much less smooth framework, these were established by  Beran et al.~\cite{beran-braun-calisti-gigli-mccann-ohanyan-rott-samann+-}. However, the compatibility of  the abstract first-order quantities they introduce with the classical notions on our given Lipschitz spacetime is unclear (cf.~\cite{beran-braun-calisti-gigli-mccann-ohanyan-rott-samann+-}*{§A.1} for the smooth framework).  To avoid interrupting the flow of ideas, we defer the self-contained (and more streamlined) proofs of \cref{Th:Dalembert powers,Th:DAlembert}, based on the arguments of Beran et al.~\cite{beran-braun-calisti-gigli-mccann-ohanyan-rott-samann+-}, to \cref{App:One}.

To formulate the comparison theorems, whose proof is analogous to Beran et al. \cite{beran-braun-calisti-gigli-mccann-ohanyan-rott-samann+-}*{Thm.~5.24}, given $K\in\R$, $N\in(1,\infty)$, and $\theta\in [0,\pi_{K/(N-1)})$ we set 
\begin{align*}
    \tilde{\tau}_{K,N}(\theta) &:= \frac{\rmd}{\rmd t}\bigg\vert_1 \tau_{K,N}^{(t)}(\theta)\\
    &= \begin{cases}
        \displaystyle\frac1N+ \frac{\theta}{N}\sqrt{K(N-1)}\cot\!\Big[\theta\sqrt{\frac{K}{N-1}}\Big] & \textnormal{if }K>0,\\
        1 & \textnormal{if }K=0,\\
        \displaystyle\frac1N+ \frac{\theta}{N}\sqrt{-K(N-1)}\coth\!\Big[\theta\sqrt{\frac{-K}{N-1}}\Big] & \textnormal{otherwise}.
    \end{cases}
\end{align*}
Moreover, in view of \cref{Th:TimelikeBonnetMyers}, given $o\in \mms$ we define the open set
\begin{align}\label{Eq:IKN}
    I_{K,N}^+(o) := I^+(o) \cap\{l_o< \pi_{K/(N-1)}\}.
\end{align}

\begin{theorem}[D'Alembert comparison theorem for  powers of Lorentz distance functions]\label{Th:Dalembert powers} Assume $(\mms,g,\meas)$ is a globally hyperbolic weighted Lipschitz spacetime. Suppose that $K\in\R$ and $N\in[\dim\mms,\infty)$ obey $\smash{\Ric_{g,\meas,N}\geq K}$ timelike distributionally. Let $q\in (0,1)$ and let $q'<0$ be its conjugate exponent. Then for every $o\in\mms$ and every  Lipschitz continuous function $\smash{\phi\colon I^+_{K,N}(o)\to\R_+}$ with compact support,
\begin{align*}
    -\int_\mms\rmd\phi\Big[\nabla\frac{l_o^q}{q}\Big]\,\Big\vert\nabla \frac{l_o^q}{q}\Big\vert^{q'-2}\d\meas\leq N\int_\mms \tilde{\tau}_{K,N}\circ l_o\,\phi\d\meas.
\end{align*}
\end{theorem}

\begin{theorem}[D'Alembert comparison theorem for Lorentz distance functions]\label{Th:DAlembert} Assume $(\mms,g,\meas)$ is a globally hyperbolic weighted Lipschitz spacetime. Suppose that $K\in\R$ and $N\in[\dim\mms,\infty)$ satisfy $\smash{\Ric_{g,\meas,N}\geq K}$ timelike distributionally. Let us fix $q'<0$. Then for every $o\in\mms$ and every  Lipschitz continuous function $\smash{\phi\colon I_{K,N}^+(o)\to\R_+}$ with compact support,
\begin{align*}
    -\int_\mms \rmd\phi(\nabla l_o)\,\vert\nabla l_o\vert^{q'-2}\d\meas\leq \int_\mms\phi\,\frac{N\,\tilde{\tau}_{K,N}\circ l_o-1}{l_o}\d\meas.
\end{align*}
\end{theorem}

In fact, like the right-hand sides, the left-hand sides in the above conclusions are independent of $q'$ by \cref{Cor:Unit slope} (and thus hold for switched signs of $q$ and $q'$, respectively). Nevertheless, making this contribution visible has made surprising elliptic methods accessible in Lorentzian geometry, as  realized by Beran et al.~\cite{beran-braun-calisti-gigli-mccann-ohanyan-rott-samann+-}; we refer  to  the reviews of McCann \cite{mccann+} and Braun \cite{braun2025}. 

As in Beran et al.~\cite{beran-braun-calisti-gigli-mccann-ohanyan-rott-samann+-}, whose results were refined by Braun \cite{braun2024+}, these two theorems imply the Lorentz distance functions in question (and their appropriate powers) admit a distributional d'Alembertian in the sense of \cref{Def:DistrDAlem}. \cref{Cor:Ex1,Cor:Ex2} below follow from \cite{beran-braun-calisti-gigli-mccann-ohanyan-rott-samann+-}*{Prop.~5.31}; the basic idea is to apply Riesz--Markov--Kakutani's representation theorem, using  \cref{Th:Dalembert powers,Th:DAlembert} to generate the needed nonnegative functionals.

For an open set $U\subset\mms$, let $\Lip_\comp(U)$ denote the set of all Lipschitz continuous functions $\phi\colon U\to\R$ with compact support. A Radon functional on $U$ is a linear map $T\colon U\to\R$ that is continuous with respect to the topology induced by uniform convergence on compact subsets of $U$. 
We will call such a Radon functional $T$ nonnegative if $T(\phi)\geq 0$ whenever its argument $\phi\in\Lip_\comp(U)$ is nonnegative. Recall by the Riesz--Markov--Kakutani representation theorem, every nonnegative Radon functional is given by (integration against) a Radon measure. The difference of two Radon measures always makes sense as a Radon functional, but in general not as a signed Radon measure (as both positive and negative parts may be infinite, which prevents the resulting object from being $\sigma$-additive).

\begin{definition}[Distributional d'Alembertian]\label{Def:DistrDAlem} Given $q'<0$, we will say a locally Lipschitz continuous function $u\colon U\to\R$ with $\meas\mres U$-a.e.~timelike gradient lies in the domain of the \emph{distributional $q'$-d'Alembertian}, symbolically $\smash{u\in\Dom(\Box_{g,\meas,q'}\mres U)}$, if there is a Radon functional $T\colon\Lip_\comp(U)\to\R$ such that for every $\phi\in\Lip_\comp(U)$,
\begin{align*}
    -\int_U \rmd\phi(\nabla u)\,\vert\nabla u\vert^{q'-2}\d\meas  =T(\phi).
\end{align*}
\end{definition}

Of course, if a map $T$ as above exists, it is unique.

\begin{corollary}[Distributional d'Alembertian for powers of Lorentz distance functions]\label{Cor:Ex1} Assume $K\in\R$ and $N\in [\dim\mms,\infty)$ are such that the globally hyperbolic weighted Lipschitz spacetime $(\mms,g,\meas)$ satisfies $\smash{\Ric_{g,\meas,N}\geq K}$ timelike distributionally. Let $q\in (0,1)$ with conjugate exponent $\smash{q'<0}$. Then for every $o\in\mms$, we have $\smash{l_o^q/q\in\Dom(\Box_{g,\meas,q'}\mres I_{K,N}^+(o))}$, where the set $\smash{I_{K,N}^+(o)}$ is from \eqref{Eq:IKN}.

More precisely, the unique Radon functional $\smash{T\colon \Lip_\comp(I_{K,N}^+(o))\to\R}$ that  certifies the previous inclusion admits the decomposition
\begin{align*}
    -T = \mu^T -\nu^T,
\end{align*}
where $\smash{\mu^T}$ is a Radon measure on $\smash{I_{K,N}^+(o)}$ and
\begin{align*}
    \nu^T := N\,\tilde{\tau}_{K,N}\circ l_o\,\meas\mres I_{K,N}^+(o).
\end{align*}

In particular, the singular part of $T$ with respect to $\meas$, defined as the $\meas$-singular part of the signed Radon measure $\smash{-\mu^T}$, is nonpositive.
\end{corollary}

\begin{corollary}[Distributional d'Alembertian for Lorentz distance functions]\label{Cor:Ex2} We assume $K\in\R$ and $N\in [\dim\mms,\infty)$ are such that the globally hyperbolic weighted Lipschitz spacetime $(\mms,g,\meas)$ satisfies $\smash{\Ric_{g,\meas,N}\geq K}$ timelike distributionally.  Then for every $o\in\mms$ and every $\smash{q'<0}$, we have $\smash{l_o\in\Dom(\Box_{g,\meas,q'}\mres I_{K,N}^+(o))}$.

More precisely, the unique Radon functional $\smash{S\colon \Lip_\comp(I_{K,N}^+(o))\to\R}$ that  certifies the previous inclusion admits the decomposition
\begin{align*}
    -S = \mu^S -\nu^S,
\end{align*}
where $\smash{\mu^S}$ is a Radon measure on $\smash{I_{K,N}^+(o)}$ and
\begin{align*}
    \nu^S := \frac{N\,\tilde{\tau}_{K,N}\circ l_o-1}{l_o}\,\meas\mres I_{K,N}^+(o).
\end{align*}

In particular, the singular part of $S$ with respect to $\meas$, defined as the $\meas$-singular part of the signed Radon measure $\smash{-\mu^S}$, is nonpositive.
\end{corollary}

\appendix

\section{Further sharp timelike diameter estimates}\label{Sec:Further}

The proof of \cref{Th:TimelikeBonnetMyers} indicates a diameter estimate in the style of corresponding   results for Riemannian manifolds by Petersen--Sprouse \cite{petersen-sprouse1998} and Aubry \cite{aubry2007}. In this appendix, we confirm its Lorentzian counterpart anticipated in \cite{petersen-sprouse1998}*{pp. 271--272} almost 30 years ago.

Let $(\mms,g,\meas)$ be a given globally hyperbolic weighted spacetime; in particular, we will only consider \emph{smooth} metric tensors and reference measures in this section.

\begin{definition}[Curvature deficit] Let $K\in\R$ and $p\in [1,\infty)$. The \emph{$L^p$-deficit} of an $\meas$-measurable function $k\colon \mms\to\R$ with respect to $K$ is defined by
\begin{align*}
    \ICD_{k,K,p}(\mms,\meas) &:= \sup\!\Big\lbrace\frac{1}{\meas[W]}\int_W\big\vert(k-K)_-\big\vert^p\d\meas :\\
    &\qquad\qquad W\subset\mms\textnormal{ \textit{nonempty, open, and precompact}}\Big\rbrace.
\end{align*}
\end{definition}

In the previous definition, of course $\smash{\ICD_{k,K,p}(\mms,\meas) =0}$ if and only if $k\geq K$ $\meas$-a.e. (which holds everywhere if $k$ is upper semicontinuous).

The normalization in this definition ensures the value of $\ICD_{k,K,p}(\mms,\meas)$ does not change if we scale $\meas$ by a positive constant; the localization to finite $\meas$-measure subsets will allow us to include  the case when $\meas$ is infinite. 

\begin{theorem}[Sharp timelike Aubry diameter estimate]\label{Th:AUBRY} Let us fix $K>0$, $N\in[2,\infty)$,  $p\in (N/2,\infty)$, and $\theta>0$. Then the constant $\smash{C_{K,N,p}>1}$ from  \cref{Th:DiamCDdensities} has the following property. Let $(\mms,g,\meas)$ be a globally hyperbolic weighted space\-time $(\mms,g,\meas)$ satisfying $\smash{\Ric_{g,\meas,N} \geq k}$ in all timelike directions, where the function  $k\colon\mms\to\R$ is continuous, with
\begin{align*}
    \ICD_{k,K,p}(\mms,\meas) \leq \frac{1}{C_{K,N,p}^{1+\theta}}.
\end{align*}
Then we have
\begin{align*}
    \diam\,(\mms,g) \leq\pi\sqrt{\frac{N-1}{K}}\,\big[1+C_{K,N,p}\,\ICD_{k,K,p}(\mms,\meas)^{1/5(1+\theta)}\big].
\end{align*}
\end{theorem}

\begin{proof} Let $\varepsilon > 0$ satisfy
\begin{align}\label{Eq:epschoice}
    \ICD_{k,K,p}(\mms,\meas)\leq \Big[\frac{\varepsilon}{\pi\,C_{K,N,p}}\sqrt{\frac{K}{N-1}}\Big]^{5(1+\theta)} \leq \frac{1}{C_{K,N,p}^{1+\theta}}.
\end{align}
By the arbitrariness of $\varepsilon$, the claim will be established if we show
\begin{align}\label{Eq:DEST}
    \diam\,(\mms,g) \leq \pi\sqrt{\frac{N-1}{K}}+\varepsilon.
\end{align}

Assume \eqref{Eq:DEST} does not hold. As in the proof of \cref{Th:TimelikeBonnetMyers}, there exist $\eta >0$, $o\in\mms$, and an open Riemannian ball $U\subset\mms$ such that
\begin{align}\label{Eq:etabd}
    \inf\{l_o(y) : y\in U\} \geq \pi\sqrt{\frac{N-1}{K}}+\varepsilon+\eta.
\end{align}
Consider the measure
\begin{align*}
    \neas := \meas[G]^{-1}\,\meas\mres G,
\end{align*}
where the set 
\begin{align*}
    G:=\scrG^+(o,U)
\end{align*}
is from \eqref{Eq:GoU}. Since $G$ is a competitor in the definition of $\smash{\ICD_{k,K,p}(\mms,\meas)}$, where precompactness is implied by global hyperbolicity of $(\mms,g)$,  \eqref{Eq:epschoice} yields
\begin{align}\label{Eq:fptheta}
    \int_G \big\vert(k-K)_-\big\vert^p\d\neas \leq \Big[\frac{\varepsilon}{\pi\,C_{K,N,p}}\sqrt{\frac{K}{N-1}}\Big]^{5(1+\theta)}. 
\end{align}

\cref{Cor:Disintegration} together with the hypothesized lower boundedness of $\smash{\Ric_{g,\meas,N}}$ by $k$ in all timelike directions ensures there exists a smooth $\CD(k,N)$ disintegration $(Q,\q,\neas_\bullet)$ of $\neas$. By construction of $\smash{\scrG^+(o,U)}$ and the disintegration,  $\mms_\alpha$ intersects $U$ for $\q$-a.e.~$\alpha\in Q$. On the other hand, let $R$ be the set of all $\alpha\in Q$ with
\begin{align*}
    \int_{\mms_\alpha}\big\vert(k-K)_-\big\vert^p\d\neas_\alpha \leq \Big[\frac{\varepsilon}{\pi\,C_{K,N,p}}\sqrt{\frac{K}{N-1}}\Big]^5.
\end{align*}
Thanks to the second inequality from \eqref{Eq:epschoice} and our choice of $\smash{C_{K,N,p}}$, the right-hand side corresponds to the number $\delta_{K,N,p,\varepsilon}$ from \cref{Re:Explicit}, which lies in $(0,1)$ by the second inequality from \eqref{Eq:epschoice} since $\smash{C_{K,N,p}>1}$. By Markov's inequality, the disintegration \cref{Cor:Disintegration}, and \eqref{Eq:fptheta},
\begin{align*}
    \q[Q\setminus R] &\leq \frac{1}{\delta_{K,N,p,\varepsilon}}\int_Q\int_{\mms_\alpha} \big\vert (k-K)_-\big\vert^p\d\neas_\alpha\d\q(\alpha)\\
    &= \frac{1}{\delta_{K,N,p,\varepsilon}}\int_G\big\vert(k-K)_-\big\vert^p\d\neas\\
    &\leq \delta_{K,N,p,\varepsilon}^{\theta}.
\end{align*}
Thus, by \cref{Cor:deltacor}, the diameter of $\q$-a.e.~ray passing through the set $R$ (of $\q$-measure at least $\smash{1-\delta_{K,N,p,\varepsilon}^{\theta}}$) is bounded from above by $\smash{\pi\sqrt{(N-1)/K}}+\varepsilon$. On the other hand, as $\q$-a.e.~ray intersects $U$ by construction, we see the diameter of $\q$-a.e.~ray is bounded from below by $\smash{\pi\sqrt{(N-1)/K}+\varepsilon+\eta}$ by \eqref{Eq:etabd}. The intersection of the two sets in question has $\q$-measure bounded from below by $\smash{1-\delta_{K,N,p,\varepsilon}^{\theta}}>0$; in particular, it is non\-empty. This is the desired contradiction.
\end{proof}

\begin{corollary}[Timelike Petersen--Sprouse diameter estimate]\label{Cor:PetersenSprouse} Let $K>0$, $N\in [2,\infty)$, and $p\in (N/2,\infty)$. Then for every $\varepsilon > 0$, there exists $\smash{\xi_{K,N,p,\varepsilon}\in (0,\delta_{K,N,p,\varepsilon})}$, where $\smash{\delta_{K,N,p,\varepsilon}}$ is from \cref{Re:Explicit}, with the following property. Let $(\mms,g,\meas)$ be a globally hyperbolic weighted spacetime and let $k\colon \mms\to\R$ be a continuous function with $\smash{\Ric_{g,\meas,N}\geq k}$ in all timelike directions such that
\begin{align*}
    \ICD_{k,K,p}(\mms,\meas) \leq \xi_{K,N,p,\varepsilon}.
\end{align*}
Then we have 
\begin{align*}
    \diam\,(\mms,g)\leq \pi\sqrt{\frac{N-1}{K}}+\varepsilon.
\end{align*}
\end{corollary}

\section{Details about \cref{Th:Dalembert powers,Th:DAlembert}}\label{App:One}

\subsection{Good geodesics} One feature that is logically implied by the TMCP as shown by Braun \cite{braun2023-good}*{Thm.~4.11}, later adapted by Beran et al.~\cite{beran-braun-calisti-gigli-mccann-ohanyan-rott-samann+-}*{Thm.~5.12}, is the existence of ``good geodesics'', viz.~geodesics of probability measures with uniformly bounded densities. This fact will be needed in the proof of \cref{Pr:DiffFormula} below. Instead of taking the  abstract results from \cite{braun2023-good,beran-braun-calisti-gigli-mccann-ohanyan-rott-samann+-} as a blackbox, we will instead give a more direct proof based on the good approximations from \cref{Def:Goodapprox}  and the ``a priori estimate'' from \cref{Le:Crude}.

\begin{proposition}[A priori estimates for limit optimal dynamical plans] Assume the globally hyperbolic weighted Lipschitz spacetime $(\mms,g,\meas)$ satisfies $\smash{\Ric_{g,\meas,N}\geq K}$ timelike distributionally, where $K\in\R$ and $N\in[\dim\mms,\infty)$. Let $(g_n)_{n\in\N}$ and $(\meas_n)_{n\in\N}$ be good approximations of $g$ and $\meas$, respectively. Moreover, let $o\in\mms$ and let $\mu_1\in\Prob(\mms)$ be compactly supported and $\meas$-absolutely continuous with $\rmd\mu_1/\rmd\meas\in L^\infty(\mms,\meas)$ and $\smash{\{o\}\times\supp\mu_1\subset I_g}$. Given $q\in (0,1)$, let $\smash{(\bdpi^n)_{n\in\N}}$ form a sequence of $\smash{\ell_{g_n,q}}$-optimal dynamical couplings $\smash{\bdpi^n}$ of $\mu_0:=\delta_o$ and $\mu_1$. Then there exist constants $\vartheta\in\R_+$ and $\rho\in\R_+$ such that every narrow limit point $\bdpi$ of $\smash{(\bdpi^n)_{n\in\N}}$ has the property that for every $t\in(0,1]$, $(\eval_t)_\push\bdpi$ is $\meas$-absolutely continuous and
\begin{align*}
    \Big\Vert\frac{\rmd(\eval_t)_\push\bdpi}{\rmd\meas}\Big\Vert_{L^\infty(\mms,\meas)} \leq \frac{1}{t^N}\,\rme^{\vartheta\sqrt{\rho N/(N-1)}}\,\Big\Vert\frac{\rmd\mu_1}{\rmd\meas}\Big\Vert_{L^\infty(\mms,\meas)}.
\end{align*}
\end{proposition}

\begin{proof} Let $n_0\in\N$ be the maximum of the integers provided by \cref{Le:Bounds,Le:Crude}, and let $n\in\N$ such that $n\geq n_0$. By \cref{Le:Bounds},  \cref{Th:UniqSmooth}, and mutual absolute continuity of $\meas_n$ and $\meas$,  $\smash{\bdpi^n}$ necessarily coincides with the unique $\smash{\ell_{g_n,q}}$-optimal dynamical coupling of $\mu_0$ and $\mu_1$; in particular, by \cref{Pr:VaryingStab}, $\bdpi$ is an $\smash{\ell_{g,q}}$-optimal dynamical coupling of $\mu_0$ and $\mu_1$. On the other hand, by \cref{Le:Bounds} there is a constant $\rho\in\R_+$ such that for every $n\in\N$ with $n\geq n_0$, every timelike affinely parametrized $g_n$-maximizing geodesic $\gamma\colon[0,1]\to\mms$ that starts in $o$ and ends in $\supp\mu_1$, and every $t\in[0,1]$,
\begin{align*}
    \Ric_{g_n,\meas_n,N}(\dot\gamma_t,\dot\gamma_t) \geq -\rho\,g_n(\dot\gamma_t,\dot\gamma_t).
\end{align*}

From this, we first derive some a priori estimates. We abbreviate
\begin{align*}
\vartheta &:= \sup\{l_{g_n,o}(y) : n\in\N\textnormal{ with }n\geq n_0,\,y\in\supp\mu_1\}.
\end{align*}
By \cref{Th:EssPwConcTMCP} and monotonicity properties of the involved distortion coefficients, $\smash{\bdpi^n}$-a.e.~$\gamma$ satisfies the following inequality for every $t\in[0,1]$:
\begin{align}\label{Eq:Bdspi_n}
\begin{split}
    \frac{\rmd (\eval_t)_\push\bdpi^n}{\rmd\meas_n}(\gamma_t)^{-1/N} &\geq \tau_{-\rho,N}^{(t)}\circ l_{g_n,o}(\gamma_1)\,\Big[\frac{\rmd\mu_1}{\rmd\meas_n}\Big](\gamma_1)^{-1/N}\\
    &\geq \tau_{-\rho,N}^{(t)}(\vartheta)\,\Big\Vert\frac{\rmd\mu_1}{\rmd\meas_n}\Big\Vert_{L^\infty(\mms,\meas_n)}^{-1/N}\\
    &\geq a^{-1/N}\,\tau_{-\rho,N}^{(t)}(\vartheta)\,\Big\Vert\frac{\rmd\mu_1}{\rmd\meas}\Big\Vert_{L^\infty(\mms,\meas)}^{-1/N},
    \end{split}
\end{align}
where the constant
\begin{align*}
    a := \sup\!\Big\lbrace\frac{\rmd\meas}{\rmd\meas_n}(x) : n\in\N \textnormal{ with }n\geq n_0,\,x\in C\Big\rbrace
\end{align*}
is clearly positive and finite by the setup of the good approximation $(\meas_n)_{n\in\N}$ of $\meas$. By  basic estimates for the involved distortion coefficients, cf.~e.g.~Cavalletti--Mondino \cite{cavalletti-mondino2017-optimal}*{Rem.~2.3}, we know every $\theta \in\R_+$ satisfies
\begin{align*}
    \tau_{-\rho,N}^{(t)}(\theta) \geq t\,\rme^{-(1-t)\theta\sqrt{\rho/N(N-1)}}.
\end{align*}
Combining this with \eqref{Eq:Bdspi_n}, we infer
\begin{align}\label{Eq:Combining}
\begin{split}
    \Big\Vert \frac{\rmd(\eval_t)_\push\bdpi^n}{\rmd\meas}\Big\Vert_{L^\infty(\mms,\meas)} &\leq b\,\Big\Vert \frac{\rmd(\eval_t)_\push\bdpi^n}{\rmd\meas_n}\Big\Vert_{L^\infty(\mms,\meas_n)}\\
    &\leq \frac{1}{t^N}\,\rme^{\vartheta\sqrt{\rho N/(N-1)}}\,\Big\Vert\frac{\rmd\mu_1}{\rmd\meas}\Big\Vert_{L^\infty(\mms,\meas)},
    \end{split}
\end{align}
where the constant
\begin{align*}
    b := \sup\!\Big\lbrace\frac{\rmd\meas_n}{\rmd\meas}(x) : n\in\N \textnormal{ with }n\geq n_0,\,x\in C\Big\rbrace
\end{align*}
is clearly positive and finite by construction of $(\meas_n)_{n\in\N}$ again and \cref{Le:UnifCvg}.

We now turn to the conclusion. Given $t\in (0,1]$, we set
\begin{align*}
    \xi := \frac{1}{t^N}\,\rme^{\vartheta\sqrt{\rho N/(N-1)}}\,\Big\Vert\frac{\rmd\mu_1}{\rmd\meas}\Big\Vert_{L^\infty(\mms,\meas)}
\end{align*}
and define the ``mass excess functional'' $\smash{\scrF_\xi\colon \Prob(C) \to \R_+}$ by
\begin{align*}
    \scrF_\xi(\mu) := \int_C \Big\vert\Big[\frac{\rmd\mu^\ac}{\rmd\meas}-\xi\Big]_+\Big\vert\d\meas + \mu^\sing[C],
\end{align*}
where $\smash{\mu^\ac}$ and $\smash{\mu^\sing}$ denote the absolutely continuous part and the singular part of $\mu$ with respect to $\meas$, respectively. By \eqref{Eq:Combining} and mutual absolute continuity of the involved reference measures, every $n\in\N$ with $n\geq n_0$ obeys
\begin{align*}
    \scrF_\xi((\eval_t)_\push\bdpi^n) =0.
\end{align*}
Since $\smash{\scrF_\xi}$ is narrowly lower semicontinuous, cf.~e.g.~Rajala \cite{rajala2012}*{Lem.~3.6}, we get
\begin{align*}
    \scrF_\xi((\eval_t)_\push\bdpi)=0.
\end{align*}
The arbitrariness of $t$ establishes the second statement.
\end{proof}

By taking the narrow limit from this proposition to be the specific optimal dynamical coupling constructed in the proof of \cref{Th:ToTMCP}, the following holds.

\begin{theorem}[Existence of good geodesics]\label{Th:GoodGeos} Assume that the globally hyperbolic weighted Lipschitz spacetime $(\mms,g,\meas)$ satisfies $\Ric_{g,\meas,N}\geq K$ timelike distributionally, where $K\in\R$ and $N\in[\dim\mms,\infty)$. Let $o\in\mms$. Furthermore, let $\mu_1\in\Prob(\mms)$ be compactly supported and $\meas$-absolutely continuous such that $\rmd\mu_1/\rmd\meas$ is $\meas$-essentially bounded and $\smash{\{o\}\times\supp\mu_1\subset I}$. Then for every $q\in(0,1)$, there exist
\begin{itemize}
    \item an $\smash{\ell_q}$-optimal dynamical coupling $\bdpi$ from $\mu_0:=\delta_o$ to $\mu_1$ and
    \item constants $\vartheta\in\R_+$ and $\rho\in\R_+$
\end{itemize}
satisfying the following properties for every $t\in (0,1]$.
\begin{enumerate}[label=\textnormal{(\roman*)}]
    \item We have
    \begin{align*}
        \scrS_N((\eval_t)_\push\bdpi\mid\meas) \leq -\int_{\mms} \tau_{K,N}^{(t)}\circ l_o\,\Big[\frac{\rmd\mu_1}{\rmd\meas}\Big]^{1-1/N}\d\meas.
    \end{align*}
    \item The measure $(\eval_t)_\push\bdpi$ is $\meas$-absolutely continuous with
    \begin{align*}
        \Big\Vert \frac{\rmd(\eval_t)_\push\bdpi}{\rmd\meas}\Big\Vert_{L^\infty(\mms,\meas)} \leq \frac{1}{t^N}\,\rme^{\vartheta\sqrt{\rho N/(N-1)}}\,\Big\Vert\frac{\rmd\mu_1}{\rmd\meas}\Big\Vert_{L^\infty(\mms,\meas)}.
    \end{align*}
\end{enumerate}
\end{theorem}

Every optimal dynamical coupling that obeys the conclusions of the previous theorem will be informally called \emph{good geodesic}.

\subsection{A Brenier--McCann theorem} A key point is an analog of Beran et al.~\cite{beran-braun-calisti-gigli-mccann-ohanyan-rott-samann+-}*{Thm.~5.19}, specified to the simpler case of powers of Lorentz distance functions. It mirrors \cref{Th:UniqSmooth}, but does not rely on the exponential map, which does not exist in general on Lipschitz spacetimes. 

\begin{proposition}[Brenier--McCann theorem]\label{Th:Brenier} We fix an exponent $q\in (0,1)$. Let $o\in\mms$ and let $\mu_1\in\Prob(\mms)$ be $\meas$-absolutely continuous with  compact support in $\smash{I^+(o)}$. Let $\bdpi$ be an $\smash{\ell_q}$-optimal dynamical coupling from $\smash{\mu_0:=\delta_o}$ to  $\mu_1$. Then $\bdpi$-a.e.~$\gamma$ obeys
\begin{align*}
    \Big\vert\nabla \frac{l_o^q}{q}\Big\vert(\gamma_1)= l(\gamma_0,\gamma_1)^{q-1}.
\end{align*}
\end{proposition}

The proof is a direct application of \cref{Le:EquiLip} (the claim makes sense by the hypothesized absolute continuity), the chain rule, and \cref{Cor:Unit slope}. 

\subsection{Proofs of \cref{Th:Dalembert powers,Th:DAlembert}}

\begin{proposition}[Horizontal vs.~vertical differentiation]\label{Pr:DiffFormula} Let $o\in\mms$. Furthermore, let $\mu_1\in\Prob(\mms)$ be compactly supported in $I^+(o)$ and $\meas$-absolutely continuous with $\rmd\mu_1/\rmd\meas$ $\meas$-essentially bounded. Given $q\in (0,1)$, let $\bdpi$ be an $\smash{\ell_q}$-optimal dynamical coupling that defines a good geodesic from $\mu_0$ to $\mu_1$. 
Finally, we let $f\colon \mms\to\R$ be a Lipschitz continuous function that is supported in $I^+(o)$. Then, denoting by $q'<0$ the conjugate exponent of $q$,
\begin{align*}
    \lim_{t\to 1^-}\int\frac{f(\gamma_1)-f(\gamma_t)}{1-t}\d\bdpi(\gamma)= \int_\mms \rmd f(\nabla l_o)\,\vert\nabla l_o\vert^{q'-2}\d(\eval_1)_\push\bdpi.
\end{align*}
\end{proposition}

In a nutshell, this result mimics the fact from the smooth framework that the integrand on the right-hand side can be computed in two ways: either by considering ``vertical'' perturbations (by slightly perturbing the function $l_o$ in question) or by taking the ``horizontal'' derivative of $f$ in the direction of the transport rays set up by $\bdpi$. We refer to the review of Braun \cite{braun2025} for more consequences of this simple realization. The following somewhat lengthy proof, following the lines of  Beran et al.~\cite{beran-braun-calisti-gigli-mccann-ohanyan-rott-samann+-}*{Thms.~4.15, 4.16}, can be skipped at first reading.

\begin{proof}[Proof of \cref{Pr:DiffFormula}] We first claim
\begin{align}\label{Eq:firstclaim}
\begin{split}
&\lim_{t\to 1^-}\int \frac{l_o(\gamma_1)^q-l_o(\gamma_t)^q}{q(1-t)}\d\bdpi(\gamma)\\
&\qquad\qquad = \int_\mms \frac{1}{q'}\Big\vert \nabla\frac{l_o^q}{q}\Big\vert^{q'}\d(\eval_1)_\push\bdpi + \int \frac{l(\gamma_0,\gamma_1)^q}{q}\d\bdpi(\gamma).
\end{split}
\end{align}
First recall $\bdpi$-a.e.~$\gamma$ satisfies $\gamma_0=o$. Hence, given $t\in[0,1]$,
\begin{align*}
    l_o(\gamma_1)^q - l_o(\gamma_t)^q = l(\gamma_0,\gamma_1)^q -l(\gamma_0,\gamma_t)^q = \big[1-t^q]\,l(\gamma_0,\gamma_1)^q.
\end{align*}
Therefore, Lebesgue's dominated convergence theorem, conjugacy of $q$ and $q'$, and \cref{Th:Brenier} yield the desired identity
\begin{align*}
    \lim_{t\to 1^-} \int \frac{l_o(\gamma_1)^q-l_o(\gamma_t)^q}{q(1-t)}\d\bdpi(\gamma) &= \int l(\gamma_0,\gamma_1)^q\d\bdpi(\gamma)\\
    & = \int_\mms \frac{1}{q'}\Big\vert\nabla\frac{l_o^q}{q}\Big\vert^{q'}\d(\eval_1)_\push\bdpi + \int \frac{l(\gamma_0,\gamma_1)^q}{q}\d\bdpi(\gamma).
\end{align*}

Let $U\subset\mms$ constitute a fixed precompact open neighborhood of $\supp f$ which is compactly contained in $\smash{I^+(o)}$. We claim the existence of $\vartheta>0$ and $\varepsilon_0>0$ such that for every  $\varepsilon \in\R\setminus\{0\}$ such that $\vert \varepsilon\vert \leq\varepsilon_0$, we have
\begin{align}\label{Eq:thet}
    g\Big[\nabla\frac{l_o^q}{q}+\varepsilon\,\nabla f,\nabla\frac{l_o^q}{q}+\varepsilon\,\nabla f\Big] \geq \vartheta\quad\meas\mres U\textnormal{-a.e.}
\end{align}
To this aim, we will abbreviate
\begin{align*}
    \vartheta_1 &:= \inf\{l_o(y) : y \in \cl\,U\},\\
    \vartheta_2 &:= \sup\{l_o(y) : y \in \cl\,U\},
\end{align*}
which are positive and finite constants by our choice of $U$ and continuity of $\smash{l_+}$, cf. \cref{Th:ConsequGH}. Recall from \cref{Le:EquiLip} that $l_o$ is locally Lipschitz continuous on $I^+(o)$. Thanks to the chain rule and \cref{Cor:Unit slope}, at every  differentiability point of $l_o$ in $\smash{I^+(o)\setminus U}$, we easily obtain the relations
\begin{align*}
    g\Big[\nabla\frac{l_o^q}{q}+\varepsilon\,\nabla f,\nabla\frac{l_o^q}{q}+\varepsilon\,\nabla f\Big] &= g\Big[\nabla\frac{l_o^q}{q},\nabla\frac{l_o^q}{q}\Big] = l_o^{2q-2}\,g(\nabla l_o,\nabla l_o)\geq \vartheta_1.
\end{align*}
irrespective of the value of $\varepsilon$. On the other hand, again by \cref{Le:EquiLip}, the chain rule,  \cref{Cor:Unit slope}, the hypothesized Lipschitz continuity of $f$, and uniform boundedness of the coefficients of $g$ in the compact set $\cl\,U$, there exist constants $c_1,c_2>0$ such that at every differentiability point of $l_o$ in $U$,
\begin{align*}
    &g\Big[\nabla\frac{l_o^q}{q}+\varepsilon\,\nabla f,\nabla\frac{l_o^q}{q}+\varepsilon\,\nabla f\Big]\\
    &\qquad\qquad = l_o^{2q-2}\,g(\nabla l_o,\nabla l_o) + 2\varepsilon\,l_o^{q-1} g(\nabla l_o,\nabla f) +\varepsilon^2\,g(\nabla f,\nabla f)\\
    &\qquad\qquad \geq \vartheta_2^{2q-2} - 2\varepsilon\,c_1 - c_2\,\varepsilon^2.
\end{align*}
The right-hand side is uniformly positive when $\varepsilon$ is sufficiently small, as claimed.

Next, we claim that given $\varepsilon$ as above,
\begin{align}\label{Eq:kiminf}
\begin{split}
    &\liminf_{t\to 1^-}\int \frac{l_o^q(\gamma_1)+q\varepsilon\,f(\gamma_1) - l_o^q(\gamma_t)-q\varepsilon\,f(\gamma_t)}{q(1-t)}\d\bdpi(\gamma)\\
    &\qquad\qquad \geq \int_\mms \frac{1}{q'}\Big\vert\nabla \frac{l_o^q}{q}+\varepsilon\, \nabla f\Big\vert^{q'}\d(\eval_1)_\push\bdpi + \int \frac{l(\gamma_0,\gamma_1)^q}{q}\d\bdpi(\gamma).
    \end{split}
\end{align}
By \cref{Le:UnifPrec} and Arzelà--Ascoli's theorem, the support of $\bdpi$ is equicontinuous. In turn, there is $\delta\in (0,1)$ such that $\bdpi$-a.e.~$\gamma$ obeys $\gamma_t\in U$ for every $t\in [1-\delta,1]$. For such a  $t$, the fundamental theorem of calculus (applicable by \cref{Le:EquiLip} and \cref{Th:GoodGeos}) and the reverse Cauchy--Schwarz and Young inequalities imply
\begin{align}\label{Eq:divide}
\begin{split}
    &\int \frac{l_o^q(\gamma_1)+q\varepsilon\,f(\gamma_1) - l_o^q(\gamma_t)-q\varepsilon\,f(\gamma_t)}{q}\d\bdpi(\gamma)\\
    &\qquad\qquad \geq\iint_{[t,1]}  \rmd\Big[\frac{l_o^q}{q}+\varepsilon\,f\Big](\dot\gamma_s)\d s\d\bdpi(\gamma)\\
    &\qquad\qquad \geq \iint_{[t,1]} \frac{1}{q'}\Big\vert \nabla \frac{l_o^q}{q}+\varepsilon\,\nabla f\Big\vert^{q'}\d(\eval_s)_\push\bdpi\d s + \iint_{[t,1]} \frac{\big\vert\dot\gamma_s\big\vert^q}{q}\d \bdpi(\gamma)\\
    &\qquad\qquad = \iint_{[t,1]} \frac{1}{q'}\Big\vert \nabla\frac{l_o^q}{q}+\varepsilon\,\nabla f\Big\vert^{q'}\,\frac{\rmd(\eval_s)_\push\bdpi}{\rmd\meas}\d\meas\d s\\
    &\qquad\qquad\qquad\qquad+ (1-t)\int \frac{l(\gamma_0,\gamma_1)^q}{q}\d \bdpi(\gamma).
    \end{split}
\end{align}
It suffices to prove the identity
\begin{align}\label{Eq:IntClaim}
\begin{split}
&\lim_{t\to 1^-} \frac{1}{1-t}\iint_{[t,1]} \Big\vert \nabla\frac{l_o^q}{q}+\varepsilon\,\nabla f\Big\vert^{q'}\,\frac{\rmd(\eval_s)_\push\bdpi}{\rmd\meas}\d\meas\d s\\
&\qquad\qquad = \int_\mms \Big\vert \nabla \frac{l_o^q}{q}+\varepsilon\,\nabla f\Big\vert^{q'}\d(\eval_1)_\push\bdpi,
\end{split}
\end{align}
since then \eqref{Eq:kiminf} follows by dividing \eqref{Eq:divide} by $1-t$ and sending $t\to 1^-$. On the one hand, the density $\rmd(\eval_s)_\push\bdpi/\rmd\meas$ is $\meas$-essentially bounded uniformly in $s\in[1-\delta,1]$ thanks to \cref{Th:GoodGeos}.  On the other hand, by \eqref{Eq:thet}, negativity of $q'$, and our choice of  $\delta$, the integrand $\smash{\vert\nabla l_o^q/q+\varepsilon\,\nabla f\vert^{q'}}$ belongs to $L^1(\mms,\meas)$. The narrow convergence of $(\eval_s)_\push\bdpi\to(\eval_1)_\push\bdpi$ as $s\to 1^-$ easily implies
\begin{align*}
    \lim_{t\to 1^-} \frac{1}{1-t}\int_{[t,1]} \frac{\rmd(\eval_s)_\push\bdpi}{\rmd\meas}\d s = \frac{\rmd\mu_1}{\rmd\meas}
\end{align*}
against functions in $L^1(\mms,\meas)$. This establishes \eqref{Eq:IntClaim}.

Finally, subtracting the identity \eqref{Eq:firstclaim} from the inequality \eqref{Eq:kiminf} yields
\begin{align*}
    \liminf_{t\to 1^-} \varepsilon\int \frac{f(\gamma_1)-g(\gamma_t)}{1-t}\d\bdpi(\gamma)
    \geq \int_\mms \frac{1}{q'}\Big[\Big\vert\nabla\frac{l_o^q}{q}+\varepsilon\,\nabla f\Big\vert^{q'} - \Big\vert\nabla\frac{l_o^q}{q}\Big\vert^{q'}\Big]\d(\eval_1)_\push\bdpi.
\end{align*}
We divide this inequality by $\varepsilon$ and then send $\varepsilon\to 0$, which creates two inequalities depending on the sign of $\varepsilon$. On the one hand, employing $\meas$-essential boundedness of $\rmd(\eval_1)_\push\bdpi/\rmd \meas$,   \cref{Le:EquiLip}, and Lebesgue's dominated convergence theorem, 
\begin{align*}
    &\liminf_{t\to 1^-} \varepsilon\int \frac{f(\gamma_1)-g(\gamma_t)}{1-t}\d\bdpi(\gamma)\\
    &\qquad\qquad \geq \lim_{\varepsilon\to 0^+} \int_\mms \frac{1}{q'\varepsilon}\Big[\Big\vert\nabla\frac{l_o^q}{q}+\varepsilon\,\nabla f\Big\vert^{q'} - \Big\vert\nabla\frac{l_o^q}{q}\Big\vert^{q'}\Big]\d(\eval_1)_\push\bdpi\\
    &\qquad\qquad =\int_\mms \rmd f\Big[\nabla\frac{l_o^q}{q}\Big]\,\Big\vert\nabla\frac{l_o^q}{q}\Big\vert^{q'-2}\d(\eval_1)_\push\bdpi.
\end{align*}
On the other hand, an analogous argument implies
\begin{align*}
        &\limsup_{t\to 1^-} \varepsilon\int \frac{f(\gamma_1)-g(\gamma_t)}{1-t}\d\bdpi(\gamma)\\
    &\qquad\qquad \geq \lim_{\varepsilon\to 0^-} \int_\mms \frac{1}{q'\varepsilon}\Big[\Big\vert\nabla\frac{l_o^q}{q}+\varepsilon\,\nabla f\Big\vert^{q'} - \Big\vert\nabla\frac{l_o^q}{q}\Big\vert^{q'}\Big]\d(\eval_1)_\push\bdpi\\
    &\qquad\qquad =\int_\mms \rmd f\Big[\nabla\frac{l_o^q}{q}\Big]\,\Big\vert\nabla\frac{l_o^q}{q}\Big\vert^{q'-2}\d(\eval_1)_\push\bdpi.
\end{align*}
Combining the previous estimates yields the claim.
\end{proof}

Now we are in a position to establish \cref{Th:Dalembert powers,Th:DAlembert}.

\begin{proof}[Proof of \cref{Th:Dalembert powers}] Given $\alpha >0$, let $\mu_0,\mu_1\in\Prob(\mms)$ be defined by
\begin{align*}
    \mu_0 &:= \delta_o,\\
    \mu_1 &:= \big[c_\alpha\,(\phi+\alpha)\big]^{N/(N-1)}\,\meas\mres\supp\phi,
\end{align*}
where $c_\alpha>0$ is a normalization constant. Evidently, we have $\supp\mu_0\times\supp\mu_1\subset I$. By \cref{Th:GoodGeos}, there exists an $\smash{\ell_q}$-optimal dynamical coupling $\bdpi$ of $\mu_0$ and $\mu_1$ that defines a good geodesic between these marginals; in particular, given $t\in (0,1]$,
\begin{align*}
    \scrS_N(\mu_t\mid\meas) \leq-\int_\mms \tau_{K,N}^{(t)}\circ l_o\,\Big[\frac{\rmd\mu_1}{\rmd\meas}\Big]^{1-1/N}\d\meas.
\end{align*}
A simple algebraic manipulation then yields
\begin{align}\label{Eq:Simplealg}
    \frac{\scrS_N(\mu_1\mid\meas) -\scrS_N(\mu_t\mid\meas)}{1-t} \geq -\int \frac{1-\tau_{K,N}^{(t)}\circ l_o}{1-t}\,\Big[\frac{\rmd\mu_1}{\rmd\meas}\Big]^{1-1/N}\d\meas.
\end{align}

We now send $\smash{t\to 1^-}$ in this identity and start with the right-hand side. As $\phi$ has compact support in $\smash{I_{K,N}^+(o)}$ and $\bdpi$ is concentrated on timelike affinely parametrized maximizing geodesics, by \cref{Th:TimelikeBonnetMyers}  and continuity of $\smash{l_+}$, the argument $l_o$ is uniformly bounded (away from $\smash{\pi_{K/(N-1)}}$) on the support of $\bdpi$. On the other hand, by  construction of $\mu_1$, the function $\smash{(\rmd\mu_1/\rmd\meas)^{1-1/N}}$ is bounded on the support of $\phi$. Smoothness of the involved $\tau$-distortion coefficients on $[0,\pi_{K/(N-1)})$ and Lebesgue's dominated convergence theorem thus imply
\begin{align*}
    -\liminf_{t\to 1^-}\int \frac{1-\tau_{K,N}^{(t)}\circ l_o}{1-t}\,\Big[\frac{\rmd\mu_1}{\rmd\meas}\Big]^{1-1/N}\d\meas &= -\int_\mms \tilde{\tau}_{K,N}\circ l_o\,\Big[\frac{\rmd\mu_1}{\rmd\meas}\Big]^{1-1/N}\d\meas\\
    &= -c_\alpha\int_{\supp\phi} \tilde{\tau}_{K,N}\circ l_o\,(\phi+\alpha)\d\meas.
\end{align*}

Now we address the left-hand side. Define $\smash{s_N\colon\R_+\to\R_-}$ by
\begin{align*}
    s_N(r) := -r^{1-1/N}.
\end{align*}
Using convexity of $s_N$, we obtain
\begin{align*}
    &\limsup_{t\to 1^-}\frac{\scrS_N(\mu_1\mid\meas) -\scrS_N(\mu_t\mid\meas)}{1-t}\\
    &\qquad\qquad = \limsup_{t\to 1^-} \int_\mms \frac{1}{1-t}\,\Big[s_N\circ \frac{\rmd\mu_1}{\rmd\meas}- s_N\circ\frac{\rmd\mu_t}{\rmd\meas}\Big]\d\meas \\
    &\qquad\qquad \leq \limsup_{t\to 1^-}\int_\mms \frac{1}{1-t}\,s_N'\circ\frac{\rmd\mu_1}{\rmd\meas}\,\Big[\frac{\rmd\mu_1}{\rmd\meas}-\frac{\rmd\mu_t}{\rmd\meas}\Big]\d\meas\\
    &\qquad\qquad = \limsup_{t\to 1^-}\int \,\frac{1}{1-t}\,\Big[s_N'\circ\frac{\rmd\mu_1}{\rmd\meas}(\gamma_1)- s_N'\circ\frac{\rmd\mu_1}{\rmd\meas}(\gamma_t)\Big]\d\bdpi(\gamma).
\end{align*}
\cref{Pr:DiffFormula} implies
\begin{align*}
    &\limsup_{t\to 1^-}\int \,\frac{1}{1-t}\,\Big[s_N'\circ\frac{\rmd\mu_1}{\rmd\meas}(\gamma_1)- s_N'\circ\frac{\rmd\mu_1}{\rmd\meas}(\gamma_t)\Big]\d\bdpi(\gamma)\\
    &\qquad\qquad \leq \int_\mms \rmd\Big[s_N'\circ\frac{\rmd\mu_1}{\rmd\meas}\Big]\Big[\nabla \frac{l_o^q}{q}\Big]\,\Big\vert\nabla\frac{l_o^q}{q}\Big\vert^{q'-2}\d(\eval_1)_\push\bdpi\\
    &\qquad\qquad = \int_\mms \rmd\Big[s_N'\circ\frac{\rmd\mu_1}{\rmd\meas}\Big]\Big[\nabla \frac{l_o^q}{q}\Big]\,\Big\vert\nabla\frac{l_o^q}{q}\Big\vert^{q'-2}\,\frac{\rmd\mu_1}{\rmd\meas}\d\meas\\
    &\qquad\qquad = \frac{c_\alpha}{N}\int_\mms  \rmd\phi\Big[\nabla \frac{l_o^q}{q}\Big]\,\Big\vert\nabla\frac{l_o^q}{q}\Big\vert^{q'-2}\d\meas,
\end{align*}
where the last identity follows from a direct computation and the chain rule. Canceling $c_\alpha$ in the inequality resulting from these expansions of \eqref{Eq:Simplealg} yields
\begin{align*}
    -\int_\mms\rmd\phi\Big[\nabla\frac{l_o^q}{q}\Big]\,\Big\vert\nabla\frac{l_o^q}{q}\Big\vert^{q'-2}\d\meas \leq N\int_{\supp\phi}\tilde{\tau}_{K,N}\circ l_o\,(\phi+\alpha)\d\meas.
\end{align*}
Sending $\smash{\alpha\to 0^+}$ together with the  $\meas$-integrability of $\smash{\tilde{\tau}_{K,N}\circ l_o}$ on $\supp\phi$ as discussed above entails the desired statement.
\end{proof}

\begin{proof}[Proof of \cref{Th:DAlembert}] Let $q\in (0,1)$ be the conjugate exponent of $q'$. By \cref{Le:EquiLip} and as $\phi$ has compact support away from $o$, the function $\phi/l_o$ is again Lipschitz continuous, nonnegative, and has compact support. By \cref{Th:Dalembert powers},
\begin{align}\label{Eq:Insert}
    \int_\mms\frac{N\,\tilde{\tau}_{K,N}\circ l_o}{l_o}\,\phi\d\meas & \geq -\int_\mms \rmd\frac{\phi}{l_o}\Big[\nabla\frac{l_o^q}{q}\Big]\,\Big\vert\nabla\frac{l_o^q}{q}\Big\vert^{q'-2}\d\meas.
\end{align}
As identities of (co)tangent vectors, the chain rule yields
\begin{alignat*}{2}
    \rmd\frac{\phi}{l_o} &= \frac{1}{l_o}\,\rmd\phi - \frac{\phi}{l_o^2}\,\rmd l_o& \quad &\meas\mres I^+(o) \textnormal{-a.e.},\\
    \nabla\frac{l_o^q}{q} &= l_o^{q-1}\,\nabla l_o &&\meas\mres I^+(o)\textnormal{-a.e.}
\end{alignat*}
In turn, a direct computation, conjugacy of $q$ and $q'$, and \cref{Cor:Unit slope} yield
\begin{align*}
    \rmd\frac{\phi}{l_o}\Big[\nabla\frac{l_o^q}{q}\Big]\,\Big\vert\nabla\frac{l_o^q}{q}\Big\vert^{q'-2} &= \rmd\phi(\nabla l_o) - \frac{\phi}{l_o}\,\vert\nabla l_o\vert^2\\
    &= \rmd\phi(\nabla l_o)\,\vert\nabla l_o\vert^{q'-2} - \frac{\phi}{l_o}\quad\meas\mres I^+(o)\textnormal{-a.e.}
\end{align*}
Inserting this into \eqref{Eq:Insert} and rearranging terms yields the claim.
\end{proof}

\bibliographystyle{amsrefs}
\bibliography{library}

\end{document}